\documentclass[11pt]{amsart}
\usepackage{amsxtra,amssymb,amsmath,amsthm,amsbsy, amscd,bbm}
\usepackage[all]{xy}
\usepackage{color}
\usepackage[normalem]{ulem}
\usepackage{cancel}
\theoremstyle{plain}

\numberwithin{equation}{section}

\newtheorem{theorem}{Theorem}[section]
\newtheorem{lemma}[theorem]{Lemma}
\newtheorem{definition-lemma}[theorem]{Definition-Lemma}
\newtheorem{proposition}[theorem]{Proposition}
\newtheorem{corollary}[theorem]{Corollary}
\newtheorem{definition}[theorem]{Definition}
\newtheorem{example}[theorem]{Example}
\newtheorem{remark}[theorem]{Remark}

\newcommand{\id}         {{\mathrm {Id}}}

\newcommand{\gra}        {{\mathrm {graph}}}

\newcommand{\pr}         {{\mathrm{pr}}}

\def\G{\mathcal G}
\def\H{\mathcal H}
\def\V{\mathcal V}
\def\R{\mathbb R}
\def\C{C^\infty}
\def\CC {\mathbb C}

\def\id{{\rm id}}

\def\pr{{\rm pr}}

\def\toto{\rightrightarrows}
\def\tensor{\otimes}
\def\u{\underline}

\def\<{\langle}
\def\>{\rangle}

\newcommand{\thistheoremname}{}
\newtheorem{genericthm}[theorem]{\thistheoremname}

\newtheorem*{genericthm*}{\thistheoremname}
\newenvironment{namedthm*}[1]
  {\renewcommand{\thistheoremname}{#1}%
   \begin{genericthm*}}
  {\end{genericthm*}}

\newcommand{\Arrow}     {\rightarrow}

\newcommand{\frakg}     {\mathfrak{g}}

\newcommand{\frakx}     {\mathfrak{X}}

\renewcommand{\G}      {\mathcal{G}}
\newcommand{\grd}      {\mathcal{G}}

\newcommand{\s}        {\mathsf{s}}
\renewcommand{\t}      {\mathsf{t}}
\renewcommand{\u}      {\mathsf{1}}
\newcommand{\m}        {\mathsf{m}}
\newcommand{\ttimes}   {_{\s}\!\times_{\t}}
\newcommand{\inv}        {\mathsf{i}}

\newcommand{\cots}     {\widetilde{\mathsf{s}}}
\newcommand{\cott}     {\widetilde{\mathsf{t}}}
\newcommand{\cotu}     {\widetilde{\mathsf{\u}}}

\newcommand{\M}        {\mathbb{M}}
\newcommand{\bG}       {\mathbb{G}}
\newcommand{\bE}       {\mathbb{E}}
\newcommand{\bA}       {\mathbb{A}}

\newcommand{\bt}       {\mathbbm{t}}
\newcommand{\bT}       {\mathbb{T}}
\newcommand{\E}        {E}
\newcommand{\Rev}      {\mathcal{R}}

\newcommand{\vl}       {{\mathrm{v}}}       

\newcommand{\Lie}        {\mathcal L}

\newcommand{\VB}        {\mathcal {VB}}

\newcommand{\dd}     {d}

\newcommand{\T}      {{\mathcal{T}}}
\renewcommand{\tensor}[2]  {{\wedge^{#1}\, T^*\G \otimes \wedge^{#2} T\G}}
\newcommand{\atensor}[2]   {{\wedge^{#1}\, T^*M \otimes \wedge^{#2} A}}

\newcommand{\U}           {U\vphantom{\xi}}
\newcommand{\X}           {X\vphantom{\xi}}

\newcommand{\calb}       {\mathcal{B}}

\newcommand{\N}          {\mathcal{N}}

\begin{document}

\title[]
{Lie theory of multiplicative tensors}

\author[]{Henrique Bursztyn}
\address{Instituto de Matem\'atica Pura e Aplicada,
Estrada Dona Castorina 110, Rio de Janeiro, 22460-320, Brasil }
\email{henrique@impa.br}

\author[]{Thiago Drummond}

\address{Departamento de Matem\'atica, Instituto de Matem\'atica,
Universidade Federal do Rio de Janeiro,
Caixa Postal 68530, Rio de Janeiro, RJ, 21941-909, Brasil.
}
\email{drummond@im.ufrj.br}

\date{}

\begin{abstract}
We study tensors on Lie groupoids suitably compatible with the groupoid structure, called {\em multiplicative}.
Our main result gives a complete description of these objects only in terms of infinitesimal data. Special cases include the
infinitesimal counterparts of multiplicative forms, multivector fields and holomorphic
structures, obtained through a unifying and conceptual method. We also give a full treatment of
multiplicative vector-valued forms, particularly Nijenhuis operators and related structures.

\end{abstract}

\maketitle


\vspace{-0.4cm}

\section{Introduction}

Lie groupoids permeate several areas of mathematics, including foliations, group actions and Poisson geometry.
In these various contexts, Lie groupoids often come equipped with additional geometric structures, suitably compatible with the
groupoid multiplication; such structures, referred to as {\em multiplicative}, are the main object of interest in this paper.


The study of multiplicative geometric structures on Lie groupoids has by now a long and rich history (see \cite{KosSurv}
for a recent survey). A basic example of multiplicative structure arises in the definition of complex Lie groups, regarded as  real Lie groups endowed with a ``compatible'' complex structure; here, compatibility means that the group multiplication map   is holomorphic. Another class of multiplicative structures on Lie groups arose in the early 80s with the emergence of Poisson-Lie groups, introduced by Drinfel'd \cite{Dr} (later extended to Poisson groupoids by Weinstein \cite{We88}). Around the same time, the first examples of multiplicative differential forms on Lie groupoids came to notice
with the advent of symplectic groupoids \cite{Kar,We87} (and in their connections with the theory of hamiltonian actions
\cite{MiWe} and equivariant cohomology; see \cite{BC,bcwz,xu}). Multiplicative structures now abound in the literature,
where one finds multi-vector fields \cite{ILX,LW,Mac-Xu,Mac-Xu1}, differential forms \cite{AC,bc,BCO,CSS}, contact and Jacobi
structures \cite{CZ,IgMa,KS-B}, holomorphic structures \cite{LMX2,LMX}, as well as distributions, foliations \cite{CSS,DJO,Haw,JO},  among others (e.g. \cite{ETV,lbsev,ortiz}).


Any Lie groupoid has an underlying Lie algebroid, which linearizes it at the units.
As in classical Lie theory, a central issue when considering multiplicative geometric structures is identifying
their infinitesimal versions, i.e., finding their description solely in terms of Lie-algebroid data.
This problem has been studied in numerous settings, through different approaches,
leading to various ``infinitesimal-global'' correspondence results. Examples include the correspondences between
symplectic groupoids and Poisson structures \cite{catfel,CF,Mac-Xu2}, Poisson groupoids and Lie bialgebroids \cite{LW,Mac-Xu,Mac-Xu2},
contact groupoids and Jacobi structures \cite{CS,CZ}, presymplectic groupoids and Dirac structures \cite{bcwz},
complex Lie groupoids and holomorphic Lie algebroids \cite{LMX}, to mention a few
(see also \cite{AC,CX,ILX,StX}). All these results rely on defining a ``Lie functor'', taking global
to infinitesimal objects, and on an ``integration'' step, which reconstructs multiplicative structures from infinitesimal data.


In spite of these various results in specific settings, the theory of multiplicative geometric structures
still lacks a complete treatment. The very notion of ``multiplicativity'' seems to be adapted to each case at hand, and
different techniques have been employed to handle seemingly analogous results. In this paper, we introduce the concept of
{\em multiplicative tensor} on Lie groupoids, which agrees with the existing notions of multiplicativity in
known situations, and devise a general method to obtain their complete infinitesimal description.
As a consequence, all aforementioned ``infinitesimal-global'' correspondences can be naturally derived from our
main result (Theorem~\ref{thm:main}), with a unified proof and conceptual approach, and new applications are obtained.
Although we focus on ordinary tensors, our method adapts to more general contexts (such as the study of 1-cocycles
on VB-groupoids), including tensors with values in representations (up to homotopy), see \cite{DE,Egea}.


\medskip

\noindent{\bf Main results.}
Let $\G \toto M$ be a Lie groupoid with source and target maps $\s$ and $\t$.
We denote its Lie algebroid by $A$, equipped with anchor map $\rho$ and bracket $[\cdot,\cdot]$. Consider a
$(q,p)$-tensor field\footnote{The assumption of skew symmetry lead to some simplifications of the results, but it is by no means essential.} $\tau \in \Gamma(\wedge^p T^*\G \otimes \wedge^q T\G)$.
Let
$$\mathbb{G}:= (\oplus^p T\G) \oplus (\oplus^q T^*\G),
$$
which carries a natural groupoid structure over $\mathbb{M}:=(\oplus^p TM)\oplus(\oplus^q A^*)$ induced from the tangent and cotangent groupoids; see Section~\ref{sec:prelim}. We say that the tensor $\tau$ is {\em multiplicative} if the map
$c_\tau: \mathbb{G} \to \mathbb{R}$,
$$
c_\tau(U_1,\ldots,U_p,\xi_1,\ldots,\xi_q) = \tau(U_1,\ldots,U_p,\xi_1,\ldots,\xi_q),
$$
is a groupoid morphism, where $\mathbb{R}$ is viewed as an abelian group; in other words,
$c_\tau$ is a differentiable {\em 1-cocycle} on $\mathbb{G}$.


To state our main result, consider the action of $\Gamma(A)$ on $\Gamma(\wedge^p T^*M \otimes \wedge^q A)$ by
$$
a\cdot (\beta\otimes\mathfrak{X})=\Lie_{\rho(a)}\beta \otimes \mathfrak{X} + \beta\otimes [a, \mathfrak{X}],
$$
where $[\cdot,\cdot]$ is the Schouten bracket on $\Gamma(\wedge^\bullet A)$.
The following theorem gives a full description of the infinitesimal counterparts of multiplicative tensors:

\begin{theorem}\label{thm:intro}
If $\G \toto M$ is source 1-connected, then there is a natural one-to-one correspondence
between multiplicative $(q,p)$-tensors $\tau$ on $\G$ and triples $(D,l,r)$, where
$l: A \Arrow \wedge^{p-1}T^*M\otimes \wedge^q A$ and $r: T^*M \Arrow \wedge^{p}T^*M\otimes \wedge^{q-1} A$
are vector bundle maps covering the identity, $D: \Gamma(A) \Arrow \Gamma(\wedge^p T^*M \otimes \wedge^q A)$ is an $\R$-linear map
satisfying the Leibniz-type condition
$$
D(fa) = fD(a) + df \wedge l(a) - a \wedge r(df), \;\;\;\; f\in C^\infty(M),\; a\in \Gamma(A),
$$
and the following equations hold: for $a,b \in \Gamma(A)$ and $\alpha,\beta \in \Omega^1(M)$,
\begin{align}
\tag{IM1}  D([a,b])& = a\cdot D(b) - b \cdot D(a),\\
\tag{IM2}  l([a,b])& = a \cdot l(b) - i_{\rho(b)} D(a),\\
\tag{IM3}  r(\Lie_{\rho(a)} \alpha) &= a\cdot r(\alpha) - i_{\rho^*(\alpha)} D(a),\\
\tag{IM4}  i_{\rho(a)}\,l(b) & = - i_{\rho(b)} \,l(a),\\
\tag{IM5}  i_{\rho^*\alpha}\, r(\beta) &= - i_{\rho^*\beta}\, r(\alpha),\\
\tag{IM6}  i_{\rho(a)}\,r(\alpha) & = i_{\rho^*\alpha}\,l(a).
\end{align}
\end{theorem}

We refer to the equations (IM1)--(IM6) as {\em cocycle equations}, or {\em IM equations}
(where IM stands for ``infinitesimally multiplicative'').  A more detailed formulation of the
previous result can be found in Theorem~\ref{thm:main} below.

The definition of the ``Lie functor'', taking multiplicative tensors $\tau$ to triples $(D,l,r)$,
relies on a useful characterization of multiplicativity proven in Theorem~\ref{thm:mult_char};
it says, in particular, that for any $\alpha\in \Omega^1(M)$, $a\in \Gamma(A)$ and $\overrightarrow{a}$ the
corresponding right-invariant vector field on $\G$, the tensors $i_{\overrightarrow{a}} \tau$, $i_{\t^*\alpha}\,\tau$,
and $\Lie_{\overrightarrow{a}}\tau$ lie in the image of the map
$$
\T: \Gamma(\wedge^\bullet T^*M\otimes \wedge^\bullet A)\to \Gamma(\wedge^\bullet T^*\G \otimes \wedge^\bullet T\G),
$$
defined on homogeneous elements by $\T(\alpha\otimes \mathfrak{X}) = \t^*\alpha\otimes \overrightarrow{\mathfrak{X}}$.
The maps $D$, $l$, and $r$ arise from the equations
\begin{equation*}
i_{\overrightarrow{a}} \tau = \T(l(a)),\;\;
i_{\t^*\alpha}\,\tau=  \T(r(\alpha)),\;\;
\Lie_{\overrightarrow{a}} \tau =  \T(D(a)).
\end{equation*}
The IM-equations, combined with the 1-connectedness of the source fibers of $\G$,
permit the reconstruction of $\tau$ out of $(D,l,r)$.
Theorem \ref{thm:intro}, when restricted to tensors of types $(0,p)$ or $(q,0)$, directly recovers
the infinitesimal descriptions of multiplicative differential forms and multivector fields proven in \cite{AC,BC,ILX},
but using other methods. For $(1,1)$-tensors, it encompasses the correspondence of complex Lie groupoids and
holomorphic Lie algebroids of \cite{LMX}. In these special cases, the operator $D$ takes different guises,
codifying  $k$-differentials \cite{ILX}, IM-forms \cite{BC} (or the ``Spencer operators'' of \cite{CSS}),
or flat partial connections defining holomorphic structures.

The content of Theorem~\ref{thm:intro} is discussed in Section~\ref{sec:multiplicative},
and its proof is presented in Section~\ref{sec:proofs}, heavily based on our viewpoint to multiplicative tensors
$\tau$ on $\G$ as multiplicative {\em functions} $c_\tau$ on the  ``big'' Lie groupoid $\mathbb{G}\rightrightarrows \mathbb{M}$.
Multiplicative functions are simple to describe infinitesimally: they correspond to Lie-algebroid 1-cocycles,
i.e., sections of the dual of the Lie algebroid which are closed under the Lie-algebroid differential.
So the proof follows from a detailed analysis of 1-cocycles of the Lie algebroid $\mathbb{A}$ of $\mathbb{G}$.
The key fact that $\mathbb{A}\to \mathbb{M}$ has a natural {\em VB-algebroid} structure over $A\to M$ allows us to
identify a special set of generators of the $C^\infty(\mathbb{M})$-module $\Gamma(\mathbb{A})$, parametrized by
$\Gamma(A)$ and $\Omega^1(M)$. We use these generators to describe 1-cocycles of $\mathbb{A}$ by means of triples $(D,l,r)$
as in Theorem~\ref{thm:intro}, and we resort to classical lifting operations to realize the cocycle
condition as the equations (IM1)--(IM6).

In Section~\ref{sec:vectorvforms}, we specialize our main result to the case of $(1,p)$-tensors, i.e.,
{\em multiplicative vector-valued forms}. As observed in \cite{bd}, the usual Fr\"olicher-Nijenhuis bracket
preserves the multiplicativity condition, so it makes  the space of multiplicative vector-valued forms into a graded Lie algebra.
 One of our key results is the identification of the corresponding graded Lie algebra at the infinitesimal level, in Prop.~\ref{prop:compFN}. In Section~\ref{1_1}, we focus on multiplicative vector-valued 1-forms, i.e., (1,1)-tensors. In this case, we obtain an
explicit infinitesimal description of their Nijenhuis torsions (Cor.~\ref{cor:nij}), which
provides a broader viewpoint to the results in \cite{LMX,StX} concerning
multiplicative (almost) complex structures and Poisson (quasi-)
Nijenhuis structures. We also treat multiplicative projections and product structures, interpreting their infinitesimal versions
in terms of matched pairs.

\smallskip

\noindent{\bf Acknowledgments}: We thank CNPq, Capes and Faperj for financial support.
We are grateful to several institutions for hosting us during various stages of this project,
including IST (Lisbon), Utrecht University, ESI (Vienna) and the Fields Institute.
Special thanks go to Y. Kosmann-Schwarzbach for her comments and interest. We also thank A. Cabrera, M. Crainic, M. del Hoyo, L. Egea, C. Ortiz, L. Vitagliano and M. Zambon for helpful advice; we are especially indebted to Noah Kieserman for his collaboration in the early stages of this project.

\tableofcontents

\section{Preliminaries} \label{sec:prelim}
This section reviews some preliminary material, including tangent and cotangent Lie groupoids; see e.g. \cite{Mac-book, Mac-Xu}. We also discuss a convenient viewpoint to classical tensor fields, regarded as real-valued functions on Whitney sums of vector bundles.


\subsection{Tangent Lie groupoids}

Let $\grd \toto M$ be a Lie groupoid. We use the following notation:
$\s ,\t : \grd \Arrow M$ are the source and target maps, $\u: M
\Arrow \grd$ is the unit map, $\inv: \G \Arrow \G$ is the inversion
map, and $\m: \grd\,_{\s}\!\times_{\t} \grd \Arrow \grd$ is the
multiplication map. We will often identify $M$ with its image under
the unit map.

The \textit{tangent groupoid} of $\grd$ is the Lie groupoid
$T\G\toto TM$ whose structural maps are all obtained by taking the
derivatives of the structural maps of $\grd$; e.g., its
multiplication map is $T\m: T\grd\,_{T\s}\!\times_{T\t}T\grd \Arrow
T\grd$, where we have identified $T(\grd\, _{\s}\!\times_{\t}\grd)
\cong T\grd\,_{T\s}\!\times_{T\t} T\grd$. We shall denote the
multiplication on $T\G$ by $\bullet$.


We denote the Lie algebroid of $\grd$ by $A \Arrow M$, or $A\grd$
if there is any risk of confusion. We identify $A$ with $\ker(T\s)|_M$,
so the Lie bracket on $\Gamma(A)$ is
induced by right-invariant vector fields on $\G$, and the anchor
$\rho: A\to TM$ is given by $T\t|_{A}$. For $a\in \Gamma(A)$,
we denote by $\overrightarrow{a} \in \mathfrak{X}(\G)$ the
corresponding right-invariant vector field, and by
$\overleftarrow{a}$ the left-invariant vector field induced by
$a-\rho(a)\in \Gamma(\ker(T\t)|_{M})$.

Note that each section $a\in \Gamma(A)$ defines a bisection $\calb
a: TM \Arrow T\G$ of $T\G \toto TM$,
\begin{equation}\label{trans_bis}
\mathcal{B}a(X) = T\u(X) + a(x),  
\end{equation}
for $X \in T_xM$, covering the map $TM\to TM$, $X \mapsto X + \rho(a)(x)$. This bisection splits
the exact sequence
\begin{equation}\label{s_seq}
0 \longrightarrow A \hookrightarrow \u^*T\G \stackrel{T\s}
\longrightarrow TM \longrightarrow 0.
\end{equation}
We refer to $\calb a$ as the \textit{translation bisection}
associated with $a$.

\subsection{Cotangent Lie groupoids}

The cotangent bundle of a Lie groupoid $\G \toto M$ also carries a
natural Lie groupoid structure, $T^*\G\toto A^*$, where $A^*$ is
the dual vector bundle to $A$. The source and target maps $\cots,
\widetilde{\t}: T^*\G \Arrow A^*$ are defined by the restriction of
covectors to the subspaces tangent to the $\s$- and $\t$-fibers,
respectively:
\begin{equation}\label{eq:cotst}
\<\cots(\xi_g), a \> =  \<\xi_g,  \overleftarrow{a}(g) \> \text{\;
and \;} \< \cott(\xi_g), b \> = \< \xi_g, \overrightarrow{b}(g) \>,
\end{equation}
for $\xi_g \in T^*_g \G, \, a \in A_{\s(g)}, b \in A_{\t(g)}$.
The unit map $ \cotu:A^* \Arrow T^*\G $ is the vector bundle
morphism covering $\u: M \Arrow \G$, determined by
\begin{equation*}
\< \cotu_{\varphi}, T1(X) + a\> = \<\varphi, a\>,
\end{equation*}
for  $(X, \varphi, a) \in TM \times_M A^* \times_M A$.

The multiplication on $T^*\G$ is defined as follows: for $\xi_1 \in
T_{g_1}^*\G$ and $\xi_2 \in T_{g_2}^*\G$ such that
$\cots(\xi_1)=\cott(\xi_2)$, their product $\xi_1 \bullet \xi_2 \in
T_{g_1g_2}^*\G$ is determined by
\begin{equation}\label{m:cotang}
\< \xi_1 \bullet \xi_2, U_1\bullet U_2 \> = \xi_1(U_1) + \xi_2(U_2),
\end{equation}
for composable $U_1 \in T_{g_1}\G, U_2 \in T_{g_2}\G$.

The source map $\cots$ fits into the following exact sequence of
vector bundles over $M$:
\begin{equation}\label{s_dual_seq}
0 \longrightarrow T^*M \stackrel{(T\t)^*}\hookrightarrow \u^*T^*\G  \stackrel{\cots} \longrightarrow A^* \longrightarrow 0.
\end{equation}
Given a differential 1-form $\alpha \in \Omega^1(M)$ on $M$, the map
$\calb \alpha: A^* \Arrow \u^*T^*\G$ given by
\begin{equation}\label{core_dual}
\< \calb \alpha(\varphi), T1(X) + a \> = \<\alpha(x), \rho(a) + X\> + \<\varphi, a\>,
\end{equation}
for $\varphi \in A_x^*, \, X \in T_xM, \, a\in A_x$, provides a
splitting of the sequence \eqref{s_dual_seq}. It is a bisection of
$T^*\G\toto A^*$ covering
$
\varphi \mapsto \varphi + \rho^*(\alpha(x)), \,\, \varphi \in A_x^*.
$
We call $\calb \alpha$ the \textit{translation bisection}
corresponding to $\alpha$.

\subsection{Whitney sums}

Tangent and cotangent groupoids satisfy the property that their
Whitney sums as vector bundles again carry natural Lie groupoid
structures, defined componentwise\footnote{This property holds, more
generally, for {\em $\VB$-groupoids}, of which tangent and cotangent
groupoids are special cases (see e.g. \cite{BCdH})}. In this paper, we will be interested
in Lie groupoids of the form $\bG^{(p,q)} \toto \mathbb{M}^{(p,q)}$,
where
\begin{equation}\label{big_group}
\bG^{(p,q)} = (\oplus^p T\G) \oplus (\oplus^q T^*\G) \,\,\, \text{
and } \,\,\,  \M^{(p,q)} = (\oplus^p TM) \oplus (\oplus^q A^*),
\end{equation}
for positive integers $p$ and $q$.  When there is no risk of
confusion, we omit the indices $(q,p)$ and write just $\bG \toto
\mathbb{M}$. We denote the source and target maps by $\mathbbm{s},
\mathbbm{t}: \bG \Arrow \mathbb{M}$ and the unit map by
$\mathbbm{1}: \mathbb{M} \Arrow \mathbb{G}$. We keep the notation
$\bullet$ for the multiplication.

We shall denote by $(\underline{\X}, \underline{\varphi})$ and $(\underline{\U},
\underline{\xi})$ the elements
$(X_1, \dots, X_p, \varphi_1, \dots, \varphi_q) \in \mathbb{M}$ and
$(U_1, \dots, U_p, \xi_1, \dots, \xi_q) \in \bG$, respectively.

\begin{remark}\label{rem:fibers}
An important observation is that each source-fiber of $\bG$ is an affine bundle
over a source-fiber of $\G$, see e.g. \cite[Rem.~3.1.1(a)]{BCdH}; it follows that
$\bG$ is source connected, or source 1-connected, if and only if so is $\G$.
\end{remark}

\subsection{Functions on the Whitney sum of vector bundles}\label{app:component}
A key viewpoint pursued in this work is that tensor fields should be regarded as functions on the Whitney
sum of the tangent and cotangent bundles. We now explain this point of view from a general perspective.
Let $\E_1, \dots, \E_p$ be vector bundles over $N$ and consider their Whitney sum $\pi: E_1 \oplus \dots \oplus E_p \to N$.

\begin{definition}\label{comp_linear}
A function $F: \E_1 \oplus \dots \oplus \E_p \Arrow \R$ is said to be {\em componentwise linear}
if $F_y: (\E_1)_y \times \dots \times (\E_p)_y \Arrow \R$ is multi-linear, for each $y \in N$.
\end{definition}

\begin{example}
For $p=1$, componentwise linear functions on a vector bundle $E\Arrow N$ are  fiberwise linear functions,
i.e., those of the form $\ell_{\mu}$, where
$$
\ell_{\mu}(e) := \<\mu(y), e\>, \,\, \forall \,e \in E_y,
$$
for $\mu \in \Gamma(\E^*)$.
\end{example}

Every tensor field $\tau \in \Gamma(\E_1^*\otimes\dots \otimes \E_p^*)$ defines a componentwise linear
function $c_{\tau}: \E_1\oplus\dots \oplus\E_p \Arrow \R$ by pulling back the linear function
$\ell_{\tau}: \E_1\otimes\dots \otimes \E_p \Arrow \R$ by the natural map
$\E_1\oplus \dots \oplus\E_p \Arrow \E_1\otimes\dots \otimes \E_p$. The letter ``$c$'' in
our notation stands for ``componentwise'', and it is used to distinguish $c_{\tau}$ from the linear functions on
$\E_1\oplus \dots\oplus\E_p$ defined by sections of its dual.
In case $\tau= \mu_1\otimes \dots \otimes \mu_p$, for $\mu_i \in \Gamma(\E_i^*)$, $i=1,\dots, p$,
$$
c_{\tau}= \ell_{\mu_1} \circ \pr^1 \dots \ell_{\mu_p} \circ \pr^p,
$$
where $\pr^j: E_1 \oplus \dots \oplus E_p \to E_j$ is the projection on the $j$-th component.

The next result is a direct consequence of the properties of tensor products.
\begin{lemma}\label{lemma:comp_tensor}
The map $\tau  \mapsto c_{\tau}$ defines a bijection between $\Gamma(\E_1^*\otimes\dots \otimes \E_p^*)$
and the space of componentwise linear functions on $\E_1\oplus \dots \oplus \E_p$ satisfying
$$
c_{f \tau} = (f \circ \pi) \, c_\tau, \,\, \, \forall \, f \in C^{\infty}(N),
$$
where $\pi: E_1 \oplus \dots \oplus E_p \to N$ is the natural projection.
\end{lemma}


When $\E_{1} = \dots = \E_{p}=\E$, we say that a function $F \in \C(\oplus_{i=1}^p \E)$ is \textit{skew-symmetric} if
$
F(e_{\sigma(1)}, \dots, e_{\sigma(p)}) = sgn(\sigma) F(e_1, \dots, e_p),
$
for every permutation $\sigma \in S(p)$. Under the correspondence of Lemma \ref{lemma:comp_tensor}, the componentwise
linear functions on $\oplus^p \E$ which are skew-symmetric correspond to $\Gamma(\wedge^{p} \E^*)$.
For $\tau = \mu_1 \wedge \dots \wedge \mu_p$, with $\mu_i \in \Gamma(E^*)$,
\begin{equation}\label{skew_comp}
c_{\tau} =  \sum_{\sigma \in S(p)} sgn(\sigma) (\ell_{\mu_{\sigma(1)}} \circ \pr^{1})\cdots(\ell_{\mu_{\sigma(p)}}. \circ \pr^{p}).
\end{equation}

There is a natural projection $\C(\oplus^p \E) \Arrow \C(\oplus^p \E)$ on the space of skew-symmetric functions, defined by
\begin{equation}\label{anti_projection}
F \mapsto \frac{1}{p!} \sum_{\sigma \in S(p)} sgn(\sigma) \, F\circ \sigma,
\end{equation}
where $\sigma: \oplus^p \E \Arrow \oplus^p \E$ is given by $\sigma(e_1, \dots, e_p) = (e_{\sigma(1)}, \dots, e_{\sigma(p)})$.

\begin{remark}
Functions which are skew-symmetric only on some components can be defined similarly by considering $E_1=\dots = E_{p'}=E$, $p' < p$.
Extending the previous observations to this case is straightforward.
\end{remark}

\section{Multiplicative tensors}\label{sec:multiplicative}

In this section, we introduce our main object of study, multiplicative tensors, and state our main theorem, which
gives their full infinitesimal description. As we follow the idea of regarding tensors as functions on Whitney sums
of the tangent and cotangent bundles, we start by discussing multiplicative functions in general.

\subsection{Multiplicative functions}\label{subsec:functions}

Let $\H \toto N$ be a Lie groupoid with Lie algebroid $A\H$.

\begin{definition}
A smooth function $f \in C^{\infty}(\H)$ is said to be {\em multiplicative} if
\begin{equation}\label{mult:func}
f(h_1 h_2) = f(h_1) + f(h_2), \,\, \forall \, (h_1,h_2) \in \H\ttimes \H.
\end{equation}
\end{definition}

In other words, a multiplicative function $f: \H \Arrow \mathbb{R}$
is a groupoid morphism from $\H \toto N$ to the abelian Lie group
$\R$. As such, it defines a Lie-algebroid morphism $Af:A\H \Arrow \R$ given by
$
\<Af,\, \chi\> = df(\chi),
$
for $\chi \in A_y \H$ and $y \in N$;
equivalently,
\begin{equation}\label{eq:cocycle}
\<Af, \chi(y)\> = \left(\Lie_{\overrightarrow{\chi}}f\right)(y),
\end{equation}
for $\chi \in \Gamma(A\H)$. When we view $Af$ as a section of the dual bundle $A^*\H$, the
condition for $Af$ to be a Lie-algebroid morphism is expressed by
the cocycle equation
$$
d_A (Af) = 0,
$$
where $d_A: \Gamma(\wedge^\bullet A^*\H)\to \Gamma(\wedge^{\bullet +
1}A^*\H)$ is the Lie-algebroid differential.

\begin{example}\label{ex:multf}
The function $f= \t^*\psi - \s^*\psi$ is multiplicative, for any
$\psi \in \C(N)$. The associated cocycle $Af$ is $d_A\psi \in
\Gamma(A^*\H)$, which is exact, i.e., a coboundary.
\end{example}

\begin{example}\label{linear_func}
Let $\pi:\E \Arrow N$ be a vector bundle, viewed as a Lie groupoid. A function
$f \in \C(\E)$ is multiplicative if and only if it is
fiberwise linear. Let $\mu \in \Gamma(\E^*)$ be such
that $f = \ell_{\mu}$. The cocycle $Af \in \Gamma(A^*\E)
= \Gamma(\E^*)$ is given by
$$
\<Af, e\> = 
df\left(\left.\frac{d}{d\epsilon}\right|_{\epsilon=0} \epsilon \,e
\right) = \left.\frac{d}{d\epsilon}\right|_{\epsilon=0} \< \mu,
\epsilon\,e\> = \< \mu, e\>, \,\,\,\;\; \forall \, e \in \E.
$$
So $Af= \mu$ (which agrees with $f$ itself, if seen as a function
$\E \Arrow \R$). The equation $d_A\mu=0$ is trivially satisfied
since $d_A=0$ in this case.
\end{example}

The next result gives a useful formula relating the cocycle $Af$ and
the function $f$.

\begin{lemma}
For a multiplicative function $f: \H \Arrow \R$, its corresponding cocycle $Af \in \Gamma(A^*\H)$ satisfies
\begin{equation}\label{func:eq}
\Lie_{\overrightarrow{\chi}} f = \t^*\<Af,\,\chi\>, \,\;\;\; \forall  \, \chi \in \Gamma(A\H).
\end{equation}
\end{lemma}

\begin{proof}
First, by differentiating equation \eqref{mult:func}, one sees
that
$$
df(U \bullet V) = df(U) + df(V), \,\, \forall (U,V) \in \,T\H_{\,T\s\!}\times_{T\t} T\H.
$$
Also, $\overrightarrow{\chi}(h) = \chi(\t(h))\bullet 0_h$ (where
$0_h$ is the zero vector field on $\H$ at $h$). Hence
$$
(\Lie_{\overrightarrow{\chi}} f)(h) = df(\overrightarrow{\chi}(h)) =
df(\chi(\t(h))\bullet 0_h) = df(\chi(\t(h))) = \<Af, \,\chi\>(\t(h)),
$$
as we wanted.
\end{proof}

It turns out that when $\H$ is source connected, equation
\eqref{func:eq} essentially characterizes multiplicative functions:

\begin{proposition}\label{func:prop}
Let $\H \rightrightarrows N$ be a source-connected Lie groupoid. A
function $f \in C^{\infty}(\H)$ is multiplicative if and only if
$f|_N = 0$ and there exists $\mu\in\Gamma(A^*\H)$ such that
\begin{equation}\label{func:eq2}
\Lie_{\overrightarrow{\chi}} f = \t^* \<\mu, \chi\>, \,\;\;\; \forall \, \chi \in \Gamma(A\H).
\end{equation}
In this case, $Af = \mu$.
\end{proposition}

\begin{proof}
If $f$ is multiplicative, then $f|_N=0$ as a consequence of
\eqref{mult:func}. Also, setting $\mu=Af$,  \eqref{func:eq2} follows
from Lemma \ref{func:eq}.

In the other direction, fix $\chi_1, \chi_2 \in \Gamma(A\H)$. Then
$$
\t^*\<\mu, [\chi_1, \chi_2]\> = \Lie_{[\overrightarrow{\chi_1},
\overrightarrow{\chi_2}]} f = \Lie_{\overrightarrow{\chi_1}}
\Lie_{\overrightarrow{\chi_2}} f - \Lie_{\overrightarrow{\chi_2}}
\Lie_{\overrightarrow{\chi_1}} f = \t^*(\Lie_{\rho(\chi_1)}
\mu(\chi_2) - \Lie_{\rho(\chi_2)} \mu(\chi_1)),
$$
which implies that
$$
d_A\mu(\chi_1, \chi_2) = \Lie_{\rho(\chi_1)}\mu(\chi_2) -
\Lie_{\rho(\chi_2)}\mu(\chi_1) - \mu([\chi_1, \chi_2]) = 0.
$$
So $\mu$ is a Lie-algebroid cocycle. If $\widetilde{\H}$ is a source
1-connected integration of $A\H$, then there exists a multiplicative
function $f_{\mu} \in C^{\infty}(\widetilde{\H})$ such that
$Af_{\mu}=\mu$.  The Lie groupoids $\H$ and $\widetilde{\H}$
correspond to the same Lie algebroid, so they are related by a
groupoid map $\sigma: \widetilde{\H}\to \H$ which is a local
diffeomorphism (and restricts to the identity map on $N$). To check
that $f$ is multiplicative, it is enough to check that $\sigma^*f$ is
multiplicative. We will verify that $\sigma^*f = f_\mu$.

By \eqref{func:eq}, since $\sigma$ is a groupoid morphism (hence commutes
with structure maps and preserves invariant vector fields), one has
that
$$
\Lie_{\overrightarrow{\chi}} \sigma^*f = \sigma^* \Lie_{\overrightarrow{\chi}}
f = \sigma^* \t^*\<\mu,\chi\> = \t^*\<Af_{\mu}, \chi\> =
\Lie_{\overrightarrow{\chi}} f_{\mu},
$$
where we have kept the same notation for the structure maps on
$\widetilde{\H}$ and $\H$. Since $\widetilde{\H}$ has source-connected fibers, it follows that $\sigma^*f-f_{\mu}$ is constant along
the $\s$-fibers. Finally, since $(\sigma^*f- f_{\mu})|_N = 0$, one has
that $\sigma^*f-f_{\mu} =0$ everywhere.
\end{proof}

\subsection{Definition and examples}

Let $\G \rightrightarrows M$ be a Lie groupoid and consider a
$(q,p)$-tensor field $\tau \in \Gamma(\wedge^p \,T^*\G \otimes
\wedge^q T\G)$. Let $c_{\tau}: \bG^{(p,q)} \Arrow \R$ be the
corresponding componentwise linear function,
$$
c_\tau(U_1, \dots, U_p, \xi_1, \dots, \xi_q) = \tau(U_1, \dots, U_p,
\xi_1, \dots, \xi_q),
$$
as in Definition \ref{comp_linear}, where $\bG^{(p,q)} \toto
\M^{p,q)}$ is the Lie groupoid \eqref{big_group}. In the following,
we shall omit the $(p,q)$-indices.

\begin{definition}\label{def:mult}
A $(q,p)$-tensor field $\tau \in \Gamma(\tensor{p}{q})$
 on $\G \toto M$ is {\em multiplicative} if the  function $c_{\tau}: \bG \Arrow \R$ is multiplicative.
\end{definition}

Note that the very same definition of multiplicativity makes sense for elements of
$\Gamma((\otimes^p T^*\G) \otimes (\otimes^q T\G))$ (or for their symmetric versions). Along the paper,
we will make some comments on how to adapt our results to this more general case.

The next examples relate our definition with known
notions of multiplicativity for special types of tensors.

\begin{example}
A differential form $\omega \in \Omega^p(\G)$ is multiplicative if
it satisfies
$$
\m^*\omega = \pr_1^*\omega + \pr_2^*\omega,
$$
where $\pr_i: \grd\,_{\s}\!\times_{\t} \grd \Arrow \grd \Arrow \G$,
$i=1,2$, are the natural projections. One can directly check (see
e.g. \cite{bc}) that this is equivalent to $c_{\omega}$ being
multiplicative.
\end{example}

\begin{example}
A multivector field $\Pi \in \frakx^q(\G)$ is said to be
multiplicative \cite{ILX} if the graph of the multiplication map is
coisotropic with respect to $\Pi \oplus \Pi \oplus (-1)^{q+1} \Pi$:
for $\xi_i \in \mathrm{Ann}(\gra(\m)) \subset
T^*(\G\times \G \times \G)$, $i=1,\ldots,q$, we have
$$
(\Pi\oplus \Pi \oplus (-1)^{q+1} \Pi)(\xi_1, \dots, \xi_q)=0.
$$
It is shown in \cite{ILX} that this condition is equivalent to
$c_{\Pi}$ being multiplicative.
\end{example}

The next result shows that our definition of multiplicativity for
$(1,p)$-tensor fields $K \in \Gamma(\wedge^p T^*\G \otimes T\G)$ agrees with the one given in \cite{LMX}. Note that $K$ can be seen as a map $\overline{K}: \oplus^p T\G \to T\G$.

\begin{proposition}\label{prop:vect_valued}
A $(1,p)$-tensor field $K \in \Gamma(\wedge^p T^*\G \otimes T\G)$ is multiplicative if and only if there is a vector-bundle map
$\overline{r}: \oplus^p TM \to TM$ covering the identity such that
\begin{equation}\label{mult_diagram}
\xymatrix{
\oplus^p T\G \ar@<-3pt>[d] \ar@<3pt>[d] \ar[r]^{\hspace{10pt} \overline{K}} & T\G \ar@<-3pt>[d] \ar@<3pt>[d]\\
\oplus^p TM \ar[r]^{\hspace{10pt} \overline{r}} & TM\\
}
\end{equation}
is a groupoid morphism.
\end{proposition}

\begin{proof}
It is straightforward to check that if \eqref{mult_diagram} is a groupoid morphism, then $K$ is multiplicative.
 Conversely, let us assume that $K$ is a multiplicative $(1,p)$-tensor field. From Proposition \ref{func:prop},
 one knows that $c_K |_{\M} =0$. This implies that, for
 $(X_1, \dots, X_p) \in TM \oplus \dots \oplus TM$, $\overline{K}(X_1,\dots, X_p) \in T\G|_M$ has
 zero component on $A$ under the decomposition $T\G|_M = TM \oplus A$. So, define
$
\overline{r}=\overline{K}|_{\oplus^p TM}.
$

It is straightforward to see that once
$$
T\t \circ \overline{K}  = \overline{r} \circ T\t,\;\;\;
T\s \circ \overline{K}  = \overline{r} \circ T\s,
$$
$\overline{K}$ will preserve the multiplication as a direct consequence of the multiplicativity of $K$. So, let $\alpha \in T_{\s(g)}^*M$.
\begin{align*}
\<T\s(\overline{K}(U_1, \dots, U_p)), \alpha \> & = c_K(U_1, \dots, U_p, (d\s)_g^*\,\alpha) \\
   & \hspace{-50pt} = c_K(U_1 \bullet T\s(U_1), \dots, U_p \bullet T\s(U_p), 0_g \bullet (d\s)_{\s(g)}^*\,\alpha)\\
   & \hspace{-50pt}=c_K(U_1, \dots, U_p, 0_g) + c_K(T\s(U_1), \dots, T\s(U_p), (d\s)_{\s(g)}^*\,\alpha))\\.
   & \hspace{-50pt} = \<\overline{r}(T\s(U_1), \dots, T\s(U_p)), \alpha\>.
\end{align*}
Here, we have used the equality
$
(d\s)_g^*\,\alpha = 0_g \bullet (d\s)_{\s(g)}^*\,\alpha,
$
which follows from \eqref{m:cotang}. The other equality follows similarly.
\end{proof}

We now consider a special class of multiplicative tensor fields on $\G$,
analogous to the multiplicative functions in Example~\ref{ex:multf}. Let $\mathcal{S}, \,\mathcal{T}: \Gamma(\wedge^p \,T^*M
\otimes \wedge^q A) \Arrow \Gamma(\wedge^p \,T^*\G \otimes
\wedge^q T\G)$ be the $\R$-linear maps defined on homogeneous
elements as
\begin{equation}
\label{T:def} \mathcal{T}(\alpha \otimes \frakx) = \t^*\alpha \otimes \overrightarrow{\frakx}, \;\;\;
 \mathcal{S}(\alpha \otimes \frakx)  =
\s^*\alpha \otimes \overleftarrow{\frakx}.
\end{equation}

\begin{proposition} \label{pull:invariance}
The maps $\mathcal{S}$ and $\T$ satisfy the following properties:
\begin{enumerate}
\item $\T(f \Phi) = (\t^*f) \,\T(\Phi)$, \; $\mathcal{S}(f\Phi) = (\s^*f) \, \mathcal{S}(\Phi)$,
\item $\mathbbm{t}^*c_{\Phi} = c_{\mathcal{T}(\Phi)}$,
\item $\mathbbm{s}^*c_{\Phi}  = c_{\mathcal{S}(\Phi)}$,
\end{enumerate}
for $\Phi \in \Gamma(\wedge^p \,T^*M \otimes \wedge^q A)$ and $f \in
C^{\infty}(M)$. In particular,
\begin{equation}\label{cohom_trivial}
(\mathcal{T}-\mathcal{S})(\alpha \otimes \frakx)=
\t^*\alpha \otimes \overrightarrow{\frakx}-\s^*\alpha \otimes \overleftarrow{\frakx}
\end{equation}
is a multiplicative $(q,p)$-tensor field on $\G$.
\end{proposition}

\begin{proof}
The formulas in (1) can be verified directly.

For $\Phi$ of the form $\alpha
\otimes \frakx$ and $\xi_1, \dots,
\xi_q \in T_g^*\G$, formulas (2) and (3) boil down to
$$
\overrightarrow{\frakx}(\xi_1,\dots,
\xi_q) = \frakx(\cott(\xi_1), \dots, \cott(\xi_q)) \text{ and }
\overleftarrow{\frakx}(\xi_1,\dots,
\xi_q) = \frakx(\cots(\xi_1), \dots, \cots(\xi_q)),
$$
respectively, and these identities are direct consequences of the definitions of the source and target maps
of the cotangent groupoid (see \eqref{eq:cotst}). The case of arbitrary $\Phi$ follows by linearity.

The last assertion follows from the fact that $\mathbbm{t}^*c_{\Phi} - \mathbbm{s}^*c_{\Phi}$ is always a
multiplicative function (see Example~\ref{ex:multf}).
\end{proof}

\subsection{The main results: statements and first examples}

We now present a complete infinitesimal characterization of multiplicative tensor fields.
Our first main theorem is an analog of Proposition \ref{func:prop} for general tensor fields.

\begin{theorem}\label{thm:mult_char}
Let $\G \toto M$ be a source-connected Lie groupoid. A $(q,p)$-tensor field $\tau \in \Gamma(\wedge^p T^*\G \otimes \wedge^q T\G)$ is
multiplicative if and only if
\begin{equation}\label{eq:vanish}
\tau(\underline{\X}, \underline{\varphi}) = 0, \;\;\;\forall \,(\underline{\X}, \underline{\varphi}) \in \M,
\end{equation}
and there exist vector-bundle maps $l: A \Arrow \wedge^{p-1} T^*M \otimes \wedge^q A$ and
$r: T^*M \Arrow \wedge^p T^*M \otimes \wedge^{q-1}A$, covering the identity map on $M$,
and an $\R$-linear map $D: \Gamma(A)\Arrow \Gamma(\wedge^p T^*M \otimes \wedge^q A)$,
such that
\begin{align}
\label{D:leibniz} D(fa)& = fD(a) + df \wedge l(a) - a \wedge r(df),\\
\label{eq:conds}
i_{\overrightarrow{a}} \,\tau & =
\T(l(a)),\;\;\;
i_{\t^*\alpha}\,\tau=
\T(r(\alpha)),\;\;\;
\Lie_{\overrightarrow{a}}\, \tau =
\T(D(a)),
\end{align}
for $a \in \Gamma(A)$, $\alpha \in \Gamma(T^*M)$ and $f \in C^{\infty}(M)$.
\end{theorem}

We refer to the triple $(D, l, r)$ as the \textit{infinitesimal components} of the multiplicative tensor $\tau$ and
to equation \eqref{D:leibniz} as the Leibniz condition for $D$.

\begin{remark}\label{rem:notation}
Regarding the notation in Theorem~\ref{thm:mult_char} above, in \eqref{D:leibniz} we view $\Gamma(\wedge^\bullet T^*M \otimes \wedge^\bullet A)$ as a left module for the exterior algebras $\Gamma(\wedge^\bullet A)$ and $\Gamma(\wedge^\bullet T^*M)$, and
both actions are denoted by $\wedge$: for $\mathfrak{Y}\in \Gamma(\wedge^\bullet A)$, $\eta\in \Gamma(\wedge^\bullet T^*M)$, and
$\omega\otimes \mathfrak{X}\in \Gamma(\wedge^\bullet T^*M \otimes \wedge^\bullet A)$,
$$
\mathfrak{Y}\wedge (\omega\otimes \mathfrak{X}) = \omega\otimes (\mathfrak{Y}\wedge \mathfrak{X}), \qquad\;
 \eta \wedge (\omega\otimes \mathfrak{X}) = (\eta\wedge \omega) \otimes \mathfrak{X}.
$$
In \eqref{eq:conds}, the contraction operators are defined as follows: for $U\in \Gamma(T\G)$, $\xi\in \Gamma(T^*\G)$ and
$\tau = \lambda\otimes W$,
$$
i_U\tau = (i_U\lambda)\otimes W,\qquad\; i_\xi\tau = U\otimes (i_\xi W).
$$
\end{remark}

We mention some particular cases of interest.

\begin{corollary}
On a source-connected Lie groupoid $\G \toto M$, a $p$-form $\omega \in \Omega^p(\G)$ is multiplicative if and only if
there is an $\mathbb{R}$-linear map $D: \Gamma(A) \Arrow \Omega^p(M)$ and a vector-bundle morphism $l: A \Arrow \wedge^{p-1} T^*M$ such that the following holds:
$$
1^*\omega  =0,\;\;\;
\Lie_{\overrightarrow{a}} \omega  = \t^* D(a),\;\;\;
i_{\overrightarrow{a}} \omega = \t^* l(a),
$$
for all $a\in \Gamma(A)$.
\end{corollary}

The next result recovers \cite[Thm.~2.19]{ILX} (showing that some conditions there are redundant, cf. \cite[Lem.~2.3]{CSX}).
\begin{corollary}\label{cor:vect}
Let $\G \toto M$ be a source connected Lie groupoid. A $q$-vector field $\Pi \in \frakx^q(\G)$ is multiplicative if and only there is an $\mathbb{R}$-linear map $D: \Gamma(A) \Arrow \Gamma(\wedge^qA)$ and a vector bundle morphism $r: T^*M \Arrow \wedge^{q-1} A$ such that the following holds:
$$
\Pi(\varphi_1, \dots, \varphi_q)=0, \;\;\;\;
\Lie_{\overrightarrow{a}}\Pi = \overrightarrow{D(a)},\;\;\;\; i_{\t^*\alpha} \Pi  = \overrightarrow{r(\alpha)},
$$
for all $(\varphi_1, \dots, \varphi_q) \in A^*\times_M \dots \times_M A^*$, $a\in \Gamma(A)$ and $\alpha\in \Omega^1(M)$.
(In the terminology of \cite{ILX,Mac-Xu2}, the first condition above says that  $M$ is coisotropic with respect to $\Pi$, while
the other conditions say that $\Pi$ is an affine multivector field (see also \cite{DLSW,Kos1}).)
\end{corollary}

Following the previous corollary, Theorem~\ref{thm:mult_char} indicates a notion of {\em affine} tensors on Lie
groupoids, in which we keep all properties of multiplicative tensors described in the theorem except for \eqref{eq:vanish}
(cf. \cite[Thm.~4.5]{We90}).

In the next example, we identify the infinitesimal components of a multiplicative tensor field of type
\eqref{cohom_trivial}. To do so, let us consider, for a Lie algebroid $(A,\rho,[\cdot,\cdot])$, the action of the Lie algebra
$\Gamma(A)$ on $\Gamma(\wedge^p T^*M\otimes \wedge^q A)$ given by
\begin{equation}\label{action}
a \cdot (\beta \otimes \frakx) = \Lie_{\rho(a)}\beta \otimes \frakx + \beta\otimes [a,\frakx],
\end{equation}
where $[\cdot, \cdot]$ is the Schouten bracket on $\Gamma(\wedge^{\bullet} A)$.

\begin{example}\label{trivial:ex}
For $\Phi \in \Gamma(\wedge^p T^*M\otimes \wedge^q A)$, consider the multiplicative tensor given by $\tau = (\mathcal{T}-\mathcal{S})(\Phi)$ (see Prop.~\ref{pull:invariance}). As particular cases, for $q=0$, $\tau = \t^*\Phi - \s^*\Phi$, while for $p=0$, $\tau= \overrightarrow{\Phi} -\overleftarrow{\Phi}$.
The infinitesimal components $(D, l, r)$ corresponding to $\tau$ are
$$
D(a)  = a \cdot \Phi,\;\;\;
l(a)  = i_{\rho(a)} \Phi,\;\;\;
r(\alpha)  = i_{\rho^*(\alpha)} \Phi.
$$
\end{example}


\begin{remark}\label{rem:non_skew1}
For $\tau \in \Gamma((\otimes^p T^*\G) \otimes (\otimes^q T\G))$, there is a result similar to Theorem \ref{thm:mult_char} characterizing  multiplicativity. In this case we have infinitesimal components $(D, l_1, \dots, l_p, r_1, \dots, r_q)$, where $D: \Gamma(A) \to \Gamma((\otimes^p T^*M) \otimes (\otimes^q A))$, $l_i: A \to (\otimes^{p-1} T^*M) \otimes (\otimes^q A)$, and $r_j: T^*M \to (\otimes^{p}T^*M) \otimes (\otimes^{q-1} A)$, defined by
\begin{equation}\label{non_skew_comp}
\Lie_{\overrightarrow{a}} \tau = \T(D(a)),\;\;\;\;\; \!\! \underbrace{\tau(\dots, \overrightarrow{a}, \dots)}_{\text{i-th $T\G$-entry}} \!\!\! \;= \T(l_i(a)),\;\;\;\;\; \!\! \underbrace{\tau(\dots, \t^*\alpha, \dots)}_{\text{ j-th $T^*\G$-entry}}\!\!\! \;= \T(r_j(\alpha)).
\end{equation}
The Leibniz equation for $D$ will change accordingly. For instance, for a multiplicative $\tau \in \Gamma(T\G \otimes T\G)$, the infinitesimal components $(D, r_1, r_2)$ satisfy
$$
D(fa) = fD(a) - a \otimes r_1(df) - r_2(df) \otimes a.
$$
\end{remark}


We now formulate our main result, which concerns the correspondence between multiplicative tensors $\tau$ and their
infinitesimal components $(D,l,r)$. The multiplicativity of $\tau$ is expressed by a set of equations satisfied by $(D,l,r)$,
described in the next definition.


\begin{definition}\label{def:IMtensor}
Let $(A,[\cdot,\cdot],\rho)$ be a Lie algebroid. An {\em IM $(q,p)$-tensor} on $A$  is a triple $(D,l,r)$,
where $l: A \Arrow \wedge^{p-1}T^*M\otimes \wedge^q A$ and $r: T^*M \Arrow \wedge^{p}T^*M\otimes \wedge^{q-1} A$
are vector-bundle maps covering the identity map on $M$,
$D: \Gamma(A) \Arrow \Gamma(\wedge^p T^*M \otimes \wedge^q A)$ is $\R$-linear and satisfies the Leibniz rule
$$
D(fa) = fD(a) + df \wedge l(a) - a \wedge r(df), \;\;\;\; \forall f\in C^\infty(M), \, a\in \Gamma(A),
$$
such that the following equations hold:
\begin{align}
\tag{IM1} \label{IM1} D([a,b])& = a\cdot D(b) - b \cdot D(a),\\
\tag{IM2} \label{IM2} l([a,b])& = a \cdot l(b) - i_{\rho(b)} D(a),\\
\tag{IM3} \label{IM3} r(\Lie_{\rho(a)} \alpha) &= a\cdot r(\alpha) - i_{\rho^*(\alpha)} D(a),\\
\tag{IM4} \label{IM4} i_{\rho(a)}\,l(b) & = - i_{\rho(b)} \,l(a),\\
\tag{IM5} \label{IM5} i_{\rho^*\alpha}\, r(\beta) &= - i_{\rho^*\beta}\, r(\alpha),\\
\tag{IM6} \label{IM6} i_{\rho(a)}\,r(\alpha) & = i_{\rho^*\alpha}\,l(a),
\end{align}
for $a,b \in \Gamma(A)$ and $\alpha,\beta \in \Omega^1(M)$.
\end{definition}

We refer to equations \eqref{IM1} -- \eqref{IM6} as the {\em IM-equations} of an IM $(q,p)$-tensor.

\begin{remark}[Redundancies]\label{rem:redundacy}
We observe that, in many cases, there are some redundancies in the IM-equations. Note first that, if $p > \dim(M)+1$ or $q > \mathrm{rank}(A)+1$, then any IM $(p,q)$-tensor is trivial. On the other hand,
when $p < \dim(M)+1$ and $q < \mathrm{rank}(A)+1$, we claim that $\{\eqref{IM1}, \eqref{IM2}, \eqref{IM6}\}$ and  $\{\eqref{IM1}, \eqref{IM3}, \eqref{IM6}\}$ are minimal sets of
independent IM-equations. Indeed, in this case, \eqref{IM1} implies that
$$
df \wedge (l([a,b]) - a\cdot l(b) + i_{\rho(b)}D(a)) = b \wedge (r(\Lie_{\rho(a)}(df)) - a\cdot r(df) + i_{df}D(a)),
$$
for all $f\in \C(M), a, b \in \Gamma(A)$. This follows from the Leibniz rule for the Lie bracket $[\cdot,\cdot]$,
the Leibniz formula for $D$, and Proposition \ref{Lie:leibniz} below. Therefore,
$$
\eqref{IM1} + \eqref{IM2} \Rightarrow \eqref{IM3} \text{ and } \eqref{IM1} + \eqref{IM3} \Rightarrow \eqref{IM2}.
$$
Similarly, one can prove that \eqref{IM2} (resp. \eqref{IM3})  implies that
\begin{align*}
df\wedge(i_{\rho(a)} l(b) + i_{\rho(b)} l(a)) & = a \wedge (i_{df}l(b)-i_{\rho(b)}r(df))\\
(\text{resp. } a \wedge (i_{df} r(\alpha) + i_{\rho^*\alpha} r(df)) & = df \wedge(i_{\rho(a)} r(\alpha) - i_{\rho^*\alpha}l(a))).
\end{align*}
As a result,
$
\eqref{IM2} + \eqref{IM6} \Rightarrow \eqref{IM4} \text{ and } \eqref{IM3} + \eqref{IM6} \Rightarrow \eqref{IM5}.
$
\end{remark}

Our main theorem can now be stated as follows:

\begin{theorem}\label{thm:main}
Let $\G \toto M$ be a source 1-connected Lie groupoid, and let $A$ be its Lie algebroid. There is a one-to-one correspondence
between multiplicative $(q,p)$-tensor fields $\tau \in \Gamma(\wedge^p T^*\G\otimes \wedge^q T\G)$ and IM $(q,p)$-tensors $(D,l,r)$
on $A$ satisfying
\begin{equation*}
\begin{array}{rl}
i_{\overrightarrow{a}} \,\tau = & \hspace{-5pt} \T(l(a))\\
i_{\t^*\alpha}\,\tau= & \hspace{-5pt} \T(r(\alpha))\\
\Lie_{\overrightarrow{a}}\, \tau =  & \hspace{-5pt} \T(D(a)),
\end{array}
\end{equation*}
where $\T$ is the map given by \eqref{T:def}.
\end{theorem}

\begin{remark}\label{rem:non_skew2}
The correspondence in Theorem \ref{thm:main} can be naturally extended to multiplicative tensors $\tau \in \Gamma((\otimes^p T^*\G) \otimes (\otimes^q T\G))$; in this case, following Rem.~\ref{rem:non_skew1},
we have more infinitesimal components  $(D, l_1, \dots, l_p,$  $r_1, \dots, r_q)$, obtained as in \eqref{non_skew_comp}, satisfying analogous IM-equations.
\end{remark}

We now show how Theorem~\ref{thm:main}, when restricted to $(0,p)$ and $(q,0)$ tensors, directly recovers the infinitesimal
descriptions of multiplicative differential forms and multivector fields, proven in \cite{AC,BC,ILX}.
It also recovers the correspondences in \cite{LMX,StX}, but we will leave this discussion to Section \ref{1_1},
where we present a more general treatment of multiplicative $(1,1)$-tensor fields.

\subsubsection*{Multiplicative differential forms}\label{diff_forms}
According to Theorem \ref{thm:main}, multiplicative $p$-forms correspond, infinitesimally, to IM $(0,p)$-tensors.
Explicitly, on a given Lie algebroid $A\to M$, IM $(0,p)$-tensors are pairs $(D,l)$, where
$D: \Gamma(A) \Arrow \Gamma(\wedge^p T^*M)$, $l: A \Arrow \wedge^{p-1} T^*M$, and these maps satisfy
$$
D(fa) = f D(a) + df \wedge l(a),
$$
for $a\in \Gamma(A)$ and $f\in C^\infty(M)$, and
\begin{align*}
D([a,b]) & = \Lie_{\rho(a)} D(b) - \Lie_{\rho(b)} D(a)\\
l([a,b]) & = \Lie_{\rho(a)} l(b) - i_{\rho(b)} D(a)\\
i_{\rho(a)}l(b) & = -i_{\rho(b)} l(a).
\end{align*}

Note that IM $(0,p)$-tensors agree with Spencer operators with values in the trivial representation, as considered in \cite{CSS}.

\begin{example}\label{poisson_(2,0)}
On a Poisson manifold $(M, \Pi)$, the cotangent bundle $T^*M$ has a Lie algebroid structure whose anchor is given by the
contraction of covectors with $\Pi$, $\Pi^\sharp: T^*M \to TM$, and the Lie bracket is given by
\begin{equation}\label{bracket:pi}
[\alpha,\beta]_\Pi = \Lie_{\Pi^\sharp(\alpha)} \beta - \Lie_{\Pi^\sharp(\beta)} \alpha - d(i_{\Pi^\sharp(\alpha)} \beta),
\,\, \alpha, \,\beta \in \Gamma(T^*M).
\end{equation}
There exists a canonical IM $(0,2)$-tensor on $T^*M$ given by $D=d$, the de Rham differential, and $l=\mathrm{id}_{T^*M}$.
\end{example}

We have the following alternative way to express IM $(0,p)$-tensors, see \cite{AC,bc,BCO}. 
We consider pairs $(\mu, \nu)$ with $\mu: A \Arrow \wedge^{p-1} T^*M$, $\nu: A \Arrow \wedge^p T^*M$ bundle maps (covering the identity)
satisfying
\begin{align*}
\nu([a,b]) & = \Lie_{\rho(a)} \nu(b) - i_{\rho(b)} d\nu(a),\\
\mu([a,b]) & = \Lie_{\rho(a)} \mu(b) - i_{\rho(b)} (d\mu(a) + \nu(a)),\\
i_{\rho(a)} \mu(b) & = -i_{\rho(b)} \mu(a),
\end{align*}
for all $a, b \in \Gamma(A)$. Such a pair $(\mu,\nu)$ is called an {\em IM $p$-form} in \cite{bc}.
The equivalence between IM $(0,p)$-tensors $(D,l)$ and IM $p$-forms $(\mu,\nu)$ is given by the following explicit relations:
\begin{align*}
D(a) & =  d\mu(a) + \nu(a), \\
l(a) & =  \mu(a).
\end{align*}

When $(D,l)$ are the infinitesimal components of a  multiplicative $(0,p)$ tensor field (i.e. a differential $p$-form)
$\omega \in \Omega^p(\G)$, we have that
$$
\Lie_{\overrightarrow{a}} \omega = \t^*(d\mu(a) + \nu(a)), \qquad  i_{\overrightarrow{a}} \omega = \t^* \mu(a)
$$
for all $a \in \Gamma(A)$. It follows that
$$
\t^*\nu(a) = \Lie_{\overrightarrow{a}} \omega - \t^*d\mu(a) =
i_{\overrightarrow{a}} d\omega + d(i_{\overrightarrow{a}} \omega - \t^*\mu(a)) = i_{\overrightarrow{a}} d\omega.
$$

In this way, we can see that Theorem~\ref{thm:main} immediately recovers the main result of \cite{bc},
relating multiplicative and IM differential forms:

\begin{corollary}\label{cor:forms}
For a source 1-connected groupoid $\G \toto M$, there is a one-to-one
correspondence between multiplicative $p$-forms $\omega \in \Omega^p(\G)$ and IM $p$-forms $(\mu, \nu)$, given by
$$
\t^*\mu(a) = i_{\overrightarrow{a}}\omega, \qquad \t^*\nu(a) = i_{\overrightarrow{a}} d\omega,
$$
for $a\in \Gamma(A)$.
\end{corollary}

For a Poisson manifold $(M, \Pi)$, if $(T^*M, [\cdot,\cdot]_\Pi)$ is the Lie algebroid of a source 1-connected groupoid
$\G \toto M$, then $\G$ has a canonical multiplicative 2-form integrating the canonical IM $(0,2)$-tensor of Example \ref{poisson_(2,0)}.
In this case, note that $\nu=0$, and this implies that $d\omega=0$.
One can also verify that $l=\mu$ being an isomorphism implies that $\omega$ is
non-degenerate. So $(\G, \omega)$ is the symplectic groupoid integrating $T^*M$.

\begin{remark}\label{rem:complexes}
Multiplicative differential forms on a Lie groupoid $\G$ form a subcomplex of the de Rham complex. On the other hand,
if $(\mu,\nu)$ is an IM $p$-form, one can directly verify that $(\nu,0)$ is an IM $(p+1)$-form. So $(\mu,\nu)\mapsto (\nu,0)$
defines a differential on the space of all IM-forms, in such a way that
the correspondence in Cor.~\ref{cor:forms} is an isomorphism of complexes.
\end{remark}

\subsubsection*{Multiplicative multivector fields}
Let us now consider the case of multiplicative $(q,0)$-tensor fields, i.e., multiplicative $q$-vector fields.
Following Theorem~\ref{thm:main}, their infinitesimal counterparts are IM $(q,0)$-tensors: on a given Lie algebroid
$A\to M$, these are pairs $(D,r)$, where $r: T^*M\to \wedge^{q-1} A$ is a vector-bundle map (covering the identity),
$D: \Gamma(A)\to \Gamma(\wedge^q A)$ is $\R$-linear and satisfies
\begin{equation}\label{eq:Leibniz_(0,2)}
D(f a) = fD(a) - a\wedge r(d f) = fD(a) + (-1)^{q}r(df) \wedge a,
\end{equation}
for $a\in \Gamma(A)$ and $f\in C^\infty(M)$, and the following compatibility conditions hold:
\begin{align*}
D([a,b]) &  = [a,D(b)] - [b, D(a)] = [D(a),b] + [a,D(b)]\\
r(\Lie_{\rho(a)}\alpha) & = [a,r(\alpha)] - i_{\rho^*\alpha} D(a) \\
i_{\rho^*\alpha} r(\beta) &= - i_{\rho^*\beta} r(\alpha),
\end{align*}
for $a, b \in \Gamma(A)$, $\alpha, \beta \in \Omega^1(M)$.

In order to make the connection with the work in \cite{ILX}, recall that a {\em $q$-differential} on a Lie algebroid $A$
is an $\R$-linear map $\delta: \Gamma(\wedge^{\bullet}A) \Arrow \Gamma(\wedge^{\bullet+q-1} A)$ satisfying
\begin{align}
\label{eq:q1}\delta(\frakx_1 \wedge \frakx_2) & = \delta(\frakx_1) \wedge \frakx_2 + (-1)^{k_1(q-1)} \frakx_1 \wedge \delta(\frakx_2)\\
\label{eq:q2}\delta([\frakx_1,\frakx_2]) & = [\delta(\frakx_1) , \frakx_2] + (-1)^{(k_1-1)(q-1)} [\frakx_1, \delta(\frakx_2)]
\end{align}
where $\frakx_i \in \Gamma(\wedge^{k_i}A)$, $i=1,2$, and $[\cdot, \cdot]$ is the Schouten bracket on $\Gamma(\wedge^\bullet A)$
(which makes it into a Gerstenhaber algebra).
In other words, $\delta$ is a derivation of degree $(q-1)$ of $\Gamma(\wedge^\bullet A)$ which is also a derivation of the Schouten
bracket. We denote the space of $q$-differentials by $\mathcal{A}_q$. The space
$\mathcal{A} = \oplus_{q \geq 0} \mathcal{A}_q$ is naturally a {\em Gerstenhaber algebra} with respect to the
bracket given by the commutator
$$
[\delta, \widetilde{\delta}] = \delta \circ \widetilde{\delta} - (-1)^{(q-1)(\widetilde{q}-1)}\, \widetilde{\delta}\circ \delta,
$$
where $\delta \in \mathcal{A}_q$ and $\widetilde{\delta} \in \mathcal{A}_{\widetilde{q}}$.

Note that a $q$-differential $\delta$ is determined by its restrictions
$$
\delta_0: \C(M) \Arrow \Gamma(\wedge^{q-1}A), \qquad \delta_1:  \Gamma(A)\Arrow \Gamma(\wedge^q A).
$$
For this reason, we may denote a $q$-differential by the pair $(\delta_0, \delta_1)$.

Before studying the relationship between $q$-differentials and IM $(q,0)$-tensors,
we list here some properties of the Schouten bracket that we need (see e.g. \cite{KM,Kosz}):
\begin{enumerate}
\item For $f \in C^{\infty}(N)$, $\frakx \in \Gamma(\wedge^q A)$,
$$
[\frakx, f] = (-1)^{q-1} i_{\rho^*df} \frakx.
$$
\item For $\frakx_i \in \Gamma(\wedge^{q_i} A)$, $i=1,2$,
$$
[\frakx_1, \frakx_2]=-(-1)^{(q_1-1)(q_2-1)}[\frakx_2, \frakx_1].
$$
\item For $\frakx_i \in \Gamma(\wedge^{q_i} A)$, $i=1,2,3$,
$$
[\frakx_1, \frakx_2\wedge \frakx_3] = [\frakx_1, \frakx_2]\wedge \frakx_3 + (-1)^{(q_1-1)q_2} \frakx_2 \wedge [\frakx_1, \frakx_3].
$$
\end{enumerate}
\begin{lemma}
There is a one-to-one correspondence between $q$-differentials $(\delta_0, \delta_1)$ and IM $(q,0)$-tensors $(D, r)$ via
\begin{equation}\label{qdif}
\delta_1= D, \qquad \delta_0 = (-1)^q  r\circ \dd,
\end{equation}
where $\dd$ is the de Rham differential.
\end{lemma}

\begin{proof}
Note that \eqref{eq:q1}, \eqref{eq:q2} give rise to five equations involving $\delta_0$ and $\delta_1$,
to be compared with the four equations characterizing IM $(q,0)$-tensors.

For $k_1=k_2=0$, \eqref{eq:q1} is equivalent to the existence of a vector-bundle map
$r_{0}: T^*M \Arrow \wedge^{q-1} A$ such that $r_{0}(d f)=\delta_0(f)$.
This guarantees that we can always assume that $\delta_0$ is of the form described in \eqref{qdif}, i.e.,
we set $r = (-1)^q r_0$ and $D=\delta_1$.

For $k_1=1$ and $k_2=0$, \eqref{eq:q1} becomes
$$
D(a f) = D(a)f + (-1)^{q-1} a\wedge (-1)^q r(df) = fD(a) - a\wedge r(df),
$$
which is just the Leibniz rule for $(D,r)$.

Next, when $k_1=k_2=0$  \eqref{eq:q2} reads
$$
0 = (-1)^q ([r(df),g] + (-1)^{q-1}[f,r(dg)]).
$$
Using that the Schouten bracket satisfies $[\frakx, f] = (-1)^{k-1} i_{\rho^*df} \frakx$,
for $f\in C^\infty(M)$ and $\frakx\in \Gamma(\wedge^k A)$, we see that this last equation boils down to
$i_{\rho^*dg} r(df) = - i_{\rho^*df} r(dg)$, which is equivalent to the condition
$$
i_{\rho^*\alpha} r(\beta) = - i_{\rho^*\beta} r(\alpha)
$$
for $\alpha,\beta \in \Omega^1(M)$.

It is immediate that, for $k_1=k_2=1$, \eqref{eq:q2} becomes
$$
D([a,b]) = [D(a),b] + [a,D(b)].
$$
Finally, when $k_1=1$, $k_2=0$, \eqref{eq:q2} amounts to
$$
(-1)^q r(\Lie_{\rho(a)} df) = (-1)^{q-1}i_{\rho^* df} D(a) + (-1)^q [a, r(df)],
$$
which is equivalent to the condition
$$
r(\Lie_{\rho(a)} \alpha) = -  i_{\rho^* \alpha} D(a) + [a, r(\alpha)].
$$
\end{proof}

For a multiplicative $q$-vector $\Pi$ on a Lie groupoid $\G\toto M$, its infinitesimal components are written
in terms of the corresponding $q$-differential as follows:
\begin{align*}
&\Lie_{\overrightarrow{a}}\Pi = \overrightarrow{D(a)} = \overrightarrow{\delta_1(a)},\\
& i_{\t^*df} \Pi = \overrightarrow{r(df)} = (-1)^q \overrightarrow{\delta_0(f)}.
\end{align*}
Since $[\overrightarrow{a},\Pi] = \Lie_{\overrightarrow{a}}\Pi$ and $[\Pi, \t^*f] = (-1)^{q-1}i_{\t^*df} \Pi$,
and using the fact that $\delta$ is a $q$-differential if and only if so is $-\delta$,
we see that Theorem~\ref{thm:main} recovers the following correspondence (which is the central result in \cite{ILX};
see also \cite{BC}):

\begin{corollary}
For a source 1-connected Lie groupoid $\G \toto M$, there is a one-to-one correspondence between
multiplicative $q$-vector fields $\Pi \in \frakx^q(\G)$ and $q$-differentials $(\delta_0, \delta_1)$, given by
$$
\overrightarrow{\delta_0(f)}  = [\Pi, \t^*f], \qquad
\overrightarrow{\delta_1(a)}  = [\Pi, \overrightarrow{a}].
$$
for $f\in C^\infty(M)$ and $a\in \Gamma(A)$.
\end{corollary}

\begin{remark}\label{rem:gerst}
In \cite{ILX}, it is verified that the space of multiplicative multivector fields on a Lie groupoid $\G$ is closed under
the Schouten bracket, so it is a Gerstenhaber subalgebra of the space of all multivector fields on $\G$. The correspondence
in the previous corollary is proven to give rise to an isomorphism of Gerstenhaber algebras.
\end{remark}

The next example shows how the correspondence between Poisson groupoids and Lie bialgebroids \cite{Mac-Xu2} fits into
the framework of IM $(2,0)$-tensors.

\begin{example}\label{exam:poisson_groupoid}
For the case $q=2$, there is a one-to-one correspondence between pairs $(D,r)$ satisfying
\eqref{eq:Leibniz_(0,2)} and pre-Lie algebroid
\footnote{A pre-Lie algebroid structure on a vector bundle $E\to M$ consist of an anchor map $\rho_E: E \to TM$ together
with a skew-symmetric bilinear bracket $[\cdot, \cdot]$ on $\Gamma(E)$ such that the Leibniz equation
$[u, fv] = f[u,v] + (\Lie_{\rho_E(u)}f)\,v$ holds (see \cite{GU3}).} structures on $A^*$ given as follows:
$\rho_*: A^* \to TM$ is the dual map to $r: A\to T^*M$, and the bracket $[\cdot, \cdot]_*$ is determined by the Koszul formula:
\begin{equation}\label{pre_lie_br}
\<[\mu_1, \mu_2]_*, a \> = \Lie_{\rho_*(\mu_2)}\<\mu_1, a\> - \Lie_{\rho_*(\mu_1)}\<\mu_2, a\> - D(a)(\mu_1, \mu_2),
\end{equation}
for $a \in \Gamma(A)$ and $\mu_1, \mu_2 \in \Gamma(A^*)$. This is just another incarnation of the known correspondence between pre-Lie algebroids structures on $A^*$ and linear bivector fields on $A$  \cite{GU3}
(see also Corollary \ref{cor:1-1}).

Recall that a {\bf Poisson groupoid} is a Lie groupoid $\G \toto M$ endowed with a multiplicative 2-vector field
$\Pi \in \mathfrak{X}^2(\G)$ such that $[\Pi,\Pi]=0$.
The IM $(2,0)$-tensor $(D,r)$ on its Lie algebroid $A$ corresponding to $\Pi$ defines a pre-Lie algebroid structure on $A^*$.
It has an associated operator $\delta: \Gamma(\wedge^\bullet A) \to \Gamma(\wedge^{\bullet+1} A)$ which is
the 2-differential defined by \eqref{qdif}. Now, the fact that $[\Pi,\Pi]=0$ implies that $\delta^2=0$,
which says that the Jacobi identity holds for the pre-Lie bracket $[\cdot, \cdot]_*$;
so $A^*$ is a Lie algebroid. Finally, the IM-equation
$$
D([a,b]) = [D(a),b] + [a, D(b)],\,\,a, \, b \in \Gamma(A),
$$
gives the compatibility condition for $(A,A^*)$ to be a Lie bialgebroid.
This is the only relevant IM-equation because of the redundancies explained in Rem.~\ref{rem:redundacy}
(see \cite{Kos0}).
\end{example}

\section{Proof of the Theorems}\label{sec:proofs}

Before delving into the proofs of Theorems~\ref{thm:mult_char} and \ref{thm:main},
let us briefly sketch the general strategy to obtain the infinitesimal description of multiplicative tensors.

Given a multiplicative $(q,p)$-tensor field $\tau \in \Gamma(\wedge^p T^*\G \otimes \wedge^q T\G)$, our main object of analysis
is formula \eqref{func:eq} applied to $c_\tau$, the corresponding multiplicative function on
the groupoid $\bG \toto \M$ in \eqref{big_group}:
\begin{equation}\label{eq:main}
\Lie_{\overrightarrow{\chi}} \, c_\tau = \bt^*\<A c_{\tau}, \chi\>.
\end{equation}
Our goal is to have a concrete description of the infinitesimal cocycle $Ac_\tau \in \Gamma(\bA^*)$,
which codifies the infinitesimal information of $\tau$.

The first key observation is that it  is enough to check the identity \eqref{eq:main} when
$\chi$ varies within a special set of generators for the $C^\infty(\M)$-module of sections of
$\bA$. These generators will be parametrized by $\Gamma(A)$ and $\Omega^1(M)$, and their pairing with
$Ac_{\tau}$ will give rise to maps from the space of parameters into $C^{\infty}(\M)$; more precisely,
we will obtain three maps, $D$, $l$, and $r$, taking values in the subspace
$\Gamma(\wedge^p T^*M \otimes \wedge^q A) \subseteq C^{\infty}(\M)$
of componentwise linear functions. These maps, which completely determine $Ac_\tau$, will agree with the infinitesimal components
of $\tau$.

Considering the left-hand side of \eqref{eq:main}, we will see that the Lie derivatives of $c_\tau$
can be expressed in terms of contraction and Lie-derivative operations on the tensor field $\tau$ itself. In this way,
the equality \eqref{eq:main} is re-written as the relations involving $\tau$ and $(D,l.r)$ in Theorem~\ref{thm:main}.
The last step is expressing the cocycle condition $d_{\bA}(Ac_\tau)= 0$,
where $d_{\bA}$ is the Lie algebroid differential on $\Gamma(\bA^*)$, in terms of $(D,l,r)$.
This will lead to the  IM-equations.

We will need a few technical tools to carry out this strategy, including the study of lifting operations (Section~\ref{lift_chapter}) and an analysis of linear tensor fields (Section~\ref{linear}).


\subsection{Lifting operations}\label{lift_chapter}
As we now see, classical lifting operations (see e.g. \cite{YI})
of vector fields are essential ingredients in relating Lie derivatives of tensor fields $\tau$
with Lie derivatives of the corresponding componentwise linear functions $c_\tau$.

Let $\pi_E: E \to N$ be a vector bundle over a smooth manifold $N$. Given a section $u \in \Gamma(\E)$,
its \textit{vertical lift} is the vector field $u^{\vl}: \E \Arrow T\E$ on $\E$ defined by
\begin{equation}\label{vert_lift_def}
u^{\vl}(e) = \left.\frac{d}{d\epsilon}\right|_{\epsilon=0} (e + \epsilon \, u(y)), \,\,\,\,y \in N, \,e \in \E_y.
\end{equation}
For a section $\psi \in \Gamma(\E^*)$, we recall that the Lie derivative of its corresponding linear function
$\ell_{\psi} \in C^{\infty}(\E)$ along $u^{\vl}$ is given by
\begin{equation}\label{vert_lift}
\Lie_{u^{\vl}}\,\ell_{\psi} = \langle \psi, u \rangle \circ \pi_E.
\end{equation}

In this paper, we are mostly interested in the cases where
$\E= TN$ or $\E=A^*$, the dual of a Lie algebroid $(A, [\cdot, \cdot], \rho)$ over $N$. We denote the bundle projections
by $\pi: TN \Arrow N$ and $\pi_*: A^* \Arrow N$. In these cases, besides the vertical lifting, there are two other important
lifting constructions that we need to recall. For a vector field $X \in \frakx(N)$, consider its (local) flow
$\phi_{\epsilon}: N \Arrow N$. The \textit{tangent lift of} $X$ is the vector field $X^T$ on $TN$ with flow given by
$\epsilon \mapsto T\phi_{\epsilon}$. For $\alpha \in \Omega^1(N)$ and $f\in C^{\infty}(N)$,
the Lie derivatives of the functions $\ell_{\alpha}$ and $f \circ \pi$ in $C^{\infty}(TN)$ along $X^T$ are given by
\begin{equation}\label{tang_lift}
\Lie_{X^{T}} \,\ell_{\alpha} = \ell_{\Lie_{X}\alpha} \,\,\, \text{ and } \,\,\,\Lie_{X^T} (f\circ \pi) = (\Lie_X f) \circ \pi.
\end{equation}

The \textit{Hamiltonian lift of} a section $a \in \Gamma(A)$ of a Lie algebroid is the vector field $H_a$ on $A^*$ defined by
\begin{equation}\label{H:def}
H_a = \Pi_{lin}^{\sharp}(d\ell_a),
\end{equation}
where $\Pi_{lin} \in \Gamma(\wedge^2 TA^*)$ is the linear Poisson structure on $A^*$ (dual to the Lie algebroid structure on $A$)
and $\ell_a \in C^{\infty}(A^*)$ is the linear function corresponding to $a$. For $b \in \Gamma(A)$ and $f \in C^{\infty}(N)$,
\begin{equation}\label{ham_lift}
\Lie_{H_a} \, \ell_b = \ell_{[a,b]}\,\, \text{ and } \Lie_{H_a}(f \circ \pi_*) = (\Lie_{\rho(a)} f) \circ \pi_*.
\end{equation}


Note that \eqref{tang_lift} and \eqref{ham_lift} completely
characterize $X^T$ and $H_a$, respectively. When $A= TN$, the linear Poisson structure on $T^*N$
comes from the canonical symplectic form, and the Hamiltonian lift $H_X$ of a vector field $X \in \frakx(N)$
coincides with the \textit{cotangent lift} $X^{T^*} \in \frakx(T^*N)$, which is the vector field with flow
$\epsilon \mapsto (T\phi_{-\epsilon})^*$.

Our aim is to extend formulas \eqref{vert_lift}, \eqref{tang_lift} and \eqref{ham_lift} to elements of the space
$\Gamma(\wedge^p T^*N \otimes \wedge^q A)$. Let us first introduce some notation. For $Y \in \frakx(M)$,
we define vector fields on $\oplus^p TM$ as follows:
\begin{align}
\label{Yp} Y^{T,\,p}(X_1, \cdots, X_p) & = (Y^T(X_1), \dots, Y^T(X_p))\vspace{5pt}\\
\label{Yv} Y^{\vl,\,p}_{(i)}(X_1, \cdots, X_p) & = (0_{X_1}, \dots, 0_{X_{i-1}}, Y^{\vl}(X_i), 0_{X_{i+1}}, \dots, 0_{X_p}), \,\,\, i=1, \dots, p.
\end{align}
Similarly, for $a \in \Gamma(A)$ and $\mu \in \Gamma(A^*)$, we define vector fields on $\oplus^q A^*$  by
\begin{align}
\label{Hq} H_a^q(\varphi_1, \cdots, \varphi_q) &= (H_a(\varphi_1), \dots, H_a(\xi_q))\\
\label{mu_v} \mu^{\vl,\,q}_{(j)}(\varphi_1, \cdots, \varphi_q) & =
( 0_{\varphi_1}, \dots, 0_{\varphi_{j-1}}, \mu^{\vl}(\varphi_j),  0_{\varphi_{j+1}}, \dots, 0_{\varphi_{p}}), \,\,\, j=1, \dots, q.
\end{align}

Define
\begin{align*}
\gamma_{(i,0)}^{(p,q)}:& (\oplus^p TN) \oplus(\oplus^q A^*) \Arrow (\oplus^{p-1} TN) \oplus(\oplus^q A^*),\\
\gamma_{(0,j)}^{(p,q)}:& (\oplus^p TN) \oplus(\oplus^q A^*) \Arrow (\oplus^p TN) \oplus (\oplus^{q-1} A^*)
\end{align*}
to be the projections
\begin{align}
\label{proj:i}\gamma^{(p,q)}_{(i,0)}(\underline{\X}, \underline{\varphi}) & =
(X_1, \dots, X_{i-1}, X_{i+1}, \dots, X_p, \varphi_1, \dots, \varphi_q)\\
\label{proj:j}\gamma^{(p,q)}_{(0,j)}(\underline{\X}, \underline{\varphi}) & =
(X_1, \dots, X_p, \varphi_1, \dots, \varphi_{j-1}, \varphi_{j+1}, \dots, \varphi_q),
\end{align}
for $1 \leq i \leq p$, $1 \leq j \leq q$. When there is no risk of
confusion, we simplify the notation by omitting the superscripts $(q,p)$ on the projections.

\begin{proposition}\label{Lie_tensor}
Let $\tau \in \Gamma(\wedge^p T^*N \otimes \wedge^q A)$,
and consider the corresponding componentwise linear function
$c_{\tau}: (\oplus^p TN) \oplus (\oplus^q A^*) \Arrow \R$. For $a \in \frakx(A)$, $\mu \in \Gamma(A^*)$ and $Y \in \Gamma(TN)$,
one has that
\begin{align*}
\Lie_{(\rho(a)^{T, p}, H_a^q)} \,c_{\tau} & = c_{\,a\cdot \tau} \\
\Lie_{(Y^{\vl,\,p}_{(i)}, 0)} \,c_{\tau} & = (-1)^{i-1}c_{i_Y \tau}\circ \gamma_{(i,0)}\\
\Lie_{(0, \mu_{(j)}^{\vl,\, q})} \,c_{\tau} & = (-1)^{j-1} c_{i_{\mu} \tau} \circ \gamma_{(0,j)},
\end{align*}
where $\cdot$ is the action \eqref{action}.
\end{proposition}

\begin{proof}
Let us consider the case $\tau = \omega \otimes \frakx$, for
$\omega \in \Omega^p(N)$ and $\frakx \in \mathfrak{X}^q(A)$. First note that
$
c_\tau = (c_\omega \circ \pr_{TN}) (c_\frakx \circ \pr_{A^*}),
$
where $\pr_{TN}$ and $\pr_{A^*}$ are the projections of $(\oplus^p TN) \oplus (\oplus^q A^*)$ onto $\oplus^p TN$ and $\oplus^q A^*$,
respectively. Using the Leibniz rule, it suffices to prove that
\begin{align*}
\Lie_{Y^{T,\, p}_{(i)}} \,c_{\omega} & =  c_{\Lie_{Y}\omega},\,\,\,\ \Lie_{Y^{\vl,\,p}_{(i)}}\,c_{\omega} =
(-1)^{i-1} c_{i_Y \omega} \circ \gamma_{(i,0)}\\
\Lie_{H_a^q} \,c_{\frakx} & =  c_{[a,\frakx]}, \,\,\,\, \Lie_{\mu^{\vl, \,q}_{(j)}}\, c_{\frakx}=
(-1)^{j-1} c_{\,i_{\mu} \frakx} \circ \gamma_{(0,j)}.
\end{align*}
Let us simplify matters once more by assuming that $\omega= \alpha_1 \wedge \cdots \wedge \alpha_p$,
for $\alpha_1, \dots, \alpha_p \in \Gamma(T^*N)$. On the one hand, we have that
$$
\begin{array}{rl}
c_{i_Y\omega} \circ \gamma_{(i)} (X_1, \dots, X_p) = & i_Y\omega (X_1, \dots, X_{i-1}, X_{i+1}, \dots, X_p) \\
= & \omega (Y, X_1, \dots, X_{i-1}, X_{i+1}, \dots, X_p)\\
 = & (-1)^{i-1} \omega(X_1, \dots, X_{i-1}, Y, X_{i+1}, \dots, X_p).
\end{array}
$$
On the other hand, by \eqref{skew_comp} and \eqref{vert_lift},
$$
\Lie_{Y^{\vl, \,p}_{(i)}}\, c_{\omega} =
\sum_{\sigma \in S(p)} sgn(\sigma) \langle \alpha_{\sigma(i)}, Y \rangle (\ell_{\alpha_{\sigma(1)}} \circ \pr^1_{TN})
\cdots \widehat{(\ell_{\alpha_{\sigma(i)}}  \circ \pr_{TN}^{i})} \cdots (\ell_{\alpha_{\sigma(p)}}  \circ \pr_{TN}^{p}),
$$
where $\pr^j_{TN}: \oplus^p TN \Arrow TN$ is the projection on the $j$-component, for $1 \leq j \leq p$.
Hence,
\begin{align*}
\left(\Lie_{Y^{\vl, \,p}_{(i)}} \,c_{\omega}\right) (X_1, \dots, X_p) & = \omega(X_1, \dots, X_{i-1}, Y, X_{i+1}, \dots, X_p)\\
     & = (-1)^{i-1} c_{i_Y\omega} \circ \gamma_{(i)} (X_1, \dots, X_p).
\end{align*}
The other equations follow similarly using \eqref{tang_lift} and \eqref{ham_lift}. The case where $\tau$ is arbitrary follows
from linearity of the Lie derivative.
\end{proof}

\begin{remark} \label{rem:cotglift}
When $A=TN$, the action $\cdot$ of $X \in \mathfrak{X}(N)$ on
$\tau \in \Gamma(\wedge^p T^*N \otimes \wedge^q TN)$ is the Lie
derivative of $\tau$ along $X$,
$X \cdot \tau = \Lie_X \tau.$
In this case, Proposition \ref{Lie_tensor} says that
$$
\Lie_{(X^{T,p}, X^{T^*,q})} \, c_{\tau} = c_{\Lie_X \tau},
$$
where $X^T$ and $X^{T^*}$ are the tangent and cotangent lifts of $X$.
\end{remark}

For our next result, we keep the notation as in Remark~\ref{rem:notation}.


\begin{proposition}\label{Lie:leibniz}
Let $\tau \in \Gamma(\wedge^p T^*N \otimes \wedge^q A)$ and $a \in \Gamma(A)$. For $f \in \C(N)$,
$$
(f a) \cdot \tau = f (a\cdot \tau) + df \wedge i_{\rho(a)}\tau - a \wedge i_{df} \tau.
$$
\end{proposition}

\begin{proof}
By linearity, it suffices to prove the result for $\tau = \omega \otimes \frakx$,
where $\omega \in \Gamma(\wedge^p T^*N)$ and $\mathfrak{X} \in \Gamma(\wedge^q A)$.
First,
$$
\Lie_{\rho(f a)} \,\omega  = f \, \Lie_{\rho(a)} \omega + df \wedge i_{\rho(a)}\omega
$$
and, by the properties of the Schouten bracket,
\begin{align*}
[f a, \mathfrak{X}] & =  -[\frakx, f a ] = - [\frakx, f] \wedge a - f [\frakx, a]\\
 & = - (-1)^{q-1} (i_{df} \frakx) \wedge a + f\, [a, \frakx]\\
 & = - a \wedge i_{df} \frakx + f\, [a, \frakx].
\end{align*}
Hence,
$(f a) \cdot (\omega \otimes \frakx)   =
f \,(a\cdot\tau) + df \wedge \left((i_{\rho(a)} \omega) \otimes \mathfrak{X}\right) -
a \wedge \left(\omega \otimes (i_{df}\mathfrak{X}) \right)$,
as we wanted.
\end{proof}


\subsection{The Lie algebroid of $\bG\toto \M$}

In this subsection, we discuss the Lie algebroid $\bA \to \M$ of the Lie groupoid $\bG\toto \M$ introduced in \eqref{big_group} and
describe a special set of generators for the $C^{\infty}(\M)$-module $\Gamma(\bA)$.

\subsubsection*{Prolongations of vector bundles}
Given a vector bundle $\pi_E:\E\Arrow M$, we may view it as a Lie
groupoid whose source and target maps are equal to $\pi_E$, the unit map
is the zero section $0: M \Arrow \E$, and the multiplication is
fiberwise addition. In this case, the tangent Lie groupoid is the
vector bundle $T\pi_E: T\E \Arrow TM$, called the \textit{tangent
prolongation} of $\E$. Similarly, the cotangent Lie groupoid is the
vector bundle $\widetilde{\pi}_E: T^*\E \Arrow \E^*$, called the
\textit{cotangent prolongation}. The projection $\widetilde{\pi}_E$ has
the following description: for $\xi \in T^*_e \E$, $\widetilde{\pi}_E(\xi)
\in \E^*_{\pi_E(e)}$ is the element defined by
$$
\langle \widetilde{\pi}_E(\xi), \dot{e} \rangle =
\langle \xi, \left.\frac{d}{d\epsilon}\right|_{\epsilon=0} (e+ \epsilon\,\dot{e}) \rangle, \;\;\;\;\; \forall\, \dot{e} \in \E_{\pi_E(e)}.
$$

The Lie algebroid $A\E$ is identified with $\E \Arrow M$ itself, with the zero
anchor and zero bracket; the right-invariant vector field
corresponding to $u \in \Gamma(\E)$ is the vertical lift $u^{\vl}
\in \frakx(\E)$, see \eqref{vert_lift_def}. The exact sequence
\eqref{s_seq} becomes
\begin{equation}\label{core_seq}
0 \longrightarrow \E \hookrightarrow 0^*T\E \stackrel{T\pi_E} \longrightarrow TM \longrightarrow 0.
\end{equation}
For $e \in \E_x$,
$$
\overline{e} = \left. \frac{d}{d\epsilon} \right|_{\epsilon=0}( \epsilon \, e) \,\in T_{0_x}\E
$$
defines the inclusion $\E \hookrightarrow 0^*T\E$ in the exact
sequence above. The translation bisection associated to $u \in
\Gamma(\E)$ is given by
\begin{equation}\label{trans_section}
\calb u(X) = T0(X) + \overline{u(x)}, \;\;\;\; X \in T_xM.
\end{equation}

In the cotangent prolongation, the exact sequence \eqref{s_dual_seq}
becomes the dual of \eqref{core_seq},
$$
0 \longrightarrow T^*M \stackrel{(T\pi_E)^*}\hookrightarrow 0^*T^*\E  \stackrel{\widetilde{\pi}_E} \longrightarrow \E^* \longrightarrow 0.
$$
For $\beta \in T_x^*M$, its image under the first map of
the exact sequence is
$$
\overline{\beta} = (T_{0_x}\pi_E)^*\beta  \in T_{0_x}^*\E.
$$
For a 1-form $\alpha \in \Omega^1(M)$, its translation bisection
$\calb\alpha: \E^* \Arrow T^*\E$ is given by
$$
\calb \alpha(\varphi) = \widetilde{0}_{\varphi} + \overline{\alpha(x)}, \,\, \varphi \in \E^*_x.
$$
Note that since $\t=\s=\pi_E$, $\calb u$ and $\calb \alpha$ are sections
of $T\pi_E: T\E \Arrow TM$ and $\widetilde{\pi}_E: T^*\E \Arrow \E^*$,
respectively. They are known as \textit{core sections} in the
general theory of double vector bundles (see e.g. \cite{Mac-book}).

\subsubsection*{Sections of prolongations}

For a section $u: M \Arrow \E$, its derivative $Tu: TM \Arrow T\E$
defines a section of the tangent prolongation. There is also an
induced section $\Rev_u: E^*\to T^*E$ of the cotangent prolongation,
as we now explain.

Recall the \textit{reversal isomorphism} $\Rev: T^*(\E^*) \Arrow
T^*\E$ (see e.g. \cite{Mac-book} for details): in local coordinates,
\begin{equation}\label{eq:reversal}
\Rev(\varphi, \beta, e) = (e, -\beta, \varphi),
\end{equation}
where we are locally writing $T^*(E^*) \cong E^* \oplus T^*M \oplus E$ and $T^*E \cong E \oplus T^*M \oplus E^*$.
Globally, $\Rev$ is both a vector bundle morphism from the cotangent bundle of $\E^*$, $T^*(\E^*) \Arrow
\E^*$,  to the cotangent prolongation of $\E$, $T^*\E \Arrow \E^*$,
and from the cotangent prolongation of $\E^*$, $T^*(\E^*) \Arrow
\E$, to the cotangent bundle of $\E$, $T^*\E \Arrow  \E$.
It fits into the following commutative diagram of vector bundle morphisms:
\begin{equation*}
\xy
(0,0)*+^{\E^*}="A"; (18,8)*+^{M}="M_1"; (30,0)*+^{\E^*}="\E^*";(48,8)*+^{M}="M_2";
(0,20)*+^{T^*(\E^*)}="T^*(\E^*)"; (18,28)*+^{\E}="\E_1"; (30,20)*+^{T^*\E}="T^*\E";(48,28)*+^{\E}="\E";
{\ar@{->}"A"; "M_1"};
{\ar@{->}_{\id_{\E^*}}"A"; "\E^*"};
{\ar@{->}^{\,\,\,\id_M}"M_1"; "M_2"};
{\ar@{->}"\E^*"; "M_2"};
{\ar@{->}"T^*(\E^*)"; "\E_1"};
{\ar@{->}^{\Rev}"T^*(\E^*)";"T^*\E"};
{\ar@{->}^{\id_\E}"\E_1"; "\E"};
{\ar@{->}"T^*\E"; "\E" };
{\ar@{->}"T^*(\E^*)"; "A"};
{\ar@{->}"\E_1";"M_1"};
{\ar@{->}"T^*\E"; "\E^*"};
{\ar@{->}"\E"; "M_2"};
\endxy
\end{equation*}

For $u \in \Gamma(\E)$ and $\mu \in \Gamma(\E^*)$, let $\ell_u \in
C^{\infty}(\E^*)$ and $\ell_{\mu}\in C^{\infty}(\E)$ be the
corresponding fiberwise linear functions. Then the composition
\begin{equation}\label{Ra_def}
\Rev_u := \Rev \circ d\ell_u: \E^* \Arrow T^*\E
\end{equation}
defines a section of the cotangent prolongation.
Note that the identities
\begin{equation}
\Rev_u(\mu(x)) = d\ell_{\mu}(u(x)) - (\pi_E^*\,d \langle \mu, u \rangle)(u(x)),
\end{equation}
and
\begin{equation}\label{R_hat}
\Rev(-\pi^*_{E^*}\alpha) = \calb \alpha,
\end{equation}
completely determine $\Rev$, where $\pi_{E^*}: E^* \to M$ is the projection of the dual bundle.

Let us fix positive integers $p,q$ and consider the Lie groupoid
\eqref{big_group}, with $\G=\E$. This is actually a vector bundle
$$
\bE^{(p,q)} \to \M^{(p,q)}.
$$
As before, we will simplify the notation by dropping the superindices.

For a section $u \in \Gamma(\E)$, we
denote by $(T^pu, \Rev_u^{\,q}): \mathbb{M} \Arrow \bE$ the section
given by
\begin{equation}\label{linear_pq}
(T^pu, \Rev_u^q)(\underline{\X},\underline{\varphi}) =
(Tu(X_1), \dots, Tu(X_p), \Rev_u(\varphi_1), \dots, \Rev_u(\varphi_q)),
\end{equation}
and by $\calb u_{(i)}: \M \Arrow \bE$ the section given by
\begin{equation}\label{a_i}
\calb u_{(i)}(\underline{\X},\underline{\varphi}) =
(T0(X_1), \dots, \calb u(X_i), \dots, T0(X_p), \widetilde{0}_{\varphi_1}, \dots, \widetilde{0}_{\varphi_q}),
\end{equation}
for $i=1, \dots, p$. Similarly, for $\alpha \in \Omega^1(M)$, denote by $\calb\alpha_{(j)}: \M \Arrow \bE$ the section defined by
\begin{equation}\label{alpha_j}
\calb\alpha_{(j)}(\underline{\X},\underline{\varphi}) =
(T0(X_1), \dots, T0(X_p), \widetilde{0}_{\varphi_1},  \dots, \calb\alpha(\varphi_j), \dots, \widetilde{0}_{\varphi_q}),
\end{equation}
for $j =1, \dots, q$. The following result is proven in
\cite{Mac-doubles}.

\begin{proposition}\label{prop:generators}
The $C^{\infty}(\M)$-module of section $\Gamma(\bE)$ is
generated by $(T^pu, \Rev_u^q)$, $\calb v_{(i)}$ and $\calb
\alpha_{(j)}$, for $u, v \in \Gamma(\E)$ and $\alpha \in
\Omega^1(M)$, $i=1, \dots, p$, $j=1, \dots, q$.
\end{proposition}

\begin{remark}\label{rem:Clinear}[$C^\infty(M)$-linearity]
One can check that
\begin{equation}\label{leibniz:tang}
T(fu) = (f\circ \pi_E) \cdot Tu\, +_{\scriptscriptstyle{\mathfrak{p}}}
\,\ell_{df} \cdot \calb u,
\end{equation}
where we used the notation $+_{\scriptscriptstyle{\mathfrak{p}}}$
and $\,\cdot\,$ for the sum and scalar multiplication on the fibers
of the tangent prolongation $T\E \Arrow TM$, respectively, and
$\ell_{df} \in C^{\infty}(TM)$ is the linear function corresponding
to $df \in \Omega^1(M)$.
Similarly,
\begin{equation}\label{leibniz:cotang}
\Rev_{fu} = (f\circ \pi_{E^*})\cdot \Rev_u
+_{\scriptscriptstyle{\mathfrak{p}}}
 \,\,\ell_{-u} \cdot \calb (df),
\end{equation}
where $+_{\scriptscriptstyle{\mathfrak{p}}}$ and $\,\cdot\,$ denote
the sum and scalar multiplication on the fibers of the cotangent
prolongation $T^*\E \Arrow \E^*$, respectively.
\end{remark}

\subsubsection*{Brackets and anchors.} \label{lie_alg}
We now recall the main features of the tangent
and cotangent Lie algebroids, see e.g. \cite{Mac-Xu} for details.

For a Lie algebroid $(A, [\cdot, \cdot], \rho)$, consider its tangent prolongation $TA \to TM$.
It has the structure of a Lie algebroid, where the Lie bracket $[\cdot,
\cdot]$ on the space of sections of $TA\to TM$ is determined by the conditions
$$
\left[Ta, Tb\right] = T\left[a,b\right], \;\;\; [Ta, \calb b] = \calb[a,b],\;\;\;
[\calb a, \calb b] = 0,
$$
for $a, \, b \in \Gamma(A)$, while the anchor $\rho_T: TA
\Arrow T(TM)$ is determined by
\begin{align*}
\rho_T(Ta) & = \rho(a)^{T}, \,\,\,
\rho_T(\calb a) = \rho(a)^{\vl}.
\end{align*}
where $(\cdot)^T$ and $(\cdot)^{\vl}$ are the tangent and vertical
lifts, respectively.

The cotangent Lie algebroid is the Lie algebroid structure on $T^*A \to A^*$ defined as follows: the anchor
$\rho_{T^*}: T^*(A) \Arrow T(A^*)$ is determined by
\begin{align*}
\rho_{T^*}(\Rev_a)& = H_a, \,\,\,
\rho_{T^*}(\calb \alpha) = (\rho^*\alpha)^{\vl},
\end{align*}
where $H_a \in \frakx(A^*)$ is the Hamiltonian lift  of $a\in
\Gamma(A)$, see \eqref{H:def}. The Lie bracket on the space of sections of
$T^*A\to A^*$ is determined by
$$
[\Rev_a, \Rev_b]=\Rev_{[a,b]}, \;\;\; [\Rev_a, \calb \alpha] = \calb
\left(\Lie_{\rho(a)}\alpha\right), \;\;\; [\calb \alpha,
\calb \beta] =0,
$$
where $a, b \in \Gamma(A)$ and $\alpha, \beta \in \Omega^1(M)$.

The Whitney sum
$$
\bA^{p,q} = (\oplus^p TA) \oplus (\oplus^q T^*A) \Arrow
(\oplus^p TM) \oplus (\oplus^q A^*)
$$
inherits a Lie algebroid structure which is determined componentwise by the
tangent and the cotangent Lie algebroids. We give a detailed description here for convenience.
The anchor map is determined by
\begin{align}
\nonumber \rho_{\bT}(T^pa, \Rev_a^q) & = (\rho(a)^{T,p}, H_a^q)\\
\label{big_anchor}\rho_\bT(\calb a_{(i)}) & = (\rho(a)^{\vl, \,p}_{(i)}, 0)\\
\nonumber \rho_\bT(\calb \alpha_{(j)}) & = (0, \rho^*\alpha^{\vl, \,q}_{(j)}),
\end{align}
for $\alpha \in \Omega^1(M)$.
The Lie bracket on the space of sections of $\bA\to M$ is determined by what it does on generators according to the following formulas:
\begin{align}
\nonumber [(T^pa, \Rev_a^q), (T^pb, \Rev_b^q)] & = (T^p[a,b], \Rev_{[a,b]}^q), \,\, [(T^pa, \Rev_a^q), \calb b_{(i)}] =
\calb [a,b]_{(i)}\\
\label{big_bracket}[(T^pa, \Rev_a^q), \calb \alpha_{(j)}] & =
\calb (\Lie_{\rho(a)}\alpha)_{(j)}, \,\, [\calb a_{(i)}, \calb b_{(i')}] = 0\\
\nonumber [\calb a_{(i)}, \calb \alpha_{(j)}] & = 0, \,\, [\calb \alpha_{(j)}, \calb \beta_{(j')}] = 0.
\end{align}

For a Lie groupoid $\G \toto M$, let $A \to M$ be its Lie algebroid. There are natural Lie-algebroid identifications $A(T\G) \cong TA$
as well as $A(T^*\G)\cong T^*A$ (see \cite{Mac-Xu}). Hence, the Lie algebroid $A(\bG^{(p,q)}) \Arrow \mathbb{M}^{(p,q)}$ of the
Lie groupoid $\bG^{(p,q)} \toto \M^{(p,q)}$ is naturally isomorphic to
\begin{equation}\label{big_alg}
\bA^{p,q} =(\oplus^p TA) \oplus (\oplus^q T^*A) \Arrow
(\oplus^p TM) \oplus (\oplus^q A^*).
\end{equation}
The right-invariant vector fields on $T\G$ corresponding to the sections of type $Ta$, $\calb a$
are given by
\begin{equation}\label{right:tang}
\overrightarrow{Ta} = \overrightarrow{a}^T \,\,\, \text{ and } \,\,\, \overrightarrow{\calb a} = \overrightarrow{a}^{\vl},
\end{equation}
see e.g. \cite{Mac-Xu} for a proof. Similarly, the right-invariant vector fields on
$T^*\G$ corresponding to the sections of type $\Rev_a$, $\calb \alpha$ are
\begin{equation}\label{right:cotang}
\overrightarrow{\Rev_a} =  \overrightarrow{a}^{T^*} \,\, \text{ and
} \,\, \overrightarrow{\calb\alpha} = (\t^*\alpha)^{\vl}.
\end{equation}
The proof of these last formulas can be found
in Appendix \ref{cot_appendix}.

It is now a straightforward consequence of \eqref{right:tang} and
\eqref{right:cotang} that the right-invariant vector fields
$\overrightarrow{(T^pa, \Rev_a^q)}$, $\overrightarrow{\calb
a_{(i)}}$ and $\overrightarrow{\calb \alpha_{(j)}} \in \frakx(\bG^{p,q})$
are given by
\begin{align}
\label{right:linear}
\overrightarrow{(T^pa, \Rev_a^q)} & = (\overrightarrow{a}^{T, \,p}, \overrightarrow{a}^{T^*, \,q})\\
\label{right:core}
\overrightarrow{\calb a_{(i)}} & = (\overrightarrow{a}^{\vl}_{(i)},0)\\
\label{right:core2}
\overrightarrow{\calb \alpha_{(j)}} & =  (0,(\t^*\alpha)^{\vl}_{(j)}),
\end{align}
for $i=1, \dots, p$, $j=1, \dots, q$.

\subsection{Proof of Theorem \ref{thm:mult_char}}

Let us begin with two important lemmas.
For the first one, we need to introduce some notation. Define
$\pi^{(p,q)}_{(i,0)}: \bG^{(p,q)} \Arrow \bG^{(p-1, q)}$,
$\pi^{(p,q)}_{(0,j)}: \bG^{(p,q)} \Arrow \bG^{(p, q-1)}$,
as the natural projections
\begin{align}
\label{projpi:i}\pi^{(p,q)}_{(i,0)}(\underline{\U}, \underline{\xi}) & = (U_1, \dots, U_{i-1}, U_{i+1}, \dots, U_p, \xi_1, \dots, \xi_q)\\
\label{projpi:j}\pi^{(p,q)}_{(0,j)}(\underline{\U}, \underline{\xi}) & = (U_1, \dots, U_p, \xi_1, \dots, \xi_{j-1}, \xi_{j+1}, \dots, \xi_q)
\end{align}
for $1 \leq i \leq p$, $1 \leq j \leq q$. When there is no risk of
confusion, we will omit the superscripts $(q,p)$ on the projections.
Observe that $ \pi_{(i,0)}$ and $ \pi_{(0,j)}$ are groupoid morphisms (covering
\eqref{proj:i} and \eqref{proj:j}, respectively).

\begin{lemma} \label{prop:lr_existence}
For any multiplicative tensor field $\tau \in \Gamma(\wedge^p T^*\G
\otimes \wedge^q T\G)$, there exist vector bundle maps $l : A \Arrow  \wedge^{p-1}
T^*M \otimes \wedge^q A$ and $r  : T^*M  \Arrow  \wedge^p T^*M
\otimes \wedge^{q-1} A$ covering the identity map such that
\begin{align}
\label{l:def} i_{\overrightarrow{a}} \,\tau & = \T(l(a)),\\
\label{r:def} i_{\t^*\alpha} \,\tau & = \T(r(\alpha)).
\end{align}
\end{lemma}

\begin{proof}
By formulas \eqref{func:eq} and \eqref{right:core}, we have that
$$
\bt^*\<A c_{\tau}, \, \calb a_{(i)}\>  = \Lie_{\overrightarrow{\calb a_{(i)}}} \,c_{\tau} =
\Lie_{(\overrightarrow{a}^{\vl}_{(i)}, 0)} \, c_{\tau} = (-1)^{i-1} c_{\,i_{\overrightarrow{a}} \tau} \circ \pi_{(i,0)}
$$
for $a \in \Gamma(A)$, where the last equality is a consequence of Prop.~\ref{Lie_tensor}. Similarly, for
$\alpha\in \Omega^1(M)$, we check that
$$
\bt^*\<A c_{\tau}, \,\calb \alpha_{(j)}\>  = (-1)^{j-1}c_{i_{\t^*\alpha} \, \tau} \circ \pi_{(0,j)}.
$$

Note that, for $(\underline{\X}, \underline{\varphi}) \in \M$,
$$
\<A c_{\tau}, \calb a_{(i)}\>(\underline{\X}, \underline{\varphi}) =
(-1)^{i-1} \tau(\overrightarrow{a}, X_1, \dots, X_{i-1}, X_{i+1}, \dots,  X_p, \varphi_1, \dots, \varphi_q),
$$
which shows that $\<A c_{\tau},  \calb a_{(i)}\>$ is a
componentwise linear function of $\gamma_{(i,0)}(\underline{\X},
\underline{\varphi})$. Hence (see Lemma~\ref{lemma:comp_tensor}) there exists  $l(a)
\in \Gamma(\wedge^{p-1}T^*M\otimes \wedge^q A)$ such that
\begin{equation}\label{l:comp}
\<Ac_{\tau}, \calb a_{(i)}\> = (-1)^{i-1} c_{l(a)} \circ \gamma_{(i,0)}.
\end{equation}
Now, note that
\begin{align*}
c_{i_{\overrightarrow{a}}\tau} \circ \pi_{(i,0)} & = (-1)^{i-1}\bt^*\<Ac_{\tau}, \,\calb a_{(i)}\> =
\bt^* \left(c_{l(a)}\circ \gamma_{(i,0)}\right)\\
& = \left(\bt^*c_{l(a)} \right) \circ \pi_{(i,0)} = c_{\T(l(a))} \circ \pi_{(i,0)}.
\end{align*}
Formula \eqref{l:def} follows from the injectivity of the correspondence between tensors and
componentwise linear functions (Lemma~\ref{lemma:comp_tensor}).

To prove that $l$ is $\C(M)$-linear, we use Proposition \ref{pull:invariance} to see that
$$
\T(l(fa)) = i_{\overrightarrow{fa}} \tau = (\t^*f) \, i_{\overrightarrow{a}} \tau = (\t^*f) \,\T(l(a)) = \T(f l(a)),
$$
for $f \in \C(M)$, so that $\C(M)$-linearity follows from the
injectivity of $\T$.

Similarly, we can prove the existence of $r: T^*M  \Arrow  \wedge^p T^*M
\otimes \wedge^{q-1} A$ by checking that $\<A c_{\tau}, \,\calb \alpha_{(j)}\>$
is componentwise linear, so that \eqref{r:def} follows from
\begin{equation}\label{r:comp}
\<Ac_{\tau}, \calb \alpha_{(j)}\>  =(-1)^{j-1} c_{r(\alpha)}\circ \gamma_{(0,j)}.
\end{equation}
\end{proof}

\begin{lemma}\label{prop:D_existence}
For any multiplicative tensor $\tau \in \Gamma(\wedge^p T^*\G \otimes \wedge^q T\G)$,
there exists an $\R$-linear map $D: \Gamma(A) \Arrow \Gamma(\atensor{p}{q})$ satisfying the
Leibniz condition \eqref{D:leibniz} and such that
\begin{equation}\label{D:def}
\Lie_{\overrightarrow{a}} \tau  = \T(D(a)).
\end{equation}
\end{lemma}

\begin{proof}
From \eqref{func:eq} and \eqref{right:linear}, we see that
$$
\bt^* \<A c_{\tau},\,(T^pa, \Rev_a^q)\> = \Lie_{\overrightarrow{(T^pa, \Rev_a^q)}} \,c_{\tau}
= \Lie_{(\overrightarrow{a}^{T, p}, \overrightarrow{a}^{T^*, q})} \, c_{\tau}
= c_{\Lie_{\overrightarrow{a}}\tau},
$$
where the last equality relies on Proposition \ref{Lie_tensor}.
In particular, for $(\underline{\X}, \underline{\varphi}) \in \M$,
$$
\<A c_{\tau},\,(T^pa, \Rev_a^q)\>(\underline{\X}, \underline{\varphi}) =
\Lie_{\overrightarrow{a}}\tau(\underline{\X}, \underline{\varphi}),
$$
which proves that $\<A c_{\tau},\,(T^pa, \Rev_a^q)\> \in C^{\infty}(\M)$ is a componentwise linear function.
By Lemma~\ref{lemma:comp_tensor}, there exists $D(a) \in \Gamma(\tensor{p}{q})$ such that
\begin{equation}\label{comp_D}
\<A c_{\tau}, (T^pa, \Rev_a^q)\> = c_{D(a)}.
\end{equation}
Now \eqref{D:def} follows from Proposition \ref{pull:invariance}
(and the injectivity part of Lemma~\ref{lemma:comp_tensor}):
$$
c_{\Lie_{\overrightarrow{a}}\tau} =\bt^*\<A c_{\tau},\,(T^pa, \Rev_a^q)\> = \bt^*c_{D(a)} = c_{\T(D(a))}.
$$
To prove the Leibniz condition \eqref{D:leibniz},  we use Proposition~\ref{Lie:leibniz} to see that
\begin{align*}
\T(D(fa)) & = \Lie_{\overrightarrow{fa}} \, \tau  = \Lie_{(\t^*f) \, \overrightarrow{a}} \,
\tau = (\t^*f) \,\Lie_{\overrightarrow{a}} \, \tau + \t^*df \wedge i_{\overrightarrow{a}} \,
\tau - \overrightarrow{a} \wedge i_{\t^*df}\,\tau\\
& = (\t^*f) \T(D(a)) + \t^*df \wedge \T(l(a)) - \overrightarrow{a} \wedge \T(r(df))\\
& = \T( f D(a) + df \wedge l(a) - a \wedge r(df)).
\end{align*}
The conclusion follows from the injectivity of $\T$.
\end{proof}

We are now in position to present the proof of Theorem \ref{thm:mult_char}.

\begin{proof}[Proof of Theorem \ref{thm:mult_char}]
If $\tau$ is multiplicative, then the existence of the triple $(D,l,r)$ is
guaranteed by Lemmas~\ref{prop:lr_existence} and
\ref{prop:D_existence}; condition \eqref{eq:vanish} follows from the fact that, since $c_\tau$ is a multiplicative function,
it must vanish along groupoid units (see Prop.~\ref{func:prop}).

Conversely, assume the existence of $(D,l,r)$. We claim that there is a unique $\mu \in \Gamma(A^*\bG)$ defined by
the following conditions:
\begin{align}
&\<\mu, \calb a_{(i)}\> = (-1)^{i-1} c_{l(a)} \circ \gamma_{(i,0)},\label{eq:mu1}\\
&\<\mu, \calb \alpha_{(j)}\> = (-1)^{j-1} c_{r(\alpha)} \circ \gamma_{(0,j)},\label{eq:mu2}\\
&\<\mu, (T^pa, \Rev_a^q)\> = c_{D(a)},\label{eq:mu3}
\end{align}
for $a \in \Gamma(A)$ and $\alpha \in \Omega^1(M)$. Uniqueness follows from Prop.~\ref{prop:generators}, so it remains to verify that $\mu$ is indeed well-defined. For local frames $(a^k)_{k=1,\ldots,\mathrm{rank}(A)}$ of $A$ and $(\alpha^s)_{s=1,\ldots,\dim(M)}$ of $T^*M$, we observe that the collection of local sections $(T^pa^k, \Rev^q_{a^k})$, $\calb a^k_{(i)}$, and $\calb \alpha^s_{(j)}$ form a local frame for $\bA$. We first define $\mu$ on this local frame using the formulas above (and extend it by linearity). In the following, we use Einstein notation. To show that $\mu$ is globally well defined, it suffices to verify that \eqref{eq:mu1}, \eqref{eq:mu2}, and \eqref{eq:mu3} hold for $a= f_k a^k$ and $\alpha= g_s \alpha^s$, where $f_k$, $g_s\in C^\infty(M)$. One can check that
$$
\calb a_{(i)} =  (f_k\circ \pr) \, \calb a_{(i)}^k, \;\;\;
\calb \alpha_{(j)} =  (g_s\circ \pr)\, \calb \alpha_{(j)}^s
$$
where $\pr: \M \to M$ is the bundle projection. Using \eqref{leibniz:tang} and \eqref{leibniz:cotang},
one also verifies that (see Remark~\ref{rem:Clinear} for notation)
$$
(T^pa,\Rev_a^q)= (f_k\circ \pr) (T^pa^k,\Rev_{a^k}^q)
+_{\scriptscriptstyle{\mathfrak{p}}} (\ell_{df_k}\circ \pr_{TM}^i) \calb a^k_{(i)} +_{\scriptscriptstyle{\mathfrak{p}}}
 (\ell_{-a_k}\circ \pr^j_{A^*})\calb \beta^k_{(j)},
$$
where $\beta^k= df_k$, and $\pr_{TM}^i: \M\to TM$, $\pr_{A^*}^j:\M\to A^*$ are given by
$\pr_{TM}^i(\underline{\X},\underline{\varphi})=X_i$, $\pr_{A^*}^j(\underline{\X},\underline{\varphi})=\varphi_j$.
Using Lemma~\ref{lemma:comp_tensor}, we see that
$$
\<\mu, \calb a_{(i)}\> = (f_k\circ \pr) \<\mu,  \calb a_{(i)}^k\> = (-1)^{i-1} (f_k\circ \pr) c_{l(a^k)}\circ \gamma_{(i,0)} = (-1)^{i-1} c_{l(a)}\circ \gamma_{(i,0)},
$$
which is \eqref{eq:mu1}. A similar argument verifies \eqref{eq:mu2}. For \eqref{eq:mu3}, note that
\begin{align*}
\<\mu, (T^pa, \Rev_a^q) \>  = & (f_k\circ \pr)c_{D(a^k)} +
(-1)^{i-1}(\ell_{df_k}\circ \pr_{TM}^i) c_{l(a^k)}\circ \gamma_{(i,0)} \\& +
(-1)^{j-1}(\ell_{-a^k}\circ \pr_{A^*}^j)c_{r(df_k)}\circ \gamma_{(0,j)}.
\end{align*}
By \eqref{skew_comp}, it follows that the right-hand side above agrees with
$$
c_{f_kD(a^k)} + c_{df_k \wedge l(a^k)}- c_{a^k\wedge r(df_k)} = c_{D(f_ka^k)},
$$
where the Leibniz condition for $D$ is used in the last equality. Hence \eqref{eq:mu3} holds and $\mu$ is well defined.

Now consider the componentwise linear function $c_{\tau} \in \C(\bG)$. It follows from the
third equality in \eqref{eq:conds}, Lemma~\ref{pull:invariance} and Theorem~\ref{Lie_tensor} that
\begin{align*}
\Lie_{\overrightarrow{(T^pa, \Rev_{a}^q)}} \, c_{\tau}  = c_{\Lie_{\overrightarrow{a}}\tau}
 = c_{\T(D(a))} = \bt^* c_{D(a)}  = \bt^* \<\mu, (T^pa, \Rev_{a}^q)\>.
\end{align*}
Similarly, we see that
$$
\Lie_{\overrightarrow{\calb b_{(i)}}} c_{\tau} = \bt^{*}\<\mu, \calb b_{(i)}\>,\;\;\mbox{ and }\;
\Lie_{\overrightarrow{\calb \alpha_{(j)}}} c_{\tau} = \bt^{*}\<\mu, \calb \alpha_{(j)}\>.
$$
By linearity, one has that $\Lie_{\overrightarrow{\chi}} c_{\tau} = \bt^*\<\mu, \chi\>$,
for every $\chi \in \Gamma(A\bG)$. As $\bG$ is source connected (because $\G$ is, see Remark~\ref{rem:fibers}),
the result follows from Proposition \ref{func:prop}.
\end{proof}

\subsection{Linear tensor fields} \label{linear}

For the proof of Theorem~\ref{thm:main}, we will need to specialize our study of multiplicative tensor fields
to Lie groupoids given by vector bundles $\pi_E: \E \Arrow M$. As we already saw, in this case the groupoid
$\mathbb{E} = (\oplus^p T\E) \oplus (\oplus^q T^*\E)$ in  \eqref{big_group} is a vector bundle over
$\M= (\oplus^p TM) \oplus (\oplus^q \E^*)$, with
Lie algebroid $A \bE$ given by $\bE \Arrow \M$ itself with zero anchor and zero bracket.
A tensor field $\tau \in \Gamma(\wedge^p T^*\E \otimes \wedge^q T\E)$
is multiplicative if and only if the associated function $c_\tau: \mathbb{E}\to \mathbb{R}$ is fiberwise linear
on $\mathbb{E} \Arrow \M$. For this reason, we refer to multiplicative tensor fields on $E\to M$ also as {\em linear}.
Our goal here is to show how one can reconstruct linear tensor fields on $\E\to M$ explicitly from their infinitesimal components.

For a linear tensor $\tau \in \Gamma(\wedge^p T^*\E \otimes \wedge^q T\E)$,
$$
Ac_{\tau} = c_{\tau},
$$
noticing that since $c_\tau$ is fiberwise linear on $\bE \to \M$, it can be seen as a section of $\bE^*\to \M$,
cf. Example \ref{linear_func}.
In particular, from \eqref{eq:mu1}, \eqref{eq:mu2} and \eqref{eq:mu3}, we see that the infinitesimal components
$(D,l,r)$ of a linear tensor $\tau$ satisfy
\begin{equation}\label{l_linear}
c_{l(u)} \circ \gamma_{(i,0)}  = (-1)^{i-1}\<c_{\tau}, \calb u_{(i)}\>,\;\;\;
c_{r(\alpha)} \circ \gamma_{(0,j)}  = (-1)^{j-1}\<c_{\tau}, \calb \alpha_{(j)}\>,
\end{equation}
\begin{equation}\label{D_linear}
c_{D(u)}  = \<c_{\tau}, (T^pu, \Rev_{u}^q)\>,
\end{equation}
for $u \in \Gamma(\E)$, $\alpha \in \Omega^1(M)$, and where $\gamma_{(i,0)}$, $\gamma_{(o,j)}$ are the projections
in \eqref{proj:i} and \eqref{proj:j}.

Before presenting the main result of this subsection, we need two lemmas.
We begin with a useful property of multiplicative tensors on Lie groupoids with $\s=\t$, so in particular linear ones.

\begin{lemma}\label{lemma:s=t}
Let $\G \toto M$ be a Lie groupoid such that $\s=\t$. If $\tau \in \Gamma(\wedge^p T^*\G \otimes \wedge^q T\G)$ is a multiplicative $(q,p)$-tensor field, then
\begin{align}
& i_{\overrightarrow{a}}\,i_{\overrightarrow{b}}\,\tau = 0,\\
& i_{\t^*\alpha} \,i_{\t^*\beta}\,\tau = 0,\\
& i_{\overrightarrow{a}}\,i_{\t^*\alpha}\,\tau = 0,
\end{align}
for $a, b \in \Gamma(A)$ and $\alpha, \beta \in \Gamma(T^*M)$.
\end{lemma}

\begin{proof}
As $\s=\t$, it follows from \eqref{T:def} that $i_{\overrightarrow{a}}\T(\Phi) = 0$ and $i_{\t^*\alpha} \T(\Phi) = 0$, for any $\Phi \in \Gamma(\wedge^{\bullet} T^*M \otimes \wedge^{\bullet} A)$. The result now follows from Theorem \ref{thm:mult_char}.
\end{proof}

The following lemma shows how the infinitesimal components of linear tensors can be obtained by means of
pointwise evaluation of the tensor on special vectors and covectors .

\begin{lemma}\label{lemma:linear_comp}
Let $\tau \in \Gamma(\wedge^p T^*\E \otimes \wedge^q T\E)$ be a linear $(q,p)$ tensor field.
For $(\underline{\X}, \underline{\varphi}) \in \M_x$, $u \in \Gamma(E)$ and $\alpha \in \Omega^1(M)$, define
\begin{align*}
\underline{U}_{(i)} & = (T0(X_1), \dots , T0(X_{i-1}), \overline{u(x)}, T0(X_{i+1}), \dots, T0(X_p)) \in \,\,\oplus^p T_{0_x}E\\
\underline{\xi}_{(j)} & =
(\widetilde{0}_{\varphi_1}, \dots, \widetilde{0}_{\varphi_{j-1}}, \overline{\alpha(m)},
\,\,\widetilde{0}_{\varphi_{j+1}}, \dots, \widetilde{0}_{\varphi_{q}}) \in \,\,\oplus^q T^*_{0_x}E.
\end{align*}
The infinitesimal components
$l:\E \Arrow \wedge^{p-1} T^*M\otimes \wedge^{q} \E$, $r:T^*M \Arrow \wedge^{p} T^*M\otimes \wedge^{q-1} \E$
and $D: \Gamma(\E) \Arrow \Gamma(\wedge^p T^*M\otimes \wedge^q \E)$ satisfy
\begin{itemize}
\item[(a)] $D(u)(\underline{\X},\underline{\varphi})  = \tau((Tu^p, \Rev^q)(\underline{\X}, \underline{\varphi}))$
\item[(b)] $l(u)(\gamma_{(i,0)}(\underline{\X},\underline{\varphi})) =
(-1)^{i-1}\tau(\underline{U}_{(i)}, \widetilde{0}_{\varphi_1}, \dots, \widetilde{0}_{\varphi_q})$ 
\item[(c)] $r(\alpha)(\gamma_{(0,j)}(\underline{\X},\underline{\varphi}))
= (-1)^{j-1}\tau(T0(X_1), \dots, T0(X_p), \underline{\xi}_{(j)}).$
\end{itemize}
\end{lemma}

\begin{proof}
From \eqref{D_linear}, it is clear that
\begin{align*}
D(u)(\underline{\X},\underline{\varphi}) & = \<c_{\tau}, (T^pu, \Rev_u^q)\>|_{(\underline{\X}, \underline{\varphi})} \\
& = \tau(Tu(X_1), \dots, Tu(X_p), \Rev_u(\varphi_1), \dots, \Rev_u(\varphi_q)),
\end{align*}
which proves (a).
Similarly, from the first equation in \eqref{l_linear},
\begin{align*}
l(u)(\gamma_{(i,0)}(\underline{\X},\underline{\varphi})) & = \<c_{\tau}, \calb u_{(i)}\>|_{(\underline{\X}, \underline{\varphi})}\\
 & = (-1)^{i-1}\tau(T0(X_1), \dots, \calb u(X_i), \dots, T0(X_p), \widetilde{0}_{\varphi_1}, \dots, \widetilde{0}_{\varphi_q}).
\end{align*}
But since $\calb u(X_i) = T0(X_i)+\overline{u(m)}$, we have that this last term equals
\begin{align*}
&\tau(T0(X_1), \dots, T0(X_p), \widetilde{0}_{\varphi_1}, \dots, \widetilde{0}_{\varphi_q})
+ \\
&\tau(T0(X_1), \dots, T0(X_{i-1}), \overline{u(m)}, T0(X_{i+1}), \dots, T0(X_p), \widetilde{0}_{\varphi_{1}},\dots, \widetilde{0}_{\varphi_{q}}).
\end{align*}
To conclude that (b) holds, note that
$\tau(T0(X_1), \dots, T0(X_p), \widetilde{0}_{\varphi_1}, \dots, \widetilde{0}_{\varphi_q}) = 0$, as $\tau$ is linear on $\bE \Arrow \M$.
The verification of (c) is similar.
\end{proof}

We can now present the main result regarding linear tensor fields.

\begin{proposition}\label{prop:linear1}
A tensor $\tau \in \Gamma(\wedge^p T^*\E \otimes \wedge^q T\E)$ is linear if and only if there exist
vector-bundle maps $l:\E \Arrow \wedge^{p-1} T^*M\otimes \wedge^{q} \E$ and $r:T^*M \Arrow \wedge^{p} T^*M\otimes \wedge^{q-1} \E$
covering the identity, and $D: \Gamma(\E) \Arrow \Gamma(\wedge^p T^*M\otimes \wedge^q \E)$ satisfying the Leibniz
condition \eqref{D:leibniz}, such that, for $U_1, \dots, U_p \in T_e \E$, $\xi_1, \dots, \xi_q \in T_e^*\E$, $e \in \E_x$,
\begin{align}\label{linear:eq}
\tau(\underline{\U}, \underline{\xi})  = &  D(u)(\underline{\X}, \underline{\varphi}) +
(-1)^{i-1}l(e_i)(\gamma_{(i,0)}(\underline{\X}, \underline{\varphi}))\\
\nonumber & + (-1)^{j-1}r(\beta_j)(\gamma_{(0,j)}(\underline{\X}, \underline{\varphi})),
\end{align}
where $u \in \Gamma(\E)$ is any section such that $u(x) = e$,
$X_i=T\pi_E(U_i)$, $\varphi_j= \widetilde{\pi_E}(\xi_j)$, and $e_i \in \E_x$, $\beta_j \in T_x^*M$ are defined by
$$
\<\psi, e_i\> = \<U_i -_{\scriptscriptstyle{\mathfrak{p}}} Tu(X_i), \widetilde{0}_{\psi}\>, \;\;\;\;
\<\beta_j, Y\> =\<\xi_j -_{\scriptscriptstyle{\mathfrak{p}}} \Rev_u(\varphi_j), T0(Y)\>,
$$
for $\psi \in \E^*_x$, $Y \in T_xM$, and $i=1, \dots, p$, $j=1, \dots, q$.
In this case, $(D,l,r)$ are the infinitesimal components of $\tau$.
\end{proposition}

\begin{proof}
First note that the right-hand side of formula \eqref{linear:eq} is well defined,
in the sense that it does not depend on the extension $u$. Indeed, let
$(u^k)_{k=1, \dots, \mathrm{rank}(E)}$, $(\alpha^s)_{s=1, \dots, \dim(M)}$ be local frames of $E$ and $T^*M$ respectively. One can write (see notation in Remark~\ref{rem:Clinear})
$$
U_i = t_{k} \cdot T u^k(X_i) +_{\scriptscriptstyle{\mathfrak{p}}} \,\, h_{ik} \cdot \calb u^k(X_i), \,\,\,\,\,
\xi_j  = t_{k} \cdot \Rev_{u^k}(\varphi_j) +_{\scriptscriptstyle{\mathfrak{p}}} \,\,g_{js} \cdot \calb \alpha^s(\varphi_j),
$$
for $t_{k}, h_{ik}, g_{js} \in \R$, where $e = t_k u^k(x)$. For any section $u = f_k u^k$, $f^k \in C^{\infty}(M)$, $u(x)=e$ if and only if $f_k(x) = t_k$. Also,
$$
e_i = (h_{ik} - \Lie_{X_i} f_k) u^k, \,\,\,\,\,\,\, \beta_j = g_{js}\alpha^s + \langle \varphi_j, u^k \rangle df_k.
$$
So, by using the Leibniz condition \eqref{D:leibniz}, one
can rewrite \eqref{linear:eq} as
\begin{align*}
\tau(\underline{\U}, \underline{\xi})  = & t_k \,D(u^k)(\underline{\X}, \underline{\varphi}) +  \sum_{i=1}^ {p}(-1)^{i-1}h_{ik}\, l(u^k)(\gamma_{(i,0)}(\underline{\X}, \underline{\varphi}))\\
& +  \sum_{j=1}^q(-1)^{j-1}g_{js} \,r(\alpha^s)(\gamma_{(0,j)}(\underline{\X}, \underline{\varphi})).
\end{align*}
Let us assume that $\tau \in \Gamma(\wedge^p T^*\E \otimes \wedge^q T\E)$ is a linear tensor, and let $(D, l, r)$ be its
infinitesimal components. One may directly check that
$$
U_i = Tu(X_i) +_{\scriptscriptstyle{\mathfrak{p}}} \underbrace{(T0(X_i) + \overline{e_i})}_{V_i} \,\,\,
\text{ and }  \,\,\, \xi_j = \Rev_u(\varphi_j) +_
{\scriptscriptstyle{\mathfrak{p}}} \underbrace{(\widetilde{0}_{\varphi_j} + \overline{\beta_j})}_{\zeta_j}.
 $$
Since $\tau$ is linear, we have
$$
\tau(\underline{\U}, \underline{\xi})= \tau(T^{p}u(\underline{\X}), \Rev^q(\underline{\varphi})) +
\tau(\underline{V\vphantom{\zeta}}, \underline{\zeta}) = D(u)(\underline{\X}, \underline{\varphi}) +
\tau(\underline{V\vphantom{\zeta}}, \underline{\zeta}).
$$
Now, using the multilinearity of the tensor $\tau$, one can expand $\tau(\underline{V\vphantom{\zeta}}, \underline{\zeta})$
as a sum in which every term is $\tau$ evaluated on a string involving $T0(X_i), \overline{e_i}$ separatedly on the
$T\E$ part and $\widetilde{0}_{\varphi_j}, \overline{\beta_j}$ separatedly on the $T^*\E$ part.
\medskip

\paragraph{\bf Claim:} The only non-zero terms on the expansion of $\tau(\underline{V\vphantom{\zeta}}, \underline{\zeta})$
as a sum are the ones in which the $\overline{(\cdot)}$ terms appear exactly once
(counting both the $T\E$ and $T^*\E$ parts). Indeed, if they do not appear at all, one has
$
\tau(T0(X_1), \dots, T0(X_p), \widetilde{0}_{\varphi_1}, \dots, \widetilde{0}_{\varphi_q})=0,
$
because $\tau$ is linear on the fibers of $\bE \Arrow \M$. If they appear twice or more, note that
$\overline{e_i} = u_i^{\vl}(0_m)$ and $\overline{\beta_j} = (\pi_E^*\alpha_j)(0_m)$,
where $u_i \in \Gamma(\E)$ and $\alpha_j \in \Gamma(T^*M)$ satisfy $u_i(m)=e_i$ and
$\alpha_j(m) = \beta_j$. So, the claim  follows from Lemma \ref{lemma:s=t}.
Therefore,
\begin{align*}
\tau(\underline{\U}, \underline{\xi}) =  D(u)(\underline{\X}, \underline{\varphi})\\
& \hspace{-80pt}+ \sum_{i=0}^p \tau(T0(X_1), \dots, T0(X_{i-1}), \overline{e_i}, T0(X_{i+1}), \dots, T0(X_p), \widetilde{0}_{\varphi_{1}},\dots, \widetilde{0}_{\varphi_{q}})\\
& \hspace{-80pt}+ \sum_{j=0}^q \tau(T0(X_1), \dots, T0(X_{p}), \widetilde{0}_{\varphi_{1}},\dots,\widetilde{0}_{\varphi_{j-1}}, \overline{\beta_j},\widetilde{0}_{\varphi_{j+1}}, \dots \widetilde{0}_{\varphi_{q}}).
\end{align*}
Formula \eqref{linear:eq} now follows from Lemma \ref{lemma:linear_comp}.

Conversely, let us assume $\tau \in \Gamma(\wedge^p T^*\E \otimes \wedge^q T\E)$ is a $(q,p)$-tensor field for
which \eqref{linear:eq} holds. It is straightforward to check that \eqref{linear:eq} is linear on the fibers
of $\bE \Arrow \M$, so $\tau$ is linear. To prove that $(D, l, r)$ are exactly the infinitesimal components of $\tau$,
one proceeds as follows: first substitute $(\underline{\U},\underline{\xi})$ with $(T^pu(\underline{\X}), \Rev_u(\underline{\varphi}))$.
In this case, $e_i=0$, $\beta_j=0$ and formula \eqref{linear:eq} becomes
$
\tau(T^pu(\underline{\X}), \Rev_u(\underline{\varphi})) = D(u)(\underline{\X},\underline{\varphi}).
$
By substituting $(\underline{\U},\underline{\xi})$ with
$(\overline{e_1}, T0(X_2), \dots, T0(X_p), \widetilde{0}_{\varphi_{1}},\dots, \widetilde{0}_{\varphi_{q}})$,
one has that $X_1=0$, $e_2=\dots=e_p=0$, $\beta_j=0$ and, therefore, formula \eqref{linear:eq} becomes
$$
\tau(\overline{e_1}, T0(X_2), \dots, T0(X_p), \widetilde{0}_{\varphi_{1}},\dots, \widetilde{0}_{\varphi_{q}}) = l(e_1)(X_2, \dots, X_p, \varphi_1, \dots, \varphi_q).
$$
Finally, by substituting $(\underline{\U},\underline{\xi})$ with
$(T0(X_1), \dots, T0(X_p), \overline{\beta_1}, \widetilde{0}_{\varphi_{2}},\dots, \widetilde{0}_{\varphi_{q}})$,
formula \eqref{linear:eq} becomes
$$
\tau( T0(X_1), \dots, T0(X_p), \overline{\beta_1}, \widetilde{0}_{\varphi_{2}},\dots, \widetilde{0}_{\varphi_{q}}) = r(\beta_1)(X_1, \dots, X_p, \varphi_2, \dots, \varphi_q).
$$
The result now follows from Lemma \ref{lemma:linear_comp}.
\end{proof}

As an immediate consequence, we have

\begin{corollary}\label{cor:1-1}
There is a one-to-one correspondence defined by \eqref{linear:eq} between linear tensors
$\tau \in \Gamma(\wedge^p T^*\E \otimes \wedge^q T\E)$ and triples $(D,l,r)$,
where $l:\E \Arrow \wedge^{p-1} T^*M\otimes \wedge^{q} \E$ and $r:T^*M \Arrow \wedge^{p} T^*M\otimes \wedge^{q-1} \E$ are
vector bundle maps covering the identity and
$D: \Gamma(\E) \Arrow \Gamma(\wedge^p T^*M\otimes \wedge^q \E)$ satisfies the Leibniz condition \eqref{D:leibniz}.
\end{corollary}

\subsection{Proof of Theorem \ref{thm:main}}
We just saw in Cor.~\ref{cor:1-1} how linear tensors $\tau$ on a vector bundle are described in terms of
triples $(D,l,r)$. Let $(A, [\cdot,\cdot], \rho)$ be a Lie algebroid, and consider linear tensors $\tau$ on $A$ for which
the corresponding fiberwise linear functions $c_{\tau}: \bA \Arrow \R$ are Lie-algebroid cocycles. We now see how to express
this additional cocycle property in terms of $(D,l,r)$.

\begin{proposition}\label{prop:lie_alg_tensor}
There is a one-to-one correspondence defined by \eqref{linear:eq} between linear $(q,p)$-tensors
$\tau \in \Gamma(\wedge^p T^*A\otimes \wedge^q TA)$ for which $c_{\tau}: \bA \Arrow \R$ is a Lie-algebroid cocycle
and IM $(q,p)$-tensors $(D, l, r)$ on
$A$, as in Definition~\ref{def:IMtensor}.
\end{proposition}

\begin{proof}
By definition, the cocycle condition $d_{\bA} \,c_{\tau} = 0$ is equivalent to the equation
\begin{equation}\label{d_zero}
\<c_{\tau},[U,V]\>= \Lie_{\rho_{\bT}(U)} \<c_{\tau}, V\> - \Lie_{\rho_{\bT}(V)} \<c_{\tau}, U\>
\end{equation}
all $U, V \in \Gamma(\bA)$.
In order to prove that this equality holds, it suffices to consider $U, V$ varying on the set of generators given by
Proposition~\ref{prop:generators}, and this will be shown to be equivalent to the set of IM-equations (IM1)--(IM6)
in Definition~\ref{def:IMtensor}.
Let us first fix $U= (T^pa, \Rev_a^{\, q})$, for $a \in \Gamma(A)$.  In the following, we shall use repeatedly
the anchor and Lie bracket equations \eqref{big_anchor}, \eqref{big_bracket} for the Lie algebroid $\bA \to \M$.

\smallskip

\paragraph{\bf Equation (IM1):} Take $V=(T^pb, \Rev_b^{\,q})$, for $b \in \Gamma(A)$.
It follows from \eqref{D_linear} that the cocycle equation \eqref{d_zero} is equivalent to
\begin{align*}
c_{D([a,b])} & =  \Lie_{(\rho(a)^{T,p}, H_a^q)} \, c_{D(b)} - \Lie_{(\rho(b)^{T,p}, H_b^q)} \, c_{D(a)} \vspace{5pt}\\
   & = c_{a\cdot D(b) - b\cdot D(a)},
\end{align*}
where the last equality follows from Proposition \ref{Lie_tensor}. So, in this case \eqref{d_zero} is equivalent to \eqref{IM1}.

\smallskip

\paragraph{\bf Equations (IM2) and (IM3):} Take $V= \calb b_{(i)}$, $1 \leq i \leq p$. From \eqref{l_linear}, it follows that
the cocycle equation \eqref{d_zero} for this pair $U,V$ can be rewritten as
\begin{align*}
(-1)^{i-1}c_{l([a,b])}\circ \gamma_{(i,0)} & = (-1)^{i-1}\Lie_{\rho_{\bT}(T^pa, \,\,\Rev_a^{\, q})}\, \left( c_{l(b)}\circ \gamma_{(i,0)}\right) - \Lie_{\rho_{\bT}(\calb b_{(i)})}\, c_{D(a)}\vspace{5pt}\\
& = (-1)^{i-1} \left( \Lie_{(\rho(a)^{T,p-1}, \,\,H_a^{\, q})}\,  c_{l(b)} \right) \circ \gamma_{(i,0)}
- \Lie_{(\rho(b)^{\vl, p}_{(i)},0)} \, c_{D(a)}\vspace{5pt} \\
& = (-1)^{i-1} \left( \left(c_{a\cdot l(b)}\right) \circ \gamma_{(i,0)} -  c_{\,i_{\rho(b)}D(a)} \circ \gamma_{(i,0)}\right)
\end{align*}
where $\gamma_{(i,0)}$ is the projection \eqref{proj:i} and the last equality follows from Proposition \ref{Lie_tensor}.
So, for the given choices of $U$ and $V$, \eqref{d_zero} is equivalent to \eqref{IM2}.

When $V=\calb\alpha_{(j)}$, $1 \leq j \leq q$, for $\alpha \in \Omega^1(M)$, one can prove analogously
that \eqref{d_zero} and \eqref{IM3} are equivalent.

\smallskip

\paragraph{\bf Equations (IM4), (IM5) and (IM6):} Let $U= \calb a_{(i)}$ and $V=\calb b_{(k)}$,
for $1 \leq i < k \leq p$. As $[\calb a_{(i)}, \calb b_{(k)}]=0$, it follows from \eqref{l_linear} that
the cocycle equation \eqref{d_zero} can be rewritten as
\begin{align*}
0 & = (-1)^{k-1} (\Lie_{(\rho(a)_{(i)}^{\vl,\, p-1}, \,0)} c_{l(b)}) \circ \gamma_{(k,0)} - (-1)^{i-1}(\Lie_{(\rho(b)_{(k-1)}^{\vl, \,p-1}, \,0)} c_{l(a)})\circ \gamma_{(i,0)}\\
 & = (-1)^{i+k-2} c_{i_{\rho(b)} l(a)} \circ \gamma^{(p-1,q)}_{(i,0)} \circ \gamma_{(k,0)} - (-1)^{i+k-3} c_{i_{\rho(a)} l(b)} \circ \gamma^{(p-1,q)}_{(k-1,0)} \circ \gamma_{(i,0)}\\
 & = (-1)^{i+k-2} \left(c_{i_{\rho(b)} l(a)} \circ \gamma^{(p-1,q)}_{(i,0)} \circ \gamma_{(k,0)} + c_{i_{\rho(a)} l(b)} \circ \gamma^{(p-1,q)}_{(k-1,0)} \circ \gamma_{(i,0)}\right),
\end{align*}
where in the second equality we have used Proposition \ref{Lie_tensor}.
One can now directly check that $\gamma^{(p-1,q)}_{(i,0)} \circ \gamma_{(k,0)}$ and $\gamma^{(p-1,q)}_{(k-1,0)} \circ \gamma_{(i,0)}$
are the same projection from $\M^{(p,q)}$ to $\M^{(p-2,q)}$, which forgets the $i$-th and the $k$-th components on $TM$.
Hence, for these choices of $U$ and $V$, \eqref{d_zero} is equivalent to \eqref{IM4}.

In a similar way, one checks that, for $U=\calb \alpha_{(j)}$, $V=\calb \beta_{(k)}$, $1 \leq j < k \leq q$,
one obtains the equivalence of  \eqref{d_zero} with \eqref{IM5}, and for
$U=\calb a_{(i)}$, $V=\calb \alpha_{(j)}$, $1 \leq i \leq p$, $1 \leq j \leq q$, one has
the equivalence of  \eqref{d_zero} with \eqref{IM6}.

These 6 cases cover all possibilities of $U,V$ varying in the set of generators, so the result follows.
\end{proof}

Let now $\G \toto M$ be a Lie groupoid with Lie algebroid $A \to M$. For a multiplicative $(q,p)$-tensor field
$\tau \in \Gamma(\wedge^p T^*\G \otimes \wedge^q T\G)$, consider the corresponding multiplicative function
$c_\tau \in C^\infty(\bG)$ on the groupoid \eqref{big_group}. Let $Ac_\tau \in \Gamma(\bA^*)\subseteq C^\infty(\bA)$
be its associated infinitesimal cocycle.

\begin{lemma}\label{main_lemma}
The following holds:
\begin{itemize}
\item[(a)] There is a linear tensor field $\tau_{A} \in \Gamma(\wedge^p T^*A \otimes \wedge^q TA)$ such
$A c_{\tau} = c_{\tau_{A}}$.
\item[(b)] The infinitesimal components $(D_{A}, l_{A}, r_{A})$ of $\tau_A$
satisfy
\begin{align*}
D_{A}(a)(\underline{\X}, \underline{\varphi}) & = \left(\Lie_{\overrightarrow{a}} \tau\right)(\underline{\X}, \underline{\varphi})\\
l_{A}(a)|_{\gamma_{(i,0)}(\underline{\X}, \underline{\varphi})} & = \left(i_{\overrightarrow{a}}
\tau\right)|_{\pi_{(i,0)} (\underline{\X}, \underline{\varphi})}\\
r_{A}(\alpha)|_{\gamma_{(0,j)}(\underline{\X}, \underline{\varphi})} & = \left(i_{\t^*\alpha} \tau\right)|_{\pi_{(0,j)} (\underline{\X}, \underline{\varphi})},
\end{align*}
where $(\underline{\X}, \underline{\varphi}) \in \M$,  $a \in \Gamma(A)$, $\alpha \in \Omega^1(M)$
and $\gamma_{(i,0)}$, $\gamma_{(0,j)}$, $\pi_{(i,0)}$ and $\pi_{(0,j)}$ are the forgetful
projections  \eqref{proj:i}, \eqref{proj:j}, \eqref{projpi:i} and \eqref{projpi:j}, respectively.
\item[(c)] The infinitesimal components $(D, l, r)$ of $\tau$ coincide with those of $\tau_{A}$.
\end{itemize}
\end{lemma}

\begin{proof}
It follows from Proposition \ref{prop:comp_group} that $Ac_{\tau}: \bA \Arrow \R$ is a componentwise linear
function which is antisymmetric on the $TA$ components as well as on the $T^*A$ components.
Hence, there exists $\tau_{A} \in \Gamma(\wedge^p T^*A \otimes \wedge^q TA)$ such that $A c_{\tau} = c_{\tau_{A}}$.
This proves (a).

By \eqref{eq:cocycle}, \eqref{D_linear} and Proposition \ref{Lie_tensor},
\begin{align*}
D_{A}(a)(\underline{\X}, \underline{\varphi}) & = \<c_{\tau_{A}}, (T^pa(\underline{\X}), \Rev_a^q(\underline{\varphi}))\>
= \<A c_{\tau}, (T^pa(\underline{\X}), \Rev_a^q(\underline{\varphi}))\>\\
& = (\Lie_{\overrightarrow{(T^pa, \Rev_a^q)}} \, c_{\tau})(\underline{\X}, \underline{\varphi}) =
(\Lie_{(\overrightarrow{a}^{T,p}, H_a^q)} \, c_{\tau})(\underline{\X}, \underline{\varphi})\\
& = (\Lie_{\overrightarrow{a}}\tau)(\underline{\X}, \underline{\varphi}).
\end{align*}
Similarly,
\begin{align*}
(-1)^{i-1} l_{A}(a)(\gamma_{(i,0)}(\underline{\X}, \underline{\varphi})) & =
\<c_{\tau_{A}}, \calb a_{(i)}(\underline{\X}, \underline{\varphi})\> = \<A c_{\tau}, \calb a_{(i)}(\underline{\X}, \underline{\varphi})\>\\
& = (\Lie_{\overrightarrow{\calb a_{(i)}}} \, c_{\tau})(\underline{\X}, \underline{\varphi}) =
(\Lie_{(\overrightarrow{a}^{\vl,p}_{(i)}, 0)} \, c_{\tau})(\underline{\X}, \underline{\varphi})\\
& = (-1)^{i-1}(i_{\overrightarrow{a}}\tau)(\pi_{(i,0)}(\underline{\X}, \underline{\varphi})).
\end{align*}
The equation involving $r_{A}$ follows similarly, and we conclude that (b) holds.

If we now let $(D, l, r)$ be the infinitesimal components of $\tau$, the equalities
 $D = D_{A}$, $l = l_{A}$ and $r = r_{A}$ follow from Theorem \ref{thm:mult_char}.
\end{proof}

We can now finally proceed to the proof of our main result.

\begin{proof}[Proof of Theorem \ref{thm:main}]
By Proposition \ref{prop:lie_alg_tensor}, there exists a Lie algebroid $(q,p)$ tensor field $\tau_{A}$
having $(D,l,r)$ as its infinitesimal components. The fact that $c_{\tau_{A}}: A\bG \Arrow \R$ is a
Lie algebroid cocycle and $\bG$ is source 1-connected implies that there exists a unique multiplicative function
$F: \bG \Arrow \R$ satisfying $AF = c_{\tau_{A}}$. By Proposition \ref{prop:comp_group}, $F$ is componentwise
linear and anti-symmetric on both the $T\G$ and $T^*\G$ components. Therefore, $F= c_{\tau}$ for a
(unique) multiplicative tensor field $\tau \in \Gamma(\wedge^p T^*\G \otimes \wedge^q T\G)$.
The fact that the infinitesimal components of $\tau$ are $(D,l,r)$ follows from Lemma \ref{main_lemma}.
\end{proof}

\section{Multiplicative vector-valued forms}\label{sec:vectorvforms}

Given a manifold $N$, by a {\em vector-valued form} on $N$ we mean an element of
$\Omega^\bullet(N,TN)=\Gamma(\wedge^\bullet T^*N \otimes TN)$. The space of vector-valued forms
is a graded Lie algebra with respect to the Fr\"olicher-Nijenhuis
bracket \cite{FN}. On a Lie groupoid, the space of {\em multiplicative} vector-valued forms is closed under the
Fr\"olicher-Nijenhuis bracket \cite{bd}, so it is also a graded Lie algebra.
We now identify its infinitesimal counterpart, in the spirit of Remarks~\ref{rem:complexes} and \ref{rem:gerst}.
Before discussing Lie groupoids, we briefly recall vector-valued forms on manifolds.


\subsection{The graded Lie algebra of vector-valued forms}

Let $\Omega^\bullet(N)$ be the graded algebra
of differential forms on $N$. A {\it degree $k$ derivation} of
$\Omega^\bullet(N)$ is a linear map $\Delta: \Omega^\bullet(N)\to
\Omega^{\bullet + k}(N)$ such that $\Delta(\alpha \wedge\beta)=
\Delta(\alpha)\wedge \beta + (-1)^{kj}\alpha\wedge \Delta(\beta)$, for $\alpha
\in \Omega^{j}(N)$. Any vector-valued form $K \in \Gamma(\wedge^pT^*N \otimes TN)$
gives rise to a degree $(p-1)$ derivation of $\Omega^\bullet(N)$ by
\begin{align}
\label{dfn:contraction}& i_K\omega(X_1, \dots, X_{p+j-1}) =  \\
\nonumber & \frac{1}{p!(j-1)!}\sum_{\sigma \in
S_{p+j-1}} sgn(\sigma) \,\omega(K(X_{\sigma(1)},
\dots, X_{\sigma(p)}), X_{\sigma(p+1)}, \dots, X_{\sigma(p+j-1)}),
\end{align}
for $\omega \in \Omega^j(N)$, $X_1, \ldots, X_{p+j-1} \in TN$. It also gives rise to a degree $p$ derivation of $\Omega^\bullet(N)$
via
\begin{equation}\label{eq:lieK}
\Lie_{K} = [i_K, d] = i_Kd -(-1)^{p-1}di_K,
\end{equation}
where $d$ is the exterior differential on $N$.

We extend $i_K$ to a contraction operation $i_K: \Omega^{\bullet}(N, TN) \Arrow \Omega^{\bullet+p-1}(N,TN)$ by
\begin{equation}\label{ext:contraction}
i_K(\omega \otimes X) = (i_K\omega)\otimes X, \,\qquad \omega \in \Omega(N), \, X \in \frakx(N).
\end{equation}

Given $K\in \Omega^p(N,TN)$ and $L\in \Omega^{p'}(N,TN)$, their {\it
Fr\"olicher-Nijenhuis bracket} \cite{FN} (see also \cite[Ch.~2]{nat}) is the vector-valued form $[K,L]\in
\Omega^{p+p'}(N,TN)$ uniquely defined by the condition
\begin{equation}\label{eq:comm}
\Lie_{[K, L]} = [\Lie_K, \Lie_L] = \Lie_K \Lie_L - (-1)^{pp'} \Lie_L
\Lie_K.
\end{equation}
When $K$ and $L$ have degree zero (i.e., they are vector fields on
$N$), \eqref{eq:comm} agrees with the usual Lie bracket of vector
fields. More generally, for $X \in \frakx(N)$,
$$
[X, K] = \Lie_X K.
$$

The Fr\"olicher-Nijenhuis bracket makes $\Omega^\bullet(N,TN)$ into a graded Lie algebra.
It is a natural bracket in the following sense: for a smooth map $F: N_1\to N_2$ , let $K_i\in \Omega^p(N_i,TN_i)$,
$L_i\in \Omega^{p'}(N_i,TN_i)$, $i=1,2$, be such that $K_1$ is
$F$-related to $K_2$ and $L_1$ is $F$-related to $L_2$ \footnote{$K_1 \in \Omega^p(N_1,TN_1)$ is
{\it $F$-related} to $K_2 \in \Omega^p(N_2,TN_2)$ if
$$
K_2(TF(X_1),\ldots,TF(X_k))=TF(K_1(X_1,\ldots,X_k)),
$$
for all $X_1,\ldots,X_k \in T_xN$, and $x\in N$.}. Then
$[K_1,L_1]$ is $F$-related to $[K_2,L_2]$. There are other important properties of the Fr\"olicher-Nijenhuis
bracket which will be recalled in subsequent sections.


\subsection{Infinitesimal description}
Let $\G$ be a Lie groupoid. As seen in Proposition \ref{prop:vect_valued}, a  multiplicative vector-valued form
$K \in \Omega^p(\G, T\G)$ as defined by \cite{LMX} is exactly a multiplicative $(1,p)$-tensor field.
From Theorem~\ref{thm:main}, one obtains a bijective correspondence between multiplicative vector-valued $p$-forms
on a Lie groupoid $\G$ and IM $(p,1)$-tensors on its Lie algebroid $A$. For this reason, we will refer to IM $(p,1)$-tensors
also as {\em IM vector-valued $p$-forms}.

We denote by $\Omega^\bullet_{\scriptscriptstyle{mult}}(\G,T\G)$ the space of multiplicative vector-valued forms on $\G$.
In the following, we will also need the following result proven in \cite[Thm.~4.3]{bd}:

\begin{proposition}\label{prop:FNgrp}
On a Lie groupoid $\G$, $\Omega^\bullet_{\scriptscriptstyle{mult}}(\G,T\G)$ is closed under the
Fr\"olicher-Nijenhuis bracket.
\end{proposition}

Hence $\Omega^\bullet_{\scriptscriptstyle{mult}}(\G,T\G) \subseteq \Omega^\bullet(\G,T\G)$ is a graded Lie subalgebra.
We now describe the graded Lie bracket on IM vector-valued forms corresponding to the Fr\"olicher-Nijenhuis bracket on
multiplicative vector-valued forms.

For a multiplicative vector-valued form $K \in \Omega^p(\G, T\G)$, consider its infinitesimal components
$D: \Gamma(A) \Arrow \Gamma(\wedge^p T^*M \otimes A)$, $l: A \Arrow \wedge^{p-1} T^*M \otimes A$ and
$r: T^*M \to \wedge^p T^*M$. Note that $r$ can be seen alternatively as an element $r \in \Omega^p(M, TM)$.
As such, Proposition \ref{prop:vect_valued} shows that $K$ is $\s,\t$-related to $r$.

Using the $\Omega^\bullet(M)$-module structure of $\Omega^\bullet(M, A) = \Gamma(\wedge^\bullet T^*M \otimes A)$,
we  extend $l$ to an operator $l: \wedge^{\bullet} T^*M\otimes A \Arrow \wedge^{\bullet+p-1} T^*M\otimes A$ by
$$
l(\alpha \otimes a) = \alpha \wedge l(a),
$$
and $D$ to an operator $D: \Omega^{j}(M, A) \Arrow \Omega^{p+j}(M, A)$ by
\begin{equation}\label{D:ext}
D(\alpha \otimes a) = \alpha \wedge D(a) + (-1)^j(d\alpha \wedge l(a) - (-1)^{j(p-1)}{\Lie_{r}\alpha}\otimes a),
\end{equation}
with $\Lie_r$ as in \eqref{eq:lieK}.
This extension for $D$ is well-defined as a consequence of the Leibniz rule \eqref{D:leibniz}. Moreover,
$$
D(\alpha \wedge \eta) = \alpha \wedge D(\eta) + (-1)^{i+j}(d\alpha  \wedge l(\eta) + (-1)^{(i+j)(p-1)} \Lie_{r}\alpha \wedge \eta),
$$
for $\alpha \in \Omega^{i}(M)$, and $\eta \in \Omega^j(M, A)$.

\begin{lemma}\label{lem:itK}
For $\eta \in \Omega^\bullet(M, A)$, we have
$i_{\T(\eta)} K = \T(l(\eta))$.
\end{lemma}

\begin{proof}
For homogeneous $\eta=\alpha \otimes a$, one has that $\T(\eta)=\t^*\alpha \otimes \overrightarrow{a}$. So, by definition of the contraction  \eqref{ext:contraction},
$$
i_{\T(\eta)} K = \t^*\alpha\wedge i_{\overrightarrow{a}} K  = \t^*\alpha \wedge \T(l(a))
  = \T(\alpha \wedge l(a))
  = \T(l(\alpha \otimes a)).
$$
%
\end{proof}

Let us consider the operation $[\cdot, K]: \Omega^{\bullet}(\G, T\G) \Arrow \Omega^{\bullet+k}(\G, T\G)$, where $[\cdot, \cdot]$ is the Fr\"olicher-Nijenhuis bracket.
The following result shows that $[\cdot, K]$ preserves the image of $\T$ inside $\Omega^\bullet(\G, T\G)$.

\begin{lemma}\label{J:inv}
For $\eta \in \Omega^{\bullet}(M, A)$, we have
$[\T(\eta), K] = \T(D(\eta))$.
\end{lemma}

\begin{proof}
If $\eta = \alpha \otimes a$, for $\alpha \in \Omega^j(M)$, then one has (see \cite[Sec.~8.7]{nat}):
\begin{align*}
[\t^*\alpha \otimes \overrightarrow{a}, K] & = \t^*\alpha \wedge [\overrightarrow{a}, K] + (-1)^j(\t^*d\alpha \otimes i_{\overrightarrow{a}}K - (-1)^{(j-1)p}\Lie_{K}(\t^*\alpha) \otimes \overrightarrow{a})\\
   &= \t^*\alpha \wedge \T(D(a)) + (-1)^j(\t^*d\alpha\otimes \T(l(a)) - (-1)^{j(p-1)}\t^*(\Lie_{r}\alpha)\otimes \overrightarrow{a}\\
   & = \T(\alpha \wedge D(a) +(-1)^j(d\alpha \otimes l(a) - (-1)^{j(p-1)}\Lie_{r}\alpha \otimes a))\\
   & = \T(D(\eta)).
\end{align*}
In the second equality, we have used the fact that $K$ and $r$ are $\t$-related.
\end{proof}

\begin{proposition}\label{prop:compFN}
Let $K_1 \in \Omega^{p_1}(\G, T\G)$, $K_2 \in \Omega^{p_2}(\G, T\G)$ be multiplicative vector-valued forms, with
infinitesimal components $(D_1,l_1,r_1)$ and $(D_2,l_2,r_2)$, respectively.
The infinitesimal components $(D,l,r)$ of their Fr\"olicher-Nijenhuis bracket $[K_1,K_2]$ are
\begin{align}
\label{D:frolicher}D & = [D_{2}, D_{1}] = D_{2}\circ D_{1} - (-1)^{p_1p_2} D_{1} \circ D_{2}\\
\label{l:frolicher}l & = [D_{2},l_{1}] - (-1)^{p_1p_2}[D_{1}, l_{2}],
\end{align}
where the brackets on the right-hand side are the (graded) commutators of endomorphisms of $\Omega^\bullet(M,A)$, and
\begin{equation}\label{r:frolicher}
r = [r_{1}, r_{2}],
\end{equation}
where the last bracket is the Fr\"olicher-Nijenhuis bracket of $r_{i} \in \Omega^{p_i}(M, TM)$, $i=1,2$.
\end{proposition}

\begin{proof}
The equation \eqref{r:frolicher} follows from the naturality of the Fr\"olicher-Nijenhuis bracket. As $K_1$ is $\t$-related to $r_1$ and $K_2$ is $\t$-related to $r_2$, it follows that $[K_1,K_2]$ must be $\t$-related to $[r_1,r_2]$.
The identity \eqref{D:frolicher} for $D$ follows from the Jacobi equation for the Fr\"olicher-Nijenhuis bracket.  Indeed, using Lemma~\ref{J:inv}:
\begin{align*}
\T(D(a)) &= [\overrightarrow{a}, [K_1,K_2]] = [[\overrightarrow{a}, K_1],K_2]+[K_1, [\overrightarrow{a},K_2]]\\
                 & = [\T(D_1(a)), K_2] + [K_1, \T(D_2(a))]\\
                 & = \T(D_2(D_1(a))) - (-1)^{p_1p_2}[\T(D_2(a)), K_1]\\
                 & = \T(D_2(D_1(a)) - (-1)^{p_1 p_2} D_1(D_2(a))) = \T([D_2, D_1](a)).
\end{align*}

Now, recall the following general property of the Fr\"olicher-Nijenhuis bracket (see Theorem 8.11 in \cite{nat}):
for $K_i \in \Omega^{p_i}(N, TN)$ and $L \in \Omega^{p'+1}(N, TN)$, $i=1,2$, we have
\[
\begin{split}
i_L[K_1, K_2]  = & [i_LK_1, K_2] + (-1)^{p_1p'}[K_1, i_LK_2] \\
&  - \left((-1)^{p_1p'}\,i_{[K_1, L]},K_2 - (-1)^{(p_1+p')p_2}\,i_{[K_2,L]}K_1\right).
\end{split}
\]
Using this identity and the previous lemmas, we obtain
\begin{align*}
i_{\overrightarrow{a}}[K_1,K_2]  & = [i_{\overrightarrow{a}}K_1, K_2] + (-1)^{p_1}[K_1, i_{\overrightarrow{a}}K_2]\\
& \hspace{-40pt} - ((-1)^{p_1}i_{[K_1, \overrightarrow{a}]} \, K_2 - (-1)^{(p_1-1)p_2}i_{[K_2, \overrightarrow{a}]}\, K_1)\\
& = [\T(l_1(a)), K_2] + (-1)^{p_1}\left( -(-1)^{p_1(p_2-1)} [\T(l_2(a)), K_1] \right)\\
& \hspace{-40pt}  + ((-1)^{p_1}i_{\T(D_1(a))} \,K_2 -(-1)^{(p_1-1)p_2}i_{\T(D_2(a))} \, K_1)\\
& = \T\left(D_2(l_1(a)) - (-1)^{p_1p_2} D_1(l_2(a)) + (-1)^{p_1} l_2(D_1(a)\right) \\
& \hspace{-40pt} - (-1)^{(p_1-1)p_2} l_1(D_2(a)))\\
& =  \T\left([D_2, l_1](a) -(-1)^{p_1 p_2} (D_1(l_2(a)) - (-1)^{p_1(p_2-1)} l_2(D_1(a)))\right)\\
& = \T\left([D_2, l_1](a) - (-1)^{p_1 p_2}[D_1, l_2](a)\right),
\end{align*}
as we wanted to prove.
\end{proof}

\begin{corollary}
The space of IM vector-valued forms is a graded Lie algebra with the bracket defined by
\eqref{D:frolicher},\eqref{l:frolicher} and \eqref{r:frolicher}:
$$
[(D_1, l_1,r_1), (D_2,l_2,r_2)] = ([D_1,D_2], [l_1,l_2], [r_1,r_2]).
$$
The correspondence established by Theorem \ref{thm:main} between multiplicative vector-valued forms and
IM vector-valued forms is a graded Lie algebra isomorphism.
\end{corollary}

This last result should be regarded as parallel to Remarks~\ref{rem:complexes} and \ref{rem:gerst}.

\section{Multiplicative $(1,1)$-tensor fields}\label{1_1}

We will now focus on multiplicative vector-valued 1-forms, or $(1,1)$-tensor fields.

\subsection{Infinitesimal components}\label{subsec:1-1ic}

Let $K \in \Omega^1(\G, T\G)$ be a multiplicative $(1,1)$-tensor field, with infinitesimal components
$$
D: \Gamma(A) \Arrow \Gamma(T^*M \otimes A), \;\;\; l: A \Arrow A, \;\;\; r: TM \Arrow TM.
$$
Note that we have dualized the $r$ component.
For $X\in TM$, we will use the notation
$$
D_X: \Gamma(A)\to \Gamma(A),\;\;\;
D_X(a)= i_X D(a).
$$

The $IM$-equations satisfied by the triple $(D,l,r)$ take the form
\begin{align}
\tag{IM1*}\label{IM1_1} D_X([a,b]) &= [a, D_X(b)] - [b, D_X(a)] + D_{[\rho(b), X]} (a) - D_{[\rho(a), X]}(b)\\
\tag{IM2*}\label{IM2_1} l([a,b]) & = [a, l(b)] - D_{\rho(b)}(a)\\
\tag{IM3*} \label{IM3_1}r([\rho(a), X]) & = [\rho(a),r(X)] - \rho(D_{X}(a))\\
\tag{IM6*}\label{IM6_1} r \circ \rho & = \rho \circ l.
\end{align}

\begin{proposition}
Let $K: T\G \Arrow T\G$ be a multiplicative $(1,1)$-tensor field on the Lie groupoid $\G$,
with infinitesimal componentes $(D,l,r)$. Then $K^n$ is also multiplicative, and its
infinitesimal components $(D', l', r')$ satisfy
$$
l' = l^n,\qquad r' = r^n, \qquad D'(a) = \sum_{j=1}^n l^{j-1} \circ D(a) \circ r^{n-j}
$$
\end{proposition}

\begin{proof}
The equations for $l'$ and $r'$ are straightforward to check. As for the equation for $D'$, the proof follows
from an induction on $n$ using the recursion formula
$$
[\overrightarrow{a}, K^n] = [\overrightarrow{a}, K^{n-1}] \circ K + K^{n-1}\circ [\overrightarrow{a}, K].
$$
\end{proof}

As a result, we obtain infinitesimal descriptions of multiplicative projections and almost complex/product structures.

\begin{corollary}\label{cor:proj}
Let $K$ be a multiplicative $(1,1)$-tensor field on a source-connected Lie groupoid $\G  \toto M$ with infinitesimal
componentes $(D, l, r)$. Then
\begin{itemize}
\item[(a)] $K$ satisfies $K^2=K$ if and only if
$$
l \circ D(a) + D(a)\circ r  =  D(a), \qquad l^2  = l,
\qquad r^2  = r.
$$
\item[(b)] $K$ satisfies $K^2=\pm \,\id_{T\G}$ if and only if
$$
l \circ D(a) + D(a)\circ r  =  0, \qquad
l^2  = \pm \,\id_{A}, \qquad
r^2  = \pm \,\id_{TM}.
$$
\end{itemize}
\end{corollary}

\begin{proof}
 In (a), the equations for $(D, l, r)$ guarantee that the multiplicative $(1,1)$-tensors $K^2$ and $K$ have the same
 infinitesimal components. As $\G$ has connected $\s$-fibers, this implies that $K^2=K$. The same argument holds for
 almost product and almost complex structures.
\end{proof}

We will now consider an additional integrability condition in terms of the Nijenhuis torsion.

\subsection{Nijenhuis torsion}\label{subsec:nij}

Given a (1,1)-tensor field $K$ on a manifold $N$, its {\em Nijenhuis torsion} is the vector-valued
2-form $\N_K\in \Omega^2(N,TN)$ given by
$$
\N_K(X,Y)=[K(X),K(Y)]-K([KX,Y]+ [KY,X]) + K^2[X,Y],
$$
for $X,Y \in TN$. The Nijenhuis torsion has a well-known relation with the Fr\"olicher-Nijenhuis bracket via
\begin{equation}\label{eq:Nij}
\frac{1}{2}[K,K] = \N_K.
\end{equation}

For a multiplicative (1,1)-tensor field on a Lie groupoid $\G\toto M$, the following description of the infinitesimal components of
its Nijenhuis torsion is an immediate consequence of this last formula and Propositions ~\ref{prop:FNgrp} and ~\ref{prop:compFN}:

\begin{corollary}\label{cor:nij}
Let $K \in \Omega^1(\G, T\G)$ be multiplicative with infinitesimal components $(D,l,r)$. Then $\N_K \in \Omega^2(\G,T\G)$ is
multiplicative and its infinitesimal components $(D',l',r')$ are
$$
D'=D^2,\qquad l'=[D,l], \qquad r'= \N_{r}.
$$
\end{corollary}

It will be useful to have a more concrete expression for
$$
D^2: \Gamma(A) \Arrow \Gamma(\wedge^2 TM \otimes A).
$$
Recall \cite[Cor.~8.12]{nat} the following expression for the Fr\"olicher-Nijenhuis bracket of $K_1, K_2 \in \Omega^1(\G, T\G)$:
\begin{align*}
[K_1,K_2](U,V)  = & [K_1(U), K_2](V) - [K_1(V), K_2](U) - K_1([K_2(U), V] - [K_2(V), U])\\
&  + (K_2K_1+K_1K_2)([U,V]).
\end{align*}
Since
$
[\overrightarrow{a}, \frac{1}{2}[K,K]] = [[\overrightarrow{a},K],K],
$
by taking $K_1=[\overrightarrow{a},K]$ and $K_2=K$, and letting $U=X$, $V=Y$ be in $TM$,
one readily obtains that
\begin{equation}\label{D:square}
D^2_{(X,Y)}  =  D_Y \circ D_X - D_X \circ D_Y - D_{[r(X), Y]} + D_{[r(Y), X]} + l\circ D_{[X,Y]} + D_{r([X,Y])}
\end{equation}
where both sides are seen as maps $\Gamma(A)\to \Gamma(A)$.

Corollary~\ref{cor:nij} gives a complete infinitesimal description of general multiplicative
Nijenhuis operators on Lie groupoids. In the next subsections, we will illustrate how this
general result can be specialized to various cases of interest.

\subsection{Poisson quasi-Nijenhuis structures}\label{subsec:PN}

Poisson-Nijenhuis structures  \cite{KM,MM} play a central role in the theory of integrable systems.
Their recent connections with Lie groupoids arose in quantization schemes for Poisson manifolds, see e.g. \cite{bon}.
In this section we revisit the more general Poisson quasi-Nijenhuis structures \cite{StX}. We establish a link with the theory of
IM $(1,1)$-tensors, which leads to an extension of the integration of Poisson quasi-Nijenhuis
structures in  \cite[Thm.~6.2]{StX} (originally based on the theory of Lie bialgebroids \cite{Kos}) as a consequence of
Theorem \ref{thm:main}.

 Given a Poisson manifold $(M,\Pi)$, consider its cotangent bundle $T^*M$ with the Lie bracket $[\cdot,\cdot]_\Pi$
as in \eqref{bracket:pi}. We say that a (1,1) tensor $r: TM \to TM$ is \textit{compatible with $\Pi$} if
\begin{equation}\label{eq:comp1}
\Pi^\sharp \circ r^* = r \circ \Pi^\sharp
\end{equation}
(equivalently, $r \circ \Pi^\sharp: T^*M \to TM$ is skew-symmetric) and the following equation holds: for all
$\alpha,\beta \in \Omega^1(M)$,
\begin{equation}\label{eq:comp2}
C^r_{\Pi}(\alpha,\beta):=
[\alpha, \beta]_{\Pi_r} - ([r^*\alpha, \beta]_\Pi + [\alpha, r^*\beta]_\Pi -
r^*([\alpha, \beta]_\Pi))= 0,
\end{equation}
where $[\cdot,\cdot]_{\Pi_r}$ is the bracket \eqref{bracket:pi} for the bivector field $\Pi_r$ defined by
$r \circ \Pi^\sharp$.

\begin{remark}\label{rem:C}
The condition $C^r_{\Pi} = 0$ implies that $[\Pi, \Pi_r] = 0$ (here $[\cdot,\cdot]$ is the Schouten bracket), but the converse does not hold in general; it does if $\Pi$ is symplectic, see e.g. \cite{Vas}.
\end{remark}

The following definition extends \cite[Def.~3.3]{StX}:

\begin{definition}
A {\em Poisson quasi-Nijenhuis structure} is a pair $(\Pi, r)$, where $\Pi$ is a Poisson bivector field, $r: TM \to TM$ is a
$(1,1)$ tensor, such that $\Pi$ and $r$ are compatible
and the following condition holds:
\begin{align}
\label{eq:comp3}
\N_r^*([\alpha, \beta]_\Pi) & = \Lie_{\Pi^\sharp(\alpha)}\N_r^*(\beta) - i_{\Pi^\sharp(\beta)} d\N_r^*(\alpha)\\
\label{eq:comp4}
i_{\Pi^\sharp(\alpha)}\,\N_r^*(\beta) & = - i_{\Pi^\sharp(\beta)} \,\N_r^*(\alpha), \;\;\; \forall \,
\alpha,\beta \in \Omega^1(M),
\end{align}
where $\N_r^*: T^*M \to \wedge^2 T^*M$ is the adjoint of the Nijenhuis torsion of $r$, given by
$\N^*_r(\alpha)(X,Y) = \<\alpha, \N_r(X,Y)\>$.
We refer to a {\em symplectic quasi-Nijenhuis structure} when $\Pi$ is nondegenerate, i.e., a symplectic structure.
\end{definition}

Note that \eqref{eq:comp3} and \eqref{eq:comp4} say that $(\N_r^*,0)$ defines an IM $3$-form on $T^*M$.

\begin{remark}\label{rem:phitype} The Poisson quasi-Nijenhuis structures considered in \cite[Def.~3.3]{StX} are slightly more
restricted than in our definition: they are required to satisfy
the condition
$\N_r^*(\alpha) = - i_{\Pi^{\sharp}(\alpha)} \phi,$
or, equivalently,
\begin{equation}\label{eq:3form}
\N_r(X,Y) = \Pi^\sharp(\phi(X,Y,\cdot)),
\end{equation}
for a given closed 3-form $\phi \in \Omega^3(M)$. One may verify that this implies that $\N_r^*$
automatically satisfies \eqref{eq:comp3} and \eqref{eq:comp4}. This difference in the definitions will become more
transparent when we talk about integration, see Thm.~\ref{thm:PqN} and Cor.~\ref{cor:PqN} below.
\end{remark}

Following the previous remark, we shall refer to the structures satisfying \eqref{eq:3form}
as {\em Poisson quasi-Nijenhuis structures relative to $\phi$}; we will specify them by triples $(\Pi,r,\phi)$
to make the dependence on the closed 3-form $\phi$ explicit.

The following proposition gives an alternative way to express the compatibility between a (1,1)-tensor $r$ and a Poisson structure $\Pi$ in terms of IM tensors.
Let $D^r: \Gamma(T^*M) \to \Gamma(T^*M\otimes T^*M)$ be defined by
\begin{equation}\label{PN:D}
\<D^r_X(\alpha), Y\> = d\alpha(X,r(Y)) - (\Lie_r\alpha)(X,Y),
\end{equation}
where $\Lie_r$ is the operator \eqref{eq:lieK}.

\begin{proposition}\label{prop:pn_comp}
The $(1,1)$-tensor $r$ and the Poisson tensor $\Pi$ are compatible if and only if the triple
$(D^r,r^*,r)$
is an IM $(1,1)$-tensor on $T^*M$.
\end{proposition}

\begin{proof}
The Leibniz equation for $D^r$ follows from the properties of $d$ and $\Lie_r$. So one only needs to check that the
IM-equations for $(D, r, r^*)$ are equivalent to $\Pi$ and $r$ being compatible. The IM equations in this case are:
\begin{align}
\label{IM1_2} D^r_X([\alpha,\beta]_\Pi) = &  [\alpha, D^r_X(\beta)]_\Pi - [\beta, D^r_X(\alpha)]_\Pi +
D^r_{[\Pi^\sharp(\beta), X]} (\alpha) - D^r_{[\Pi^\sharp(\alpha), X]}(\beta), \\
\label{IM2_2} r^*([\alpha,\beta]_\Pi) = & [\alpha, r^*(\beta)]_\Pi - D^r_{\Pi^\sharp(\beta)}(\alpha),\\
\label{IM3_2} r([\Pi^\sharp(\alpha), X])  = & [\Pi^\sharp(\alpha),r(X)] - \Pi^\sharp(D^r_{X}(\alpha)), \\
\label{IM6_2} r \circ \Pi^\sharp  = & \Pi^\sharp \circ r^*.
\end{align}
The compatibility equation \eqref{eq:comp1} is exactly \eqref{IM6_2}. Note that \eqref{IM2_2}
is equivalent, for $df, dg \in \Omega^1(M)$, to
\begin{align*}
r^*([df, dg]_\Pi) & = [df, r^*(dg)]_\Pi - i_{\Pi^\sharp(dg)} dr^*(df)\\
& \hspace{-50pt}= [df, r^*(dg)]_\Pi - \Lie_{\Pi^\sharp(dg)} r^*(df) + di_{\Pi^\sharp(dg)}r^*(df)\\
& \hspace{-50pt}= [df, r^*(dg)]_\Pi - \underbrace{(\Lie_{\Pi^\sharp(dg)} r^*(df) - \Lie_{\Pi^\sharp(r^*(df))} dg - di_{\Pi^\sharp(dg)}r^*(df))}_{[dg, r^*(df)]_{\Pi}}
- \Lie_{\Pi^\sharp(r^*(df))} dg.
\end{align*}
Using that $[dg, df]_{\Pi_r} = -\Lie_{\Pi^\sharp(r^*(df))} dg$, it follows that \eqref{IM2_2}
is equivalent to \eqref{eq:comp2}. So, $\Pi$ and $r$ are compatible if and only if $\eqref{IM2_2}$ and $\eqref{IM6_2}$ hold.
The remaining equations follow from these two. Indeed, a long but straightforward computation shows that
\eqref{IM1_2} is equivalent, for $\alpha=df$ and $\beta = dg$, to
$$
d (C^r_{\Pi}(df,dg)) = 0.
$$
Finally, the redundancies among the IM equations (see Remark~\ref{rem:redundacy}) guarantee that \eqref{IM3_2} holds.
\end{proof}

From now on, let us assume that the Lie algebroid $(T^*M,\Pi^\sharp,[\cdot,\cdot]_\Pi)$ integrates to a  source 1-connected symplectic groupoid
$(\G, \omega)$ (see Section~\ref{diff_forms}).
Let $\Pi_\omega$ be the Poisson structure defined by $\omega$, so that
\begin{equation}\label{pi_omega}
\Pi_\omega^\sharp: T^*\G \to T\G
\end{equation}
is the inverse map to $T\G\to T^*\G$, $U\mapsto i_U\omega$. It follows from Theorem \ref{thm:main} and Proposition \ref{prop:pn_comp} that any $(1,1)$ tensor $r: TM \to TM$ compatible with $\Pi$ corresponds to a multiplicative $(1,1)$ tensor $K: T\G \to T\G$ on $\G$ integrating $(D^r, r^*, r)$.

%

\begin{lemma}\label{prop:nij}
Let $r: TM \to TM$ be a $(1,1)$ tensor compatible with $\Pi$. One has that $(\N_r^*, 0)$ is an IM 3-form if and only if there exists a closed 3-form $\lambda \in \Omega^3(\G)$ such that
\begin{equation}\label{eq:nij_symp}
\N_K(U,V) = \Pi_\omega^\sharp(\lambda(U,V,\cdot)), \,\,\, \forall \,U, \, V \in \mathfrak{X}(\G),
\end{equation}
where $K: T\G \to T\G$ is the multiplicative $(1,1)$ tensor integrating $(D^r, r^*, r)$.
In this case, $\lambda$ is the multiplicative 3-form integrating $(-\N_r^*, 0)$.
\end{lemma}

\begin{proof}
Following Cor.~\ref{cor:nij}, the infinitesimal components $((D^r)^2, [D^r, r^*], \N_r)$ of the Nijenhuis torsion of $K$ can be explicitly calculated using \eqref{D:square} and \eqref{PN:D}. Indeed, using Cartan calculus and the Jacobi identity repeatedly, one may verify that, for $\alpha \in \Omega^1(M)$, $X, Y \in \mathfrak{X}(M)$,
\begin{align*}
[D^r,r^*](\alpha) & = - \N_r^*(\alpha),\\
\<(D^r)^2_{(X,Y)} (\alpha), Z\> & =  -d\N^*_r(\alpha)(X,Y,Z) + d\alpha(Z, \N_r(X,Y)).
\end{align*}
Similarly, for any multiplicative closed 3-form $\lambda \in \Omega^3(\G)$, consider the multiplicative $(1,2)$ tensor field $\tau \in \Omega^2(\G, T\G)$ defined by $\tau(U,V) = \Pi_\omega^\sharp(\lambda(U,V,\cdot))$. Its infinitesimal components $\widetilde{D}: \Gamma(T^*M) \to \Omega^2(M, T^*M)$, $\widetilde{l}: T^*M \to T^*M \otimes T^*M$ and $\widetilde{r}: T^*M \to \wedge^2 T^*M$ are:
$$
\widetilde{l}= \mu, \,\,\, \widetilde{r} = -\mu, \,\,\ \widetilde{D}_{(X,Y)}(\alpha) =  d\mu(\alpha)(X,Y,\cdot) + d\alpha(\underbrace{\mu_{(X,Y)}}_{\in TM}, \cdot),
$$
where  $\mu: T^*M \to \wedge^2 T^*M$ is such that $(\mu,0)$ is the IM 3-form corresponding to $\lambda$. The result now follows from the fact that \eqref{eq:nij_symp} holds if and only if $\widetilde{r}= \N_r$, $\widetilde{l}=-\N_r$, $\widetilde{D} = (D^r)^2$.
\end{proof}

Our aim now is to prove that $\Pi_\omega$ and $K$ define a symplectic quasi-Nijenhuis structure on $\G$; the
previous lemma shows that this structure will be of the more restricted type of Remark~\ref{rem:phitype}.
As $\Pi_\omega$ is symplectic, following Remark~\ref{rem:C},
it suffices to show that
\begin{align}
&\Pi_\omega^\sharp \circ K^*  = K \circ \Pi_\omega^\sharp, \label{eq:commute0}\\
&[\Pi_\omega,\Pi_K] = 0,\label{eq:commute}
\end{align}
where $\Pi_K$ is the bivector field defined by $K\circ \Pi_{\omega}^\sharp$.


\subsubsection*{Integration of Poisson quasi-Nijenhuis structures}

We start by analyzing conditions \eqref{eq:commute0} and \eqref{eq:commute} in general. Our setting will be:
\begin{itemize}
 \item a source 1-connected Poisson groupoid $(\H, \widetilde{\Pi}) \toto N$;
 \item a multiplicative $(1,1)$-tensor $K: T\G \to T\G$ with $(D,l,r)$ as infinitesimal components.
\end{itemize}
Let $(A^*\H, \rho_*, [\cdot,\cdot]_*)$ be the corresponding Lie algebroid structure on $A^*\H$ and denote by $\delta: \Gamma(\wedge^\bullet A\H) \to \Gamma(\wedge^{\bullet+1} A\H)$ the 2-differential associated with $\widetilde{\Pi}$. Define $\delta_K: \Gamma(A\H) \to \Gamma(A^*\H \otimes A^*\H)$ by the expression
\begin{equation}\label{delta_K}
\delta_K(a)(\mu_1,\mu_2) := \delta(a)(\mu_1, l^*(\mu_2)) - \<D_{\rho_*(\mu_1)}(a), \mu_2\>, \;\;\; a\in \Gamma(A\H),
\end{equation}

\begin{lemma}\label{lemma:pn1}
One has that
$$
K \circ \widetilde{\Pi}^\sharp = \widetilde{\Pi}^\sharp \circ K^* \Longleftrightarrow
\left\{
\begin{array}{l}
 r \circ \rho_* = \rho_*\circ l^*\\
 \delta_K(\mu_1, \mu_2) = - \delta_K(\mu_2, \mu_1).
\end{array}
\right.
$$
In this case, the IM $(0,2)$-tensor on $A\H$ corresponding to the multiplicative bivector field $\widetilde{\Pi}_K$ is
given by $l \circ (\rho_*)^*: T^*N \to A\H$ and $\delta_K : \Gamma(A\H) \to \Gamma(\wedge^2 A\H)$.
\end{lemma}

\begin{proof}
Let $\tau_L, \tau_R \in \Gamma(T\H \otimes T\H)$ be (not necessarily skew-symmetric) tensor fields on $\H$ defined by
$$
\tau_L(\xi_1, \xi_2) =  \<\xi_2, \,K(\Pi^\sharp(\xi_1))\>,\,\,\, \tau_R(\xi_1, \xi_2)= - \<\xi_1, \,K(\Pi^\sharp(\xi_2))\>,
$$
for $\xi_1$, $\xi_2 \in T^*\H$. Note that $K \circ \widetilde{\Pi}^\sharp = \widetilde{\Pi}^\sharp \circ K^*$ if and only if $\tau_L=\tau_R$. Now $\tau_L, \tau_R$ are multiplicative and, therefore, $\tau_L=\tau_R$ if and only if
$$
\underbrace{\Lie_{\overrightarrow{a}}\tau_L = \Lie_{\overrightarrow{a}}\tau_R}_{(i)}, \;\;
\underbrace{\tau_L(\t^*\alpha, \cdot) = \tau_R(\t^*\alpha, \cdot)}_{(ii)}, \; \mbox{ and }\;
\underbrace{\tau_L(\cdot\,, \t^*\alpha)  = \tau_R(\cdot\,, \t^*\alpha)}_{(iii)}, \,\, \alpha \in \Gamma(T^*M),
$$
using that $\H$ is source 1-connected (see Remarks \ref{rem:non_skew1} and \ref{rem:non_skew2}).
Since
$$
\tau_L(\t^*\alpha, \xi) = \<\widetilde{\t}(\xi), \,l \circ \rho^*_*(\alpha)\>, \;\;\;
 \tau_R(\t^*\alpha, \xi) = \<\widetilde{\t}(\xi), \,\rho_*^*\circ r^*(\alpha)\>,\;\;\;
\tau_L(\xi_2, \xi_1) = - \tau_R(\xi_1, \xi_2),
$$
it follows that (ii) and (iii) hold if and only if $r \circ \rho_* = \rho_*\circ l^*$. Finally, from
\begin{align*}
(\Lie_{\overrightarrow{a}} \tau_L)(\xi_1, \xi_2) & = \<\xi_2, \,[\overrightarrow{a}, K](\Pi^\sharp(\xi_1))\> + \<\xi_2,\, K((\Lie_{\overrightarrow{a}}\Pi)^\sharp(\xi_1)),\\
(\Lie_{\overrightarrow{a}} \tau_R)(\xi_1, \xi_2) & = - \<\xi_1, \,[\overrightarrow{a}, K](\Pi^\sharp(\xi_2))\> - \<\xi_1,\, K((\Lie_{\overrightarrow{a}}\Pi)^\sharp(\xi_2)),
\end{align*}
one can substitute $\xi_i=\mu_i \in A^*\H \subset T^*\H$ and use that $\Pi^\sharp|_{A^*\H} = - \rho_*$ to show that (i) holds if and only if $\delta_K$ is skew-symmetric.
\end{proof}

Let us assume that $K \circ \widetilde{\Pi}^\sharp = \widetilde{\Pi}^\sharp \circ K^*$.
The IM $(0,2)$-tensor on $A\H$ associated to the multiplicative bivector field $\widetilde{\Pi}_K$ defines a pre-Lie algebroid
structure on $A^*\H$. It follows from \eqref{pre_lie_br} and \eqref{delta_K} that the pre-Lie bracket is given by
$$
[\mu_1, \mu_2]_K = [l^*(\mu_1), \mu_2]_* + [\mu_1, l^*(\mu_2)]_* - l^*([\mu_1, \mu_2]_*) + \Theta(\mu_1, \mu_2),
$$
where $\Theta \in \Gamma(\wedge^2 A\H \otimes A^*\H)$ is defined by
$$
\<\Theta(\mu_1, \mu_2), a \> = \<D_{\rho_*(\mu_1)}(a), \mu_2\> + \delta(a)(l^*(\mu_1), \mu_2) - \delta(l(a))(\mu_1, \mu_2).
$$

For any element $\Omega \in \Gamma(\wedge^k A\H \otimes A^*\H)$, one can define a contraction operator
$i_{\Omega}: \Gamma(\wedge^\bullet A\H) \to \Gamma(\wedge^{\bullet+k-1} A\H)$ exactly as in \eqref{dfn:contraction}.
It is a graded derivation of degree $k-1$ of the graded algebra $\Gamma(\wedge^\bullet A\H)$. Note that
$$
\delta_K = [i_{l^*}, \delta] + i_{\Theta},
$$
where $[\cdot,\cdot]$ is the commutator of derivations of $\Gamma(\wedge^\bullet A\H)$.

\begin{lemma}\label{lemma:pn2}
The bivector fields $\widetilde{\Pi}$, $\widetilde{\Pi}_K$ satisfy
$[\widetilde{\Pi}, \widetilde{\Pi}_K] = 0$ if and only $[\delta,i_\Theta]=0$.
\end{lemma}

\begin{proof}
It is shown in \cite{StX} that $[\widetilde{\Pi}, \widetilde{\Pi}_K] = 0 \Leftrightarrow [\delta, \delta_K]=0$.
The result then follows from $[\delta, [i_{l^*}, \delta] ]=0$.
\end{proof}

We can now conclude the description of the integration of Poisson quasi-Nijenhuis structures using Thm.~\ref{thm:main}:

\begin{theorem}\label{thm:PqN}
Let $(\G, \omega)$ be the source 1-connected symplectic groupoid integrating a Poisson manifold $(M,\Pi)$.
There is a one-to-one correspondencee between Poisson quasi-Nijenhuis structures $(\Pi,r)$ on $M$
and symplectic quasi-Nijenhuis structures $(\omega, K,\lambda)$ relative to $\lambda$,
where $K: T\G \to T\G$ is the multiplicative $(1,1)$ tensor integrating $(D^r, r^*, r)$ and $\lambda \in \Omega^3(\G)$
is the multiplicative closed 3-form integrating $(-\N_r^*, 0)$.
\end{theorem}

\begin{proof}
We know that ${\Pi}_\omega$, the Poisson structure defined by $\omega$, makes $\G$ into a Poisson groupoid.
The dual Lie algebroid in this case is $A^* = TM$, with anchor $\rho_* = \mathrm{id_{TM}}$ and bracket $[\cdot, \cdot]_*$
given by the Lie bracket of vector fields. Note that $\delta=d$ is the de Rham differential.


Let $K: T\G \to T\G$ be a multiplicative $(1,1)$ tensor on $\G$ and $\lambda$ a multiplicative closed 3-form. Consider the IM (1,1)-tensor $(D,l,r)$ and the IM 3-form $(\mu,0)$ associated to $K$ and $\lambda$, respectively. From Lemmas \ref{lemma:pn1} and \ref{lemma:pn2}, one knows that $K$ and $\Pi_\omega$ are compatible if and only if $l=r^*$, $\delta_K$ in \eqref{delta_K} is skew-symmetric and, by writing $\delta_K = [i_r,d] + i_\Theta$, the condition $[d, i_\Theta]=0$ holds (notice that this last condition says that $\Theta=0$). Hence, using that $[d,i_r] = \Lie_r$,
\begin{align*}
\<D_X(\alpha), Y\>  = d\alpha(X, r(Y)) - \delta_K(\alpha)(X,Y)
 & = d\alpha(X, r(Y)) - \Lie_r(\alpha)(X,Y)\\
 & = \<D^r_X(\alpha), Y\>.
\end{align*}
Therefore $K$ and $\Pi_\omega$ are compatible if and only if $\Pi$ and $r$ are compatible and, moreover, $K$ is
the (1,1) tensor integrating $(D^r, r^*, r)$. Finally, from Lemma \ref{prop:nij}, the equation $\N_K(U,V) = \Pi^\sharp_\omega(\lambda(U,V,\cdot))$ holds if and if $(\mu,0)= (-\N_r^*,0)$ is the IM 3-form associated to $\lambda$.
%
%
%
\end{proof}

By restricting the previous correspondence to Poisson quasi-Nijenhuis structures on $M$ relative to closed 3-forms,
we recover \cite[Thm.~6.2]{StX} with a different viewpoint:

\begin{corollary}\label{cor:PqN}
Let $(\G, \omega)$ be a source 1-connected symplectic groupoid integrating the Poisson manifold $(M, \Pi)$.
There is one-to-one correspondence between Poisson quasi-Nijenhuis structures $(\Pi,r,\phi)$ on $M$ relative to a
a closed 3-form $\phi$ and symplectic quasi-Nijenhuis structures
$(\omega, K,\lambda)$ on $\G$ relative to $\lambda= \t^*\phi-\s^*\phi$ such that $K$ is multiplicative.
\end{corollary}
\begin{proof}
It follows from the fact that $\t^*\phi - \s^*\phi$ is the multiplicative 3-form integrating  the
IM 3-form $(\alpha \mapsto i_{\Pi^{\sharp}(\alpha)}\phi, 0)$ (see Example \ref{trivial:ex}).
\end{proof}

Our methods also work, more generally, to describe the infinitesimal counterparts of multiplicative
Poisson-Nijenhuis structures on Lie groupoids, not necessarily symplectic; here
one obtains compatibilities between the IM (1,1)-tensor corresponding to Nijenhuis structure
and the Lie bialgebroid associated with the Poisson groupoid. We will discuss this case elsewhere.


\subsection{Multiplicative (almost) complex structures}

Our general results on multiplicative (1,1) tensors in Sections~\ref{subsec:1-1ic} and \ref{subsec:nij} can be readily applied to
the study of complex structures on Lie groupoids, giving another viewpoint to results in \cite{LMX} concerning their infinitesimal
versions.

Recall that, for a (real) vector bundle $E\to M$, a holomorphic structure is specified by a triple $(J_E,J_M,\nabla)$, where $J_E: E\to E$ is an endomorphism satisfying $J_E^2=-\id$ (which makes $E$ into a complex vector bundle),
$J_M$ is a complex structure on $M$, and $\nabla$ is a flat $T^{(0,1)}$-connection
on $\E$, in such a way that holomorphic sections $u: M \Arrow \E$ are characterized by
$\nabla u =0$; see \cite{Rawnsley}. We shall call $\nabla$ the \textit{Dolbeault connection} on $E$.
More generally, we will be interested in holomorphic structures on  Lie groupoids and Lie algebroids.

A {\em holomorphic structure} on a Lie groupoid $\G\toto M$ is a multiplicative complex structure; i.e.,
a multiplicative $J\in \Omega^1(\G,T\G)$ such that $J^2=-\id$ and $\N_J=0$. One refers to $(\G,J)$
as a {\em complex} Lie groupoid. A {\em holomorphic structure} on a
Lie algebroid $A \Arrow M$ is a holomorphic structure $(J_A,J_M,\nabla)$ on its underlying vector bundle
such that the following compatibility conditions are satisfied:
\begin{enumerate}
\item[(H1)] $[\cdot, \cdot]$ restricts to a Lie bracket $[\cdot, \cdot]_{\mathrm{hol}}$ on the holomorphic sections;
\item[(H2)] $[\cdot, \cdot]_{\mathrm{hol}}$ is $\CC$-linear.
\end{enumerate}
As we will now see, the following correspondence, proven in \cite[Thm.~3.17]{LMX}, is a consequence of Theorem~\ref{thm:main},
along with Corollaries~\ref{cor:proj} and \ref{cor:nij}.

\begin{corollary}\label{cor:complex}
Let $\G\toto M$ be source 1-connected.
Then holomorphic structures on $\G$ are in natural one-to-one  correspondence with holomorphic structures on
its Lie algebroid $A\to M$.
\end{corollary}

The remainder of this section proves this result.
By Corollaries~\ref{cor:proj} and \ref{cor:nij}, we immediately see that the correspondence in Theorem~\ref{thm:main}
restricts to a bijective correspondence between holomorphic structures
$J$ on $\G\toto M$ and IM $(1,1)$-tensors $(D,l,r)$ on $A$ satisfying
\begin{equation}\label{eq:j2=-1}
l \circ D(a) + D(a)\circ r  =  0, \qquad
l^2  = -\id_{A}, \qquad
r^2  = -\id_{TM}.
\end{equation}
and
\begin{equation}\label{eq:nij=0}
D^2=0,\qquad [D,l]=0, \qquad \N_{r}=0.
\end{equation}

So we must verify that an IM $(1,1)$-tensor $(D,l,r)$, for which  \eqref{eq:j2=-1} and \eqref{eq:nij=0} hold,
is equivalent to a holomorphic structure on the Lie algebroid $A$. We start by checking that
\eqref{eq:j2=-1} and \eqref{eq:nij=0} exactly say that the triple $(D,l,r)$ defines a holomorphic structure on the vector bundle
underlying $A$.

From \eqref{eq:j2=-1}, it is clear that $r$ is an almost complex structure on $M$ and $l$ is a
complex structure on the fibres of $A$ (so we regard $A$ as a complex vector bundle).
Defining $\nabla: \Gamma(T^{01})\times \Gamma(A) \to \Gamma(A)$ by
\begin{equation}\label{eq:nabla}
\nabla_{X+ i r(X)}(a) := -l (D_X(a)),
\end{equation}
we also see that the first equation in \eqref{eq:j2=-1} says that
$$
\nabla_{i (X+ i r(X))} a = i (\nabla_{X+ ir(X)} a).
$$
Moreover, this last property along with the Leibniz rule for $D$ imply that
$$
\nabla_{X+ir(X)} f a = f \nabla_{X + i r(X)} a + (\Lie_{X + ir(X)}f) a,
$$
for real functions $f\in C^\infty(M)$.

Let us now consider \eqref{eq:nij=0}. The third condition says that $r$ is a complex structure on $M$, while the
second says that $\nabla$ is complex linear in $A$. It follows that $T^{01}$ is a (complex) Lie algebroid and $\nabla$ is a
(complex) $T^{01}$-connection.
Finally, using \eqref{D:square}, one can also check that the first equation in \eqref{eq:nij=0} amounts to $\nabla$
being flat. In conclusion, conditions \eqref{eq:j2=-1} and \eqref{eq:nij=0} say that $(D,l,r)$ endow $A$ with the structure of a
holomorphic vector bundle: $J_A=l$, $J_M=r$ and $\nabla$ given by \eqref{eq:nabla}. The result in
Corollary~\ref{cor:complex} now follows from

\begin{lemma}
The IM-equations for $(D,l,r)$ are equivalent to conditions (H1) and (H2) above.
\end{lemma}

\begin{proof}
We saw that equations \eqref{eq:j2=-1} and \eqref{eq:nij=0} say that
 $(D,l,r)$ defines a holomorphic structure on the vector bundle $A$; moreover,
a section $a\in \Gamma(A)$ is holomorphic if and only if $Da=0$.

\noindent ($\Rightarrow$):
Note that \eqref{IM1_1} implies that $[\cdot, \cdot]$ restricts to a Lie bracket
$[\cdot, \cdot]_{\mathrm hol}$ on the holomorphic sections of $A$, and \eqref{IM2_1} implies that
$[\cdot, \cdot]_{\mathrm hol}$ is $\CC$-linear.

\noindent
($\Leftarrow$:) Assume that $(A, [\cdot, \cdot], \rho)$ is a holomorphic Lie algebroid, and consider
\begin{align*}
E_l(a,b) & := l([a,b]) - [a, l(b)] + D_{\rho(b)}(a)\\
E_r(a,X) & := r([\rho(a), X]) - [\rho(a), r(X)] + \rho(D_{X}(a))\\
E_D(a,b,X) & := D_X([a,b]) - [a, D_X(b)] + [b, D_X(a)] - D_{[\rho(b), X]}(a) + D_{[\rho(a), X]}(b).
\end{align*}
One can check that $E_l$ and $E_r$ are $\C(M)$-linear on both components and $E_D$ is anti-symmetric on the $\Gamma(A)$-components. Moreover, for $f\in C^\infty(M)$,
\begin{equation}\label{cal_D}
E_D(a,fb, X) = f E_D(a,b,X) + (\Lie_{X}f) E_l(a,b) - (\Lie_{E_r(a,X)} f) b.
\end{equation}

Note that $[\cdot, \cdot]_{\mathrm hol}$ being $\CC$-linear implies that $E_l(a,b)=0$ for $a,b \in \Gamma(A)$ holomorphic.
As $\Gamma(A)$ is generated as a $\C(M)$-module by the holomorphic sections, it follows that $E_l \equiv 0$. The redundancies
discussed in Remark~\ref{rem:redundacy} imply that $\rho \circ l = r \circ \rho$, so $\rho$ is a complex vector-bundle morphism.
Furthermore, note that $\rho: A \Arrow TM$ sends holomorphic sections to holomorphic sections. Indeed, if $h \in \C(M, \CC)$ is a
 holomorphic function and $u_1, u_2 \in \Gamma(A)$ are arbitrary holomorphic sections,
$$
(\Lie_{\rho(u_2)} h) u_1 = h[u_1, u_2] - [u_1, hu_2] = h [u_1, u_2]_{\mathrm hol} - [u_1, hu_2]_{\mathrm hol}
$$
is a holomorphic section, which implies that $\rho(u_2)$ is a holomorphic section of $TM$. Using the $\C(M)$-linearity of $E_r$,
one can argue as above to prove that $E_r(a,X)=0$ for all $a \in \Gamma(A)$ and $X \in \Gamma(TM)$ (use that $r= J_M$, the
almost complex structure of $M$, and that $TM$ with the Lie bracket of vector fields and the identity as anchor is a holomorphic
Lie algebroid). Finally, from \eqref{cal_D} it follows that $E_D$ is tensorial, and the fact that $[\cdot, \cdot]$ restricts to
$[\cdot, \cdot]_{\mathrm hol}$ on holomorphic sections implies, as before, that $E_D \equiv 0$.
\end{proof}


As this section and Section~\ref{subsec:PN} illustrate, Theorem~\ref{thm:main}
provides tools that can be directly applied to treat multiplicative geometric structures on holomorphic
Lie groupoids, including holomorphic symplectic groupoids \cite{LMX2} or more general holomorphic Poisson groupoids,
as well as multiplicative generalized complex structures \cite{JSX}, offering complementary information
about the latter in terms of infinitesimal components. We will further discuss these cases in a separate work.

\subsection{Multiplicative projections}
For a Lie groupoid $\G\rightrightarrows M$ with Lie algebroid $A$,
we consider a multiplicative $(1,1)$ tensor field $K \in \Omega^1(\G, T\G)$ satisfying $K^2 = K$,
referred to as a {\em multiplicative projection}.
In this section we apply our previous results to describe multiplicative projections infinitesimally,
making connections with the theory of matched pairs \cite{KS-M,Lu2,mokri}.

We start by observing that projections  can be used to treat other types of multiplicative $(1,1)$ tensors.

\begin{example}\label{ex:product}
Suppose that $Q:T\G\to T\G$ satisfies $Q^2=\id$ and is multiplicative; i.e., $Q$ is a multiplicative
{\em (almost) product structure} on $\G$. Then $K:= (Q+\id)/2$ is a multiplicative projection\footnote{By working with complexifications, one can also cast (almost) complex structures as projections.}.

\end{example}

We know that a multiplicative projection $K$ admits an infinitesimal description by its infinitesimal components $(D,l,r)$.
We start by discussing alternative ways to express the operator $D$, that will be
convenient when we consider the Nijenhuis torsion of $K$.

Since $r^2 = r: TM\to TM$ and $l^2=l:A\to A$ (by Cor.~\ref{cor:proj}), the bundles $TM$ and $A$ decompose as
$$
A = A^0\oplus A^1, \,\, \, TM = T^0 \oplus T^1,
$$
where $A^0, T^0$ are the kernels of the maps $l,r$, and $A^1, T^1$ are their images, respectively.

\begin{lemma}
One has that
\begin{equation}\label{proj:anchor}
\rho(A^0) \subset T^0 \text{ and } \rho(A^1) \subset T^1.
\end{equation}
Also, for $a \in \Gamma(A)$,
\begin{align*}
D_X(a) \in A_1&=\mathrm{im}(l),\;\;\; \mbox{ if } \; X \in \Gamma(T^0),\\
D_X(a) \in A_0&={\ker}(l),\;\;\; \mbox{ if } \; X \in \Gamma(T^1).
\end{align*}
\end{lemma}

\begin{proof}
The conditions in \eqref{proj:anchor} follow from \eqref{IM6_1} (Section~\ref{subsec:1-1ic}), whereas the
statements about $D$ follow from Cor.~\ref{cor:proj},
part (a).
\end{proof}

So, upon restriction, $D$ gives rise to two operators:
$$
D^{+}: \Gamma(A) \Arrow \Gamma(T^{0\,*} \otimes A^1),\;\;\; D^-: \Gamma(A) \Arrow \Gamma(T^{1\,*}\otimes A^0).
$$
Let us consider the operators
\begin{align}
\label{def:plus}\Lambda^+ = D^+|_{\Gamma(A^0)}, & \;\;\,\nabla^+= D^+|_{\Gamma(A^1)}\\
\label{def:minus}\Lambda^- = D^-|_{\Gamma(A^1)}, & \;\;\,\nabla^- = - D^-|_{\Gamma(A^0)}.
\end{align}

\begin{proposition}\label{prop:connections}
$\Lambda^+, \Lambda^-$ are tensorial, whereas $\nabla^+, \nabla^-$ satisfy
\begin{align*}
&\nabla^+(fa) = f \nabla^+(a) + df|_{T^0} \otimes a,\\
&\nabla^-(fb) = f \nabla^-(b) + df|_{T^1} \otimes b,
\end{align*}
for $f \in \C(M)$, $a \in \Gamma(A^0), \, b \in \Gamma(A^1)$.
\end{proposition}

\begin{proof}
This follows immediately from the Leibniz rule for $D$ \eqref{D:leibniz}.
\end{proof}

\subsubsection*{Vanishing of the Nijenhuis torsion.}

Let $\N_K \in \Omega^2(\G, T\G)$ be the Nijenhuis torsion of $K$.
We say that $K$ is a \textit{flat projection} if $\N_K=0$.
The next result gives an equivalent description of the Nijenhuis vanishing condition.

\begin{proposition}\label{prop:Nij_vanish} Let  $K \in \Omega^1(\G, T\G)$ be a multiplicative projection on a source-connected
Lie groupoid $\G \toto M$. Then
$\N_K = 0$ if and only if
\begin{itemize}
\item $\Lambda^+=0$, $\Lambda^- = 0$;
\item $T^0$ and $T^1$ are involutive distributions;
\item $\nabla^+$ is a flat $T^0$-connection, and  $\nabla^-$ is a flat $T^1$-connection.
\end{itemize}
\end{proposition}


\begin{proof}
As $\G$ is source connected, the Nijenhuis torsion $\N_K$ vanishes if and only if its infinitesimal components
$(D^2, [D,l], N_r)$ are zero (see Corollary \ref{cor:nij}).
A direct verification shows that
\begin{equation*}
[D,l]_X(a) = D_X(l(a)) - l(D_X(a)) =
\begin{cases}
\Lambda_X^+(l(a)-a), & \text{ if } X \in T^0,\\
\Lambda_X^-(l(a)), & \text{ if } X \in T^1.
\end{cases}
\end{equation*}
Hence, $[D,l] =0$ if and only if both $\Lambda^+$ and $\Lambda^-$ vanish.
Similarly, $N_r = 0$ is equivalent to $T^0$ and $T^1$
being involutive\footnote{In general, for a projection $K \in \Omega^1(M, TM)$ a projection, the Nijenhuis torsion $\N_K$
can be written as $\N_K = R_K + \overline{R}_K,$
where $R_K, \overline{R}_K \in \Omega^2(M, TM)$ are the \textit{curvature} and \textit{co-curvature} of $K$ given, respectively, by
\begin{align*}
R_K(X,Y)  &= K([(\id-K)(X), (\id-K)(Y)])\\
\overline{R}_K(X,Y) & =(\id-K)([K(X), K(Y)]).
\end{align*}
In particular, $\N_K=0$ if and only if both $\ker(K)$ and $\mathrm{im}(K)$ are involutive distributions.}.
By Proposition \ref{prop:connections}, $\nabla^+$ is a $T^0$-connection, and $\nabla^-$ is a $T^1$-connection.
Using \eqref{D:square}, we see that
$$
D^2_{(X,Y)}(a) =
\left\{
\begin{array}{ll}
\mathrm{Curv}^+_{(Y,X)}(a), & \text{ if } X, Y \in \Gamma(T^0), \, a \in \Gamma(A^1)\\
\mathrm{Curv}^-_{(Y,X)}(a), & \text{ if } X, Y \in \Gamma(T^1), \, a \in \Gamma(A^0)\\
0, & \text{ otherwise, }\\
\end{array}
\right.
$$
where $\mathrm{Curv}^{+}$ (resp. $\mathrm{Curv}^{-}$) is the curvature of $\nabla^{+}$ (resp. $\nabla^-$).
Hence, $D^2 = 0$ if and only if both $\nabla^+, \nabla^-$ are flat. This concludes the proof.
\end{proof}

\begin{remark}
A distribution $\Delta \subset T\G$ is said to be multiplicative if $\Delta$ is a Lie subgroupoid.
As observed in \cite{bd}, a multiplicative projection is equivalent to a pair of multiplicative
distributions $\Delta_1$, $\Delta_2$ such that $\Delta_1\oplus \Delta_2=T\G$.
Also, $\N_K =0$ is equivalent to both distributions being involutive. In this context,
Proposition \ref{prop:Nij_vanish} agrees with the integrability criteria for multiplicative distributions given in
\cite{CSS, JO}.
\end{remark}

\begin{example}\label{ex:product2}
Following Example~\ref{ex:product}, consider a multiplicative $Q: T\G\to T\G$ satisfying $Q^2=\id$,
and let $K=(Q+\id)/2$ be the corresponding multiplicative projection.
Let $(D,l,r)$ and $(D',l',r')$ be the infinitesimal components of $K$ and $Q$, respectively. Then
$$
l= (l'+ \id)/2,\;\;\; r= (r'+\id)/2,\;\;\;\;\; D=D'.
$$
The bundles $T^0$, $A^0$ (resp. $T^1$, $A^1$) are now the $-1$ (resp. $+1$) eigenbundles of $r'$ and $l'$. Also, $D'$ decomposes into tensors $\Lambda^+ \in \Gamma(T^{0 \, *} \otimes A^{0 \,*} \otimes A^1)$, $\Lambda^- \in \Gamma(T^{1 \, *} \otimes A^{1 \, * } \otimes A^0)$ and connections $\nabla^+: \Gamma(A^1) \to \Gamma(T^{0 \, *} \otimes A^1)$, $\nabla^-: \Gamma(A^0) \to \Gamma(T^{ 1\, *} \otimes A^0)$. Noticing that $\N_K = 0 \Leftrightarrow \N_Q = 0$, we see that
Prop.~\ref{prop:Nij_vanish} directly applies to $Q$ instead of $K$.
\end{example}

We now illustrate the infinitesimal components of a multiplicative projection in the classical example of a
projection defined by a connection on a principal bundle.

\begin{example}
Let $P \to M$ be a principal bundle for a Lie group $G$,
and consider its gauge groupoid $\G(P):=(P\times P)/G \toto M$ (see e.g. \cite[Sec,~1.1]{Mac-book}).
In \cite{bd}, it is shown that there is a one-to-one correspondence between principal connections
$\theta \in \Omega^1(P, \frakg)$ and multiplicative projections $K: T\G(P) \to T\G(P)$ such that
$\mathrm{im}(K)=\ker(T\s) \cap \ker(T\t)$. To explicitly describe this projection it is useful to identify the
tangent groupoid $T\G(P)$ with the gauge groupoid $\G(TP)= (TP \times TP)/TG \to TM$ of the principal $TG$-bundle $TP \to TM$.
The quotient map $TP \times TP \to \G(TP)$ is denoted by
$$
(X,Y) \mapsto \overline{(X,Y)}.
$$
The projection $K$ is now defined as
$$
K(\overline{(X, Y)}) = \overline{(\theta(X)_P, \theta(Y)_P)},
$$
where $u_P$ denotes the infinitesimal generator on $P$ corresponding to $u\in \frakg$.
One can check (see \cite{bd}) that the Nijenhuis torsion of $K$ is given by
$$
\N_K = \mathcal{S}(F_{\theta}) - \T(F_{\theta}),
$$
where $F_{\theta} \in \Omega^2(M, P\times_G \frakg)$ is the curvature of $\theta$, and the maps $\mathcal{T}$ and $\mathcal{S}$
are defined in \eqref{T:def};
here $P\times_G \frakg$ is the associated bundle with respect to the adjoint representation, and we
are using the Atiyah sequence
$$
0 \longrightarrow P\times_G \frakg \longrightarrow TP/G \longrightarrow TM \longrightarrow 0,
$$
to view $P\times_G \frakg$ as a subbundle of $A(\G(P))=TP/G$.

The infinitesimal components $(D,l,r)$ of the multiplicative projection $K$ are  given as follows:
$r=0$, while $l: TP/G \to P\times_G \frakg \subset TP/G$ is the map induced by $\theta: TP\to \mathfrak{g}$,
and $D: \Gamma(TP/G) \to \Gamma(T^*M \otimes TP/G)$ is given by
$$
D_X(\overline{Y}) = \theta([X_H, Y]),
$$
where $X_H \in \frakx(P)$ is the horizontal lift of $X$. Note that $T^0=TM$, $T^1=0$, $A^0= H/G \cong TM$, $A^1= P\times_G \frakg$,
where $H=\ker(\theta) \subset TP$ is the horizontal distribution. Under the splitting
$D= D^++D^-$, one may directly  check that $D^-=0$ and
$\nabla^+ : \Gamma(P\times_G \frakg) \to \Gamma(T^*M\otimes P\times_G \frakg)$ is the natural connection on the associated bundle
$P\times_G \frakg$, whereas $\Lambda^+: \Gamma(TM) \to \Gamma(T^*M \otimes P\times_G \frakg)$ is
$$
\Lambda^+_X(Y) = \theta([X_H, Y_H]) = -F_{\theta}(X,Y).
$$
\end{example}

Prop.~\ref{prop:Nij_vanish} admits yet another geometric interpretation, that we discuss next.



\subsubsection*{Characterization via matched pairs}

In the remainder of this section we provide a characterization of multiplicative flat projections using the theory
of matched pairs of Lie algebroids \cite{mokri}.

\begin{definition}
Let $A, B \Arrow M$ be Lie algebroids. We say that $(A, B)$ is a {\em matched pair} if $A$ has a representation on $B$,
and $B$ has a representation on $A$ such that
\begin{align}
\label{match1}[\rho_A(a), \rho_B(b)] & = -\rho_A(\nabla_b \,a) + \rho_B (\nabla_a \,b),\\
\label{match2}\nabla_a \,[b_1,b_2] & = [\nabla_a \,b_1, b_2] + [b_1, \nabla_a \,b_2] + \nabla_{\nabla_{b_2} \,a} \,b_1 -
\nabla_{\nabla_{b_1}\,a} \,b_2,\\
\label{match3}\nabla_b \,[a_1,a_2] & = [\nabla_b \,a_1, a_2] + [a_1, \nabla_b \,a_2] + \nabla_{\nabla_{a_2} \,b} \,a_1 -
\nabla_{\nabla_{a_1}\,b} \,a_2.
\end{align}
Here we denote both representations by $\nabla$.
\end{definition}

\begin{definition}
A {\em morphism of matched pairs} from $(A_1, B_1)$ to $(A_2,B_2)$ is a pair of
Lie algebroid morphisms $F_A: A_1 \Arrow A_2$, $F_B: B_1 \Arrow B_2$ such that
\begin{align}
\label{morf1}\nabla_{F_A(a)} \,F_B(b) &= F_B(\nabla_a \,b),\\
\label{morf2}\nabla_{F_B(b)} \, F_A(a) & = F_A (\nabla_b\, a).
\end{align}
\end{definition}

 A matched pair $(A, B)$ is equivalent to a Lie algebroid structure on the
Whitney sum $A\oplus B$ such that $A$ and $B$ are Lie subalgebroids, see \cite{Mac-doubles, mokri}.
From this viewpoint, the representations are determined by the Lie bracket:
$$
\nabla_a \,b = \pr_B([a,b]) \,\,\, , \,\,\, \nabla_b \,a = \pr_A([b,a]),
$$
where $\pr_A: A\oplus B\Arrow A$ and $\pr_B: A\oplus  B \Arrow B$ are the projections.
In this context, a morphism of matched pairs from $(A_1,A_2)$ to $(B_1,B_2)$ is equivalent to a
Lie-algebroid morphism $F: A_1\oplus B_1 \Arrow A_2 \oplus B_2$ which restricts to Lie-algebroid morphisms
from $A_1$ to $A_2$, and from $B_1$ to $B_2$.

For a flat multiplicative projection $K \in \Omega^1(\G, T\G)$, Proposition~\ref{prop:Nij_vanish} implies that the
decomposition $TM = T^0 \oplus T^1$ defines a matched pair $(T^0, T^1)$. Also, for $a, b \in \Gamma(A^0)$, the condition
$$
l([a,b]) = [a,l(b)] - D_{\rho(b)}(a) = - \Lambda^+_{\rho(b)}(a) = 0
$$
implies that $A^0 \subset A$ is a subalgebroid. Similarly, one can check that $A^1 \subset A$ is a subalgebroid.
Thus, the decomposition $A = A^0 \oplus A^1$ defines a matched pair $(A^0,A^1)$.
For $a \in \Gamma(A^0), b \in \Gamma(A^1)$, the representations are defined by
$$
\nabla_a \,b = l([a,b]),\;\;\;\;\; \nabla_b \,a = [a,b] - l([a,b]).
$$

We obtain the following infinitesimal characterization of flat multiplicative projections:

\begin{theorem}
Let $\G$ be a source 1-connected groupoid.
There is a one-to-one correspondence between flat multiplicative projections on $\G$ and
decompositions $A = A^0 \oplus A^1$ and $TM = T^0 \oplus T^1$, where
\begin{itemize}
\item[(i)] $A^0, A^1 \subset A$ and $T^0, T^1 \subset TM$ are Lie subalgebroids;
\item[(ii)] $(A^0, T^1)$ and $(T^0, A^1)$ are matched pairs;
\item[(iii)] The sides of the commutative square
$$
\begin{CD}
(A^0,A^1) @>(\id_{A^0},\, \rho)>> (A^0, T^1)\\
@V (\rho,\, \id_{A^1}) VV     @VV (\rho,\, \id_{T^1})V \\
(T^0, A^1) @>> (\id_{T^0},\, \rho) > (T^0,T^1)
\end{CD}
$$
are morphisms of matched pairs.
\end{itemize}
\end{theorem}

\begin{proof}
Consider decompositions $A= A^0\oplus A^1$ and $TM = T^0\oplus T^1$ for which (i), (ii) and (iii) hold.
We will prove that they define the infinitesimal components of a flat multiplicative projection.

Define $l:A \Arrow A$ (resp. $r: TM \Arrow TM$) to be the projection on $A^1$ (resp. $T^1$) along $A^0$ (resp. $T^0$).
Using that $\mathrm{Ann}(T^1)\cong T^{0\,*}$ and $\mathrm{Ann}(T^0) \cong T^{1\,*}$, define
$D: \Gamma(A) \Arrow \Gamma(T^*M \otimes A)$ to be the map
\begin{equation}\label{D:nabla}
D(a) = \nabla^+(l(a)) - \nabla^-(a-l(a)),
\end{equation}
where $\nabla^+: \Gamma(A^1) \Arrow \Gamma(T^{0\,*}\otimes A^1)$, $\nabla^-: \Gamma(A^0) \Arrow \Gamma(T^{1\,*}\otimes A^0)$
are the flat connections (i.e., representations) corresponding to the matched pairs $(A^0, T^1)$, $(T^0, A^1)$, respectively.
One may directly verify that $D$ satisfies the Leibniz rule \eqref{D:leibniz}. We now prove that the triple $(D,l,r)$ satisfies
the IM-equations (Section~\ref{subsec:1-1ic}).

\vspace{5pt}
\paragraph{\bf Equation \eqref{IM6_1}:}
From (iii), one has that $\rho(A^0) \subset T^0$, $\rho(A^1) \subset T^1$, which is equivalent to $\rho \circ l = r \circ \rho$.

\paragraph{\bf Equation \eqref{IM2_1}:}
The subbundle $A^0\subset A$ is a Lie subalgebroid if and only if
$
l([a,b]) = 0,
$
for $a, b \in \Gamma(A^0)$. As $l(b)= 0 = D_{\rho(b)}(a)$, this implies that
$$
l([a,b]) = [a,l(b)] - D_{\rho(b)}(a),\;\;\; \forall \, a,\, b \in \Gamma(A^0).
$$
Similarly, one can check that \eqref{IM2_1} holds for $a, b \in \Gamma(A^1)$ using that $A^1 \subset A$ is a
Lie subalgebroid. It remains to verify that \eqref{IM2_1} holds for crossed terms.
Using (ii), let
$\nabla^{+,\,bas}: \Gamma(A^1) \times \Gamma(T^0) \Arrow \Gamma(T^0)$,
$\nabla^{-, \,bas}: \Gamma(A^0)\times \Gamma(T^1) \Arrow \Gamma(T^1)$ be the representations of $A^1$, $A^0$ on
$T^0$, $T^1$, respectively. For equation \eqref{match1} to hold for the matched pairs $(A^0, T^1)$, $(T^0, A^1)$,
one must have that
\begin{align}
\nabla^{+, \,bas}_a \, X & = [\rho(a), X] + \rho(\nabla^+_X \,a),\\
\nabla^{-, \,bas}_b \, Y & = [\rho(b),Y] + \rho(\nabla^-_Y \,b).
\end{align}
Now, $(\id_{A^0}, \,\rho):(A^0,A^1)\Arrow (A^0, T^1)$ and $(\id_{A^1}, \,\rho):(A^0,A^1) \Arrow (T^0, A^1)$ are morphisms of
matched pairs if and only if, for all  $a \in \Gamma(A^0)$, $b \in \Gamma(A^1)$,
\begin{align*}
\l([a,b])& = [a,b] -  D_{\rho(b)}(a) = [a, l(b)] - D_{\rho(b)}(a),\\
l([b,a]) &= -D_{\rho(a)}(b) = [b, l(a)] - D_{\rho(a)}(b),
\end{align*}
respectively. Altogether, this proves that \eqref{IM2_1} holds.
\paragraph{\bf Equation \eqref{IM3_1}:}
Similarly to the previous case, one can check that \eqref{IM3_1} follows from $T^0$, $T^1$
being involutive and $(\rho,\,\id_{T^1}): (A^0,T^1) \Arrow (T^0, T^1)$ and
$(\id_{T^0}, \rho): (T^0,A^1)\Arrow (T^0, T^1)$ being morphisms of matched pairs.
\paragraph{\bf Equation \eqref{IM1_1}:}
We recall \eqref{IM1_1} for convenience:
$$
D_X([a,b])= [a, D_X(b)] - [b, D_X(a)] + D_{[\rho(b), X]}(a) - D_{[\rho(a), X]}(b).
$$
There are 6 cases to check by taking $X$ in $T^0$ or $T^1$, and $a,b$ in $A^0$ or $A^1$.
\begin{enumerate}
\item $X \in \Gamma(T^0)$ and $a, b \in \Gamma(A^0)$,
\item $X \in \Gamma(T^0)$ and $a, b \in \Gamma(A^1)$,
\item $X \in \Gamma(T^0)$ and $a \in \Gamma(A^0), \, b \in \Gamma(A^1)$.
\end{enumerate}
The other 3 cases where $X \in \Gamma(T^1)$ work analogously. For (1), both sides of \eqref{IM1_1} are
trivially zero. For (2), using that $D|_{T^1}(a)=D|_{T^1}(b)=0$, one can show that \eqref{IM1_1} is
equivalent to \eqref{match2} for the matched pair $(T^0,A^1)$. Finally, for (3), one has that
\begin{align*}
D_X([a,b]) = - \nabla^+_{X} \, l([b,a]) & \stackrel{\eqref{IM2_1}}= - \nabla^+_X \, \left([b, \cancelto{0}{l(a)}] -D_{\rho(a)}(b)\right)
 = \nabla^+_X \nabla^+_{\rho(a)} \,b.
\end{align*}
On the other hand, using that $D|_{T^0}(a)=0$, the RHS of \eqref{IM1_1} simplifies to
\begin{align*}
[a, D_X(b)] + D_{r([\rho(b), X])}(a) - D_{[\rho(a), X]}(b) & \\
& \hspace{-100pt} \stackrel{\eqref{IM3_1}} =
 [a, D_X(b)] - D_{\rho(D_X(b))}(a) - \nabla^+_{[\rho(a), X]} \,b\\
& \hspace{-100pt} \stackrel{\eqref{IM2_1}} = l([a, D_X(b)]) - \nabla^+_{[\rho(a), X]} \,b.
\end{align*}
Using \eqref{IM2_1} once again, one can prove that
$$
l([a, D_X(b)]) = -l([D_X(b), a]) = D_{\rho(a)}D_{X}(b) = \nabla^+_{\rho(a)}\nabla^+_X \, b,
$$
so the case (3) follows from the flatness of $\nabla^+$. Putting everything together, we have proved that $(D,l,r)$
defines an IM $(1,1)$-tensor, so it integrates to a multiplicative $(1,1)$
tensor $K \in \Omega^1(\G, T\G)$ by Theorem \ref{thm:main}.
It then follows from Corollary \ref{cor:proj} and Proposition \ref{prop:Nij_vanish} that $K$ is a flat projection.

The converse, i.e., that the infinitesimal components of a multiplicative flat projection give rise to matched pairs as in
the statement of the theorem, is proven by similar arguments and is left to the reader.

%
%
%
%
%
\end{proof}

Following Example~\ref{ex:product2}, an analogous result holds for product structures; see \cite[Thm.~4.8]{LMX2} for a parallel result in the context of holomorphic structures.


\appendix

\section{Lie theory of componentwise linear functions}
In this appendix, we study componentwise linear functions on Whitney sums of VB-groupoids. We start by presenting
a useful characterization of these functions.

Let $\E_i\to M$ be vector bundles, $i=1,\ldots,p$, and let $\E= \E_1 \oplus \dots \oplus \E_p$.
Consider the multiplication by non-negative scalars $h: \R_{\geq 0}\times \E \Arrow \E$, $h_\lambda(e)=\lambda e$.
We shall refer to $h$ as the \textit{homogeneous structure} on $\E$. The next result gives a characterization of componentwise
linear functions on $E$ in terms of its homogeneous structure. This is an extension of the characterization of vector bundle maps in \cite{GR}, for $p=1$.

\begin{proposition}\label{comp:alt}
A smooth function $F: \E \Arrow \R$ is componentwise linear if and only if
\begin{itemize}
\item[(1)]$F\circ h_{\lambda} = \lambda^p F$, for all $\lambda \geq 0$,
\item[(2)] $F\circ 0_{i} = 0$,
\end{itemize}
where $0_{i}: \oplus_{1 \leq j \neq i \leq p }\E_j \to \oplus_{1 \leq j \leq p} E_j$ is the map
\begin{equation}\label{dfn:zero_map}
(v_1, \dots, v_{i-1}, v_{i+1}, \dots, v_p)  \mapsto (v_1, \dots, v_{i-1}, 0, v_{i+1}, \dots, v_p),
\end{equation}
for $i=1, \dots, p$.
\end{proposition}

\begin{proof}
We will consider the case $p=2$, the general case being a direct generalization.
By restricting $F$ to the fibers of $\E_1$ and $\E_2$, we can assume that $F: \R^{m}\times \R^n \Arrow \R$
satisfies $F(\lambda x, \lambda y) = \lambda^2 F(x,y)$, for all $\lambda \geq 0$,  and $F(x,0)=F(0,y)=0$.
Note that this implies that $F(0,0)=0$ and $DF(0,0)=0$. Now, one can use Taylor's Theorem and the homogeneity of $F$ to prove that
$F(z)= \sum_{1\leq i, j \leq m+n} \frac{\partial^2 F}{\partial z_i\partial z_j}(0,0) z_iz_j$, for $z=(x,y)$ satisfying $|z|=1$. Using the homogeneity once more, it is possible to extend the equality to arbitrary $z$. The condition that
$F(x,0)=F(0,y)=0$ implies that the terms $x_ix_j$ and $y_iy_j$ do not appear in the sum. Hence $F: \R^m \times \R^n \to \R$ is bilinear. This completes the proof.
\end{proof}

We now consider $\VB$-groupoids (see e.g. \cite[Sec.~11.2]{Mac-book} for details and original references),
following the viewpoint of \cite{BCdH}.

\begin{definition}
A $\VB$-groupoid is a square
\begin{equation*}
\xymatrix{
\V \ar@<-3pt>[d] \ar@<3pt>[d] \ar[r] & \G \ar@<-3pt>[d] \ar@<3pt>[d]\\
E \ar[r] & M,\\
}
\end{equation*}
whose horizontal sides are vector bundles, the vertical sides are Lie groupoids, satisfying the following compatibility condition:
denoting by $h$ and $h^E$ the homogeneous structures on $\V$ and $E$, then,
for each $\lambda\geq 0$, $h_{\lambda}: \V \Arrow \V$ is a groupoid morphism over $h_\lambda^E: E\to E$.
\end{definition}

The Lie algebroid $A\V \Arrow E$ of a $\VB$-groupoid inherits a vector bundle structure over $A$ by differentiation of $h_{\lambda}$.
If $h^A$ is the corresponding homogeneous structure, then $h_{\lambda}^{A}: A\V \Arrow A\V$ is a Lie algebroid morphism over $h_\lambda$ for each
$\lambda$.

In the following, we consider VB-groupoids $\V_1 \toto E_1$, $\dots$,  $\V_p \toto E_p$ over $\G\toto M$
and their Whitney sum $\V= \V_1\oplus \dots \oplus \V_p$ (as vector bundles over $\G$).
This defines a $\VB$-groupoid $\V \toto E$ over $\G \toto M$,
whose Lie algebroid $A\V \to E$ splits naturally as a Whitney sum $A\V_1\oplus \dots \oplus A\V_p$ over $A\G$.


\begin{proposition}\label{prop:comp_group}
For a source 1-connected Lie groupoid $\G \toto M$, a multiplicative function $F: \V \Arrow \R$ is componentwise linear if and only if so is $AF:A\V\to \R$. Moreover, in the case $\V_1=\dots=\V_p$, $F$ is skew-symmetric if and only if so is $AF$.
%
%
\end{proposition}

\begin{proof}
We will treat the case $p=2$, the general case follows similarly. Consider
$$
F\circ h_{\lambda}, \;\lambda^2 F: \V \to \R \text{ and } F\circ 0_{\V_i}: \V_i \to \R, \,\,\, i=1,2.
$$
 These are multiplicative functions and their associated infinitesimal cocycles are: $AF \circ h_\lambda^A, \,\lambda^2 AF : A\V \to \R$ and $AF \circ 0_{A\V_i} : A\V_i \to \R$, where $0_{A\V_i}: A\V_i \Arrow A\V$ is the zero map \eqref{dfn:zero_map} for the $A\V= A\V_1 \oplus A\V_2$, $i=1,2$. Note that $0_{\V_i}$ is a Lie groupoid morphism and $A(0_{\V_i}) = 0_{A\V_i}$. The fact that $\G$ is source 1-connected implies that $\V$ is source 1-connected (see Remark \ref{rem:fibers}) and, therefore, by uniqueness of integration, one has that
\begin{align*}
 F\circ h_\lambda = \lambda^2 F & \Leftrightarrow AF \circ  h_\lambda^A = \lambda^2 AF\\
 F \circ 0_{\V_i} = 0 & \Leftrightarrow AF \circ 0_{A\V_i} = 0.
\end{align*}
So, it follows from Proposition \ref{comp:alt} that $F$ is componentwise linear if and only if so is $AF$.

When $\V_1=\dots =\V_p$, each permutation $\sigma \in S(p)$ acts by groupoid
morphism on $\V$ via $\sigma(v_1, \dots, v_p) = (v_{\sigma(1)}, \dots, v_{\sigma(p)})$. Applying the Lie functor,
$A\sigma$ acts on $A\V= A\V_1\oplus \dots \oplus A\V_1$ permuting the elements of $A\V_1$ by $\sigma$ itself.
Hence, for a function $F \in C^{\infty}(\V)$,
$$
A(\mathfrak{p}_{\V}(F))= \mathfrak{p}_{A\V}(AF)
$$
where $\mathfrak{p}_{\V}$, $\mathfrak{p}_{A\V}$ are the projections \eqref{anti_projection} for $\V$ and $A\V$, respectively.
The result now follows exactly as before using the uniqueness of integration.
\end{proof}

\section{Cotangent Lie algebroid.}\label{cot_appendix}
Let $\G \toto M$ be a Lie groupoid, with Lie algebroid $\pi_A: A \to M$. Consider the cotangent
Lie groupoid $T^*\G \toto A^*$.
It is well-known that the canonical symplectic form $\omega_{can}$ on $T^*\G$ is multiplicative and its infinitesimal
component $l_{can}: A(T^*\G) \Arrow T^*(A^*)$, which satisfies (see \eqref{l:def})
\begin{equation}\label{eq:canonical}
\omega_{can}(\overrightarrow{\chi}, \cdot) = \cott^*l_{can}(\chi), \;\;\; \forall \, \chi \in \Gamma(A(T^*\G)),
\end{equation}
is a Lie algebroid isomorphism, where the Lie algebroid structure on $T^*(A^*)$ comes from the linear Poisson
structure $\Pi_{lin}$ on $A^*$. In fact, $(T^*\G, \omega_{can})$ is the
symplectic groupoid integrating the linear Poisson structure of $A^*$, see \cite{CDW}.

The composition of $l_{can}$ with the reversal isomorphism $\Rev: T^*(A^*) \Arrow T^*A$ (see \eqref{eq:reversal})
defines the isomorphism of Lie algebroids between $A(T^*G)$ and the cotangent Lie algebroid $T^*A$
mentioned in Section~\ref{lie_alg},
$$
\vartheta:= \Rev \circ l_{can}: A(T^*\G) \longrightarrow T^*A
$$
(see \cite[Theorem~5.3]{Mac-Xu}).
So, by $\overrightarrow{\Rev_a}$ and $\overrightarrow{\calb\alpha}$ we mean the right-invariant vector fields
of $T^*\G \toto A^*$ corresponding to $\vartheta^{-1}(\Rev_a)$ and $\vartheta^{-1}(\calb\alpha)$, for
$a \in \Gamma(A)$ and $\alpha \in \Omega^1(M)$, respectively. Note that, from \eqref{Ra_def} and \eqref{R_hat},
one has that
$$
\vartheta^{-1}(\Rev_a)=l_{can}^{-1}(d\ell_a), \;\; \mbox{ and } \;\;
\vartheta^{-1}(\calb\alpha)= l_{can}^{-1}(-\pi_{A^*}^*\alpha),
$$
where $\pi_{A^*}: A^*\to M$ is the projection of the dual bundle.

\begin{lemma}
For $a \in \Gamma(A)$, consider $\ell_{\overrightarrow{a}} \in C^{\infty}(T^*\G)$. One has
\begin{equation}\label{eq:omegacan}
\omega_{can}(\overrightarrow{\Rev_a}, \cdot) = d\ell_{\overrightarrow{a}} = \omega_{can}(\overrightarrow{a}^{T^*}, \cdot),
\end{equation}
where $\overrightarrow{a}^{T^*}$ is the cotangent lift of $\overrightarrow{a}$ \eqref{ham_lift}.
In particular, $\overrightarrow{\Rev_a} = \overrightarrow{a}^{T^*}$.
\end{lemma}

\begin{proof}
It is simple to check that the pull-back of the 1-form $d\ell_a \in \Omega^1(A^*)$
by the target map $\cott: T^*\G \to A^*$ is exactly $d\ell_{\overrightarrow{a}}$, i.e.,
$
d\ell_{\overrightarrow{a}} = \cott^*d\ell_a.
$
Then, by \eqref{eq:canonical},
$$
\omega_{can}(\overrightarrow{\Rev_a}, \cdot) = \cott^*l_{can}(l_{can}^{-1}(d\ell_{a})) =
d\ell_{\overrightarrow{a}}.
$$
The second identity in \eqref{eq:omegacan} follows from the fact that the cotangent lift of a vector field is exactly its Hamiltonian lift with respect to $\omega_{can}$. The last statement follows from the non-degeneracy of $\omega_{can}$
\end{proof}

In the following, we shall need the useful relationship between the sum and multiplication
on the cotangent bundle known as {\em interchange law}\footnote{Interchange laws hold more generally for
VB-groupoids, where they express the
compatibility between the vector bundle and groupoid structures. See \cite{Mac-book}}:
\begin{equation}\label{inter_law}
(\xi_1 + \eta_1)\bullet (\xi_2 + \eta_2) = (\xi_1\bullet \xi_2) +
(\eta_1\bullet \eta_2),
\end{equation}
for $\xi_1, \eta_1 \in T_{g_1}^*\G$, $\xi_2, \eta_2 \in T_{g_2}^*\G$
such that $(\xi_1, \xi_2)$, $(\eta_1, \eta_2)$ are composable
pairs.

\begin{lemma}
Given a 1-form $\alpha \in \Omega^1(M)$,
$$
\overrightarrow{\calb\alpha} = (\t^*\alpha)^{\vl},
$$
the vertical lift of $\t^*\alpha \in \Omega^1(\G)$.
\end{lemma}

\begin{proof}
To simplify notation, denote $\t^*\alpha$ by $\eta$. First note that $\cots(\eta(g))=0$, $\forall \, g \in \G$.
Indeed, for any $a \in A_{\s(g)}$, it follows from the definition of $\cots$ that
$$
\<\cots(\eta(g)), a \>
= \<\alpha(\t(g)), T\t(\overleftarrow{a}(g))\,\>
= 0.
$$
Hence
$$
T\cots(\eta^{\vl}(\xi))
= \left.\frac{d}{d\epsilon}\right|_{\epsilon=0} \cots(\xi + \epsilon \eta(g))
= \left.\frac{d}{d\epsilon}\right|_{\epsilon=0} (\cots(\xi) + \epsilon \,\cots(\eta(g)))=0.
$$
Also, one has that
\begin{equation}\label{pull_back}
\eta(g) = \eta(\t(g)) \bullet 0_g.
\end{equation}
Indeed, for any $U \in T_g \G$, using \eqref{m:cotang}, one obtains
\begin{align*}
\< \eta(t(g)) \bullet 0_g, U\> = \<\eta(t(g)) \bullet 0_g, T\t(U)\bullet U\> =  \<\eta(t(g)), T\t(U)\> = \<\eta(g), U\>.
\end{align*}

Let us now prove that $\eta^{\vl}$ is a right-invariant vector field. For $\xi \in T_g^*\G$,
\begin{equation*}
\begin{split}
\eta^{\vl}(\xi) & = \left.\frac{d}{d\epsilon}\right|_{\epsilon=0} (\xi + \epsilon \,\eta(g))
\stackrel{\eqref{pull_back}} = \left.\frac{d}{d\epsilon}\right|_{\epsilon=0}(\cotu_{\cott(\xi)} \bullet \xi + \epsilon\, \eta(\t(g)) \bullet 0_g)\\
& \stackrel{\eqref{inter_law}} = \left.\frac{d}{d\epsilon}\right|_{\epsilon=0} [(\cotu_{\cott(\xi)} + \epsilon\, \eta(\t(g)))\bullet (\xi + 0_g)]
 = dR_{\xi}[ \eta^{\vl}(\cotu_{\cott(\xi)}) ].
\end{split}
\end{equation*}
This proves that $\eta^{\vl}$ is right-invariant.

To conclude the proof, we just need to prove that $
\omega_{can}(\eta^\vl(\varphi), \Upsilon) = \omega_{can}(\overrightarrow{\calb\alpha}(\varphi), \Upsilon)
$
for any $\Upsilon \in T_{\varphi} A^* \subset T(T^*\G)$, $\varphi \in A_x^*$, $x \in M$. Choose any
projectable vector field $\widetilde{\Upsilon} \in \mathfrak{X}(T^*\G)$ extending $\Upsilon$, with respect to
the cotangent bundle projection $\pr: T^*\G \to \G$. Recall that $\pr$ is a groupoid morphism covering
$\pi_{A^*}: A^* \to M$. As $\omega_{can}= - d\theta_{can}$ for the tautological 1-form $\theta_{can} \in \Omega^1(T^*\G)$,
one has that
\begin{equation*}
\begin{split}
\omega_{can}(\eta^{\vl}(\varphi), \Upsilon) & = -\Lie_{\eta^{\vl}} \theta_{can}(\widetilde{\Upsilon})|_{\varphi}
 = - \left.\frac{d}{d\epsilon}\right|_{\epsilon=0}\<\varphi+ \epsilon \,\eta(x), \,
Tpr(\widetilde{\Upsilon}(\varphi+ \epsilon \,\eta(x)) \>\\
& = - \< \eta(x), T \pr(\Upsilon)\>,
\end{split}
\end{equation*}
where we have used that $\theta_{can}(\eta^{\vl}) = \theta_{can}([\eta^{\vl}, \widetilde{\Upsilon}])=0$ and
$
T\pr(\widetilde{\Upsilon}(\varphi+ \epsilon \eta(x))) = T\pr(\widetilde{\Upsilon}(\varphi)). $
The proof now follows from \eqref{eq:canonical}
$$
\omega_{can}(\overrightarrow{\calb \alpha}, \cdot) = - \cott^*(\pi_{A^*}^*\alpha) = - \pr^*(\t^*\alpha).
$$
\end{proof}



\end{document}